\documentclass[leqno,11pt]{article}

\usepackage{amssymb,amsmath,amscd,amsfonts,a4,psfrag,graphicx,latexsym,epsfig,color,amsthm}
\usepackage{tikz}

\usetikzlibrary{positioning,matrix}
\usepackage{stackrel}

\usepackage{verbatim}

\newtheorem{theorem}{Theorem}[section]
\newtheorem{lemma}[theorem]{Lemma}
\newtheorem{lemdefi}{Lemma-Definition}
\newtheorem{proposition}[theorem]{Proposition}
\newtheorem{corollary}[theorem]{Corollary}
\newtheorem{definition}[theorem]{Definition}
\newtheorem{remark}[theorem]{Remark}
\newtheorem{example}[theorem]{Example}

\newcommand{\ggot}{\ensuremath{\mathfrak{g}}}
\newcommand{\hgot}{\ensuremath{\mathfrak{h}}}
\newcommand{\kgot}{\ensuremath{\mathfrak{k}}}

\newcommand{\qgot}{\ensuremath{\mathfrak{q}}}

\newcommand{\tgot}{\ensuremath{\mathfrak{t}}}

\newcommand{\Rgot}{\ensuremath{\mathfrak{R}}}

\newcommand{\Xgot}{\ensuremath{\mathfrak{X}}}

\newcommand{\Bcal}{\ensuremath{\mathcal{B}}}
\newcommand{\Ccal}{\ensuremath{\mathcal{C}}}

\newcommand{\Kcal}{\ensuremath{\mathcal{K}}}
\newcommand{\Lcal}{\ensuremath{\mathcal{L}}}

\newcommand{\Ocal}{\ensuremath{\mathcal{O}}}
\newcommand{\Pcal}{\ensuremath{\mathcal{P}}}

\newcommand{\Scal}{\ensuremath{\mathcal{S}}}

\newcommand{\Xcal}{\ensuremath{\mathcal{X}}}

\newcommand{\Zcal}{\ensuremath{\mathcal{Z}}}
\newcommand{\Ucal}{\ensuremath{\mathcal{U}}}
\newcommand{\Vcal}{\ensuremath{\mathcal{V}}}

\newcommand{\Z}{\ensuremath{\mathbb{Z}}}
\newcommand{\C}{\ensuremath{\mathbb{C}}}

\newcommand{\Q}{\ensuremath{\mathbb{Q}}}
\newcommand{\R}{\ensuremath{\mathbb{R}}}
\newcommand{\N}{\ensuremath{\mathbb{N}}}
\newcommand{\Pbb}{\ensuremath{\mathbb{P}}}

\newcommand\quot{{/\hspace{-0.5ex}/}}

\newcommand\inv{{^{-1}}}

\newcommand\conv{{\operatorname{Conv}}}
\newcommand\cone{{\operatorname{Cone}}}
\newcommand\Wt{{\operatorname{Wt}}}
\newcommand\Proj{{\operatorname{Proj}}}
\newcommand\Spec{{\operatorname{Spec}}}

\newcommand\Hom{{\operatorname{Hom}}}

\newcommand{\T}{\ensuremath{\hbox{\bf T}}}

\newcommand{\iu}{{i\mkern1mu}}

\begin{document}

\title{HKKN-stratifications in a non-compact framework}

\author{Paul-Emile Paradan\thanks{IMAG, Univ. Montpellier, CNRS, email : paul-emile.paradan@umontpellier.fr} \and 
Nicolas Ressayre\thanks{ICJ, Univ. Lyon 1, CNRS, email : ressayre@math.univ-lyon1.fr}}

\date{October 2025}

\maketitle

\begin{abstract} 
The aim of this paper is twofold. First, we study HKKN stratifications, both algebraically and analytically, 
for a Cartesian product between a vector space and a compact K\"ahler manifold. 
We then use these stratifications to prove convexity properties of the moment map 
for non-compact analytic subsets invariant under a Borel subgroup.
\end{abstract}

\tableofcontents

%%%%%%%%%%%%%%%%%%%%%%%%%%%%%%%%%%%%%%%%%%%%%%%%%%%%
%%%%%%%%%%%%%%%%%%%%%%%%%%%%%%%%%%%%%%%%%%%%%%%%%%%%
\section{Introduction}
%%%%%%%%%%%%%%%%%%%%%%%%%%%%%%%%%%%%%%%%%%%%%%%%%%%%
%%%%%%%%%%%%%%%%%%%%%%%%%%%%%%%%%%%%%%%%%%%%%%%%%%%%

When we want to describe the Kirwan polytope $\Delta(M)$ associated with the action of a reductive Lie group $G$ on a projective variety $M$, we are led to consider the Cartesian 
product $M\times G/B$ with the flag manifold $G/B$. Each regular dominant weight $\mu$ determines an ample line bundle $\Lcal_\mu$ which polarizes the $G$-action on 
$M\times G/B$, and the shifting trick gives that $\mu\notin \Delta(M)$ if and only if the open subset $(M\times G/B)^{\Lcal_\mu-ss}$  of semistable points relatively to $\Lcal_\mu$ 
is empty. In this case, the Hesselink-Kempf-Kirwan-Ness stratification of the polarized  variety $(M\times G/B,\Lcal_\mu)$ shows that a dense open subset of $M\times G/B$ is 
diffeomorphic to a dense open subset of a variety of the form $G\times_{P(\beta)}Y_\beta$, where $P(\beta)$ is a parabolic subgroup and $Y_\beta$
is a Bia\l ynicki-Birula variety, both attached to  a $1$-parameter subgroup $\beta$.

The origin of this article comes from the desire to make the same kind of argument work for an aﬃne $G$-variety $X$. For this we need an HKKN stratification for a 
Cartesian product between an affine and a projective variety, or more generally for a projective over affine variety. In this article, we study such HKKN stratifications 
both algebraically and analytically. 

Let us explain the content of the various sections of the article.

In Section~\ref{sec:HKKN-algebraic}, we work out the algebraic point of view by considering an HKKN stratification of the 
$G$-variety\footnote{Here $E$ and $V$ are two $G$-modules.} $V\times \Pbb E$. To this end, we introduced the following {\em relative} notions:
\begin{itemize}
\item $x:=(v,[w])\in V\times \Pbb E$ is said relatively semistable when $\overline{G(v,w)}\cap (V\times \{0\})=\emptyset$, and relatively unstable otherwise.
\item For any $1$-parameter subgroup $\tau:\C^*\to G$, we introduced a relative numerical invariant $\varpi_{rel}(-,\tau):V\times \Pbb E\to \R\cup\{-\infty\}$.
\item The relative Mumford invariant $\mathbf{M}_{rel}: V\times \Pbb E\to \R\cup\{-\infty\}$ is defined as $\mathbf{M}_{rel}(x)=\sup_{\tau\neq 0}\frac{\varpi_{rel}(x,\tau)}{\|\tau\|}$.
\end{itemize}

A direct application of Hilbert-Mumford Lemma  shows that $x\in V\times \Pbb E$ is relatively semistable if and only if $\mathbf{M}_{rel}(x)\leq 0$. 
When $x\in V\times \Pbb E$ is relatively unstable, we show the existence of a rational optimal destabilizing $1$-parameter subgroup $\tau_x$ such that 
$$
\mathbf{M}_{rel}(x)= \frac{\varpi_{rel}(x,\tau_x)}{\|\tau_x\|}=\|\tau_x\|.
$$
Here, the parabolic subgroup $P(\tau_x)$ is uniquely defined, as is the $G$-orbit $\langle \tau_x\rangle$ of the $1$-PS $\tau_x$. 
The set $\{\langle \tau_x\rangle, x \ {\rm relatively\ unstable}\}$ is finite, equal to $\{\langle \tau^1\rangle, \cdots, \langle \tau^n\rangle\}$.
The HKKN stratification 
\begin{equation}\label{eq:stratification-relative-intro}
V\times \Pbb E = (V\times \Pbb E)^{ss}\bigcup \Scal_{\langle \tau^1\rangle}\bigcup\cdots\bigcup \Scal_{\langle \tau^n\rangle}
\end{equation}
is defined by the relations $(V\times \Pbb E)^{ss}=\{x \ {\rm relatively\ semistable}\}$ and \break $\Scal_{\langle \tau^i\rangle}=\{x \ {\rm relatively\ unstable},\ \langle \tau_x\rangle=\langle \tau^i\rangle\}$. 
Finally, we show that each stratum $\Scal_{\langle \tau^i\rangle}$ is a smooth subvariety that admits the decomposition
$$
\Scal_{\langle \tau^i\rangle}\simeq G\times_{P(\tau^i)}\Scal_{\tau^i}
$$
where $\Scal_{\tau^i}=\{x \in\Scal_{\langle \tau^i\rangle} ,P(\tau_x)=P(\tau^i)\}$ is a dense open subset of a Bia\l ynicki-Birula variety attached to $\tau^i$.

The stratification (\ref{eq:stratification-relative-intro}) descends to any $G$-subvariety $X\subset V\times \Pbb E$: we have 
\begin{equation}\label{eq:stratification-relative-intro-X}
X = X^{ss}\bigcup X_{\langle \tau^1\rangle}\bigcup\cdots\bigcup X_{\langle \tau^n\rangle}
\end{equation}
where $X^{ss}=\{x\in X \ {\rm relatively\ semistable}\}$, and $X_{\langle \tau^i\rangle}:=X\cap \Scal_{\langle \tau^i\rangle}$.

Stratification (\ref{eq:stratification-relative-intro-X}) is already considered in the works of D. Halpern-Leistner \cite{Hal15}, but its existence 
is proven in a different way to ours. When $\dim E=1$, our stratification recovers the one given in the twisted affine setting by V. Hoskins \cite{Hos14} 
and M. Harada and G. Wilkin \cite{HW11}. 

\medskip

In Section~\ref{sec:polyhedron}, we explain how to associate a moment polyhedron $\Ccal(X)$ with any $G$-invariant closed subvariety $X\subset V\times \Pbb E$. 
As M. Brion did in the projective case \cite{Br87}, we show that for any relatively unstable $x\in V\times\Pbb E$, Mumford numerical invariant $\mathbf{M}_{rel}(x)>0$ 
is equal to the distance between $0$ and $\Ccal(\overline{Gx})$, and moreover, the projection of $0$ onto $\Ccal(\overline{Gx})$ defines an optimal 
destabilizing  $1$-parameter subgroup for some $x'\in G.x$. 

\medskip

In Sections~\ref{sec:semistability-kahler} and \ref{sec:KN-stratification}, we develop the analytic point of view by considering 
HKKN stratifications of non-compact K\"ahler Hamiltonian $G$-manifolds.  Many of the results and ideas presented in these sections arise 
from numerous studies of quotients on K\"ahler manifolds, including those of V. Guillemin, S. Sternberg, L. Ness, F. Kirwan, 
R. Sjamaar, P. Heinzner, I. Mundet i R., A. Teleman, L. Bruasse, C. Woodward, X. Chen, S. Sun and alii. In particular, we will adapt some key proofs 
drawn from the manuscript \cite{GRS21}.  In these sections, we show, in a self-contained exposition, that classical properties of 
HKKN stratifications in the compact setting are still valid when working with manifolds of the form $V\times M$, where $V$ is a 
$G$-module and $M$ is a compact K\"ahler Hamiltonian $G$-manifold. 

Let's recall some classic features of the K\"ahler framework. Let $K$ be a maximal compact subgroup of $G$, and let $\Phi$ the moment map associated to the Hamiltonian action of $K$ on $V\times M$. An element $x\in V\times M$ is called $\Phi$-semistable if $\overline{Gx}\cap \Phi^{-1}(0)\neq\emptyset$, and $\Phi$-unstable otherwise. Heinzner-Huckleberry-Loose have shown that the set $(V\times M)^{\Phi-ss}$ formed by the $\Phi$-semistable elements is an open subset  which permits to define a ``good'' quotient of $V\times M$ \cite{HHL94,HH96}. In this context, we can characterize $\Phi$-semistability as in the algebraic framework. The moment map defines a numerical invariant map 
$\varpi_{\Phi}:V\times M\times \kgot\to \R\cup\{-\infty\}$, and one defines a 
Mumford invariant $\mathbf{M}_{\Phi}: V\times M\to \R\cup\{-\infty\}$ by the relation $\mathbf{M}_{\Phi}(x)=\sup_{\lambda\neq 0}\frac{\varpi_{\Phi}(x,\lambda)}{\|\lambda\|}$. A result of A. Teleman \cite{Teleman04} on ``energy complete'' Hamiltonian manifolds applies here: $x$ is $\Phi$-semistable if and only if $\mathbf{M}_{\Phi}(x)\leq 0$.

The novelty here is that we can characterize $\Phi$-semistability by other means, even if the moment map $\Phi$ is not necessarily proper. 

First, we associate a Kempf-Ness function $\Psi_x:G\to\R$ to any $x\in V\times M$ \cite{Mun00}. Notice that, unlike the compact setting, the $\Psi_x$ functions are generally not globally Lipschitz. 
We prove in Section~\ref{sec:kempf-ness-functions} that, as in the compact framework,  $x$ is $\Phi$-semistable if and only if $\Psi_x$ is bounded below. 
This fundamental result is used in Section~\ref{sec:convexity} to prove certain convexity properties.

Second, we show that the negative gradient flow line $t\mapsto x(t)$ of 
$\frac{1}{2}\|\Phi\|^2$ through $x\in V\times M$ is well-defined for any $t\geq 0$, and moreover the limit $x_\infty:=\lim_{t\to\infty}x(t)$ always exists. 
Next, we show, as in the compact setting, that $x$ is $\Phi$-semistable if and only if $\Phi(x_\infty)=0$.

One of the main tool of our article concerns the existence and uniqueness of an optimal destabilizing vector $\lambda_x\in\kgot$ for each 
$x\in V\times M$ that is $\Phi$-unstable (see Theorems~\ref{theo:maximal-destabilizing-1} and \ref{theo:maximal-destabilizing-2}). This vector satisfies two 
important relations\footnote{For (ii), we use an identification $\kgot^*\simeq\kgot$ given by an invariant scalar product on $\kgot^*$.} :
\begin{enumerate}
\item $\mathbf{M}_{\Phi}(x)=\|\lambda_x\|=\frac{\varpi_{\Phi}(x,\lambda_x)}{\|\lambda_x\|}$.
\item $\Phi(x_\infty)$ belongs to the orbit $K\lambda_x$.
\end{enumerate}

The existence and uniqueness of an optimal destabilizing vector satisfying (i) is a consequence of a result by Bruasse-Teleman \cite{BT05}. 
In Section~\ref{sec:numerical invariant-phi}, we propose a self-contained proof of (i) and (ii) that follows from a result essentially due to Chen-Sun \cite{ChenSun}, and our proof of uniqueness closely follows that given by Kempf \cite{Kem78} in the compact algebraic setting.

\medskip

We can now define the HKKN stratification of $V\times M$ as in the compact framework. Let $\mathrm{Crit}(f)$ be the set of critical points of the function $f=\frac{1}{2}\|\Phi\|^2$. We first verify that $\Phi(\mathrm{Crit}(f))=\bigcup_{\beta\in\Bcal} K\beta$ is a {\em finite} collection of coadjoint orbits. For any $\beta\in\Bcal$, we then define the strata $\Scal_{\langle\beta\rangle}:=\{x\in V\times M,\  \Phi(x_\infty)\in K\beta\}$. 
Here, $\Scal_{\langle0\rangle}$ corresponds to the open set  
$(V\times M)^{\Phi-ss}$, so that 
$$
V\times M=(V\times M)^{\Phi-ss} \bigcup \bigcup_{\beta\in\Bcal-\{0\}}\Scal_{\langle\beta\rangle}.
$$
When $\beta\neq 0$, we define the substrata
$$
\Scal_\beta:=\{x\in V\times M,\  \lambda_x=\beta\}\subset \Scal_{\langle\beta\rangle}.
$$
In Section~\ref{sec:decomposition-S-beta}, we prove that $\Scal_\beta$ is a $P(\beta)$-invariant open subset of a 
Bia\l ynicki-Birula complex submanifold $Y_\beta$, and that $\Scal_{\langle\beta\rangle}$ is a 
locally closed complex submanifold isomorphic to $G\times_{P(\beta)}\Scal_\beta$.

\medskip

The last section is dedicated to the proof of some convexity properties of the moment map. This subject has been studied by numerous authors, 
but we are contributing new results/ideas that can be summarized as follows: we use Kempf-Ness functions and HKKN stratifications 
to prove convexity results for moment polytopes in the K\"ahler framework.

Let $B\subset G$ be the Borel subgroup associated with a Weyl chamber $\tgot^*_{0,+}$. For any subset $X\subset V\times M$, we define $\Delta(X):=\Phi(X) \cap \tgot^*_{0,+}$.

Our main result is the following 

\bigskip

{\bf Theorem A.} \ {\em  Let $X$ be a $B$-invariant irreducible analytic subset of $V\times M$. Then $\Delta(X)$ is closed and convex.}

\bigskip

Theorem A generalizes a result of Guillemin and Sjamaar \cite{GSj06} in two ways: the analytic subset $X$ is not necessarily compact and the K\"ahler structure on $M$ is not assumed to be integral.

\bigskip

Next, we specialize Theorem A for projective over affine varieties.

\bigskip

{\bf Theorem B.} \ {\em  Let $X$ be a $G$-invariant irreducible algebraic variety of $V\times \Pbb E$. There is a set $\{\lambda_1,\ldots,\lambda_\ell\}$ of dominant weights and a set $\{\xi_1,\ldots,\xi_p\}$ of rational dominant elements such that 
$$
\Delta(X)=\mathrm{convex \ hull}(\{\xi_1,\ldots,\xi_p\})+\sum_{j=1}^\ell\ \R^{\geq 0} \lambda_j.
$$
}

\bigskip

Theorem B is both a generalization of Brion's result \cite{Br87}
obtained for $G$-invariant projective algebraic varieties and of that
obtained by Sjamaar \cite{Sjamaar98} for $G$-invariant affine
varieties. These results assert that $\Delta(X)$ is a rational
polytope when $X$ is projective, it is a rational convex polyhedral
cone in the affine case and it is a rational polyhedron in general. 

\bigskip

There is another way to describe the moment polytope of $\overline{Bx}\subset V\times \Pbb E$. Let $\Phi_{T_0}: V\times \Pbb E\to\tgot_0^*$ 
be the moment map for the maximal torus $T_0\subset K$. Let $U:=[B,B]$ be the unipotent radical of the Borel subgroup.

\bigskip

{\bf Theorem C.} \ {\em  For any $x\in V\times \Pbb E$, we have
$$
\Delta(\overline{Bx})=\tgot_{+,0}^*\cap \bigcap_{u\in U}\Phi_{T_0} (\overline{Tux}).
$$
}

Theorem C generalizes a result obtained by Guillemin and Sjamaar \cite{GSj05} in the projective setting.

%%%%%%%%%%%%%%%%%%%%%%%%%%%%%%%%%%%%%%%%%%%%%%%%%%%%
\subsection*{Notation} \label{sec:notations}
%%%%%%%%%%%%%%%%%%%%%%%%%%%%%%%%%%%%%%%%%%%%%%%%%%%%

Throughout the paper, $K$ will denote a compact, connected Lie group, and $\kgot$ its Lie algebra. Let $G$ be its complexification with Lie algebra $\ggot=\kgot\oplus i\kgot$. 
Let $\exp:\ggot\to G$ be the exponential map. The elements of $\ggot_{ell}:=G(\kgot)=\{g\xi, \xi\in\kgot\}\subset\ggot$ are called elliptic.

Let $T_0\subset K$ and $T\subset G$ be maximal tori such that $T$ is the complexification of $T_0$. Let us denote by $\Xgot_{*}(T)$ the set of 
$1$-parameter subgroups\footnote{$1$-PS for short.} of $T$. The map 
\begin{equation}\label{eq:1-PS}
\gamma\mapsto \gamma_0:=\frac{d}{dt}\vert_{t=0} \gamma(e^{it})
\end{equation} 
defines a group isomorphism between $\Xgot_{*}(T)$ and the lattice $\Lambda:=\frac{1}{2\pi}\ker(\exp\vert_{\tgot_0})\subset\tgot_0$, so that 
$\gamma(e^{z})=\exp(-iz\gamma_0),\forall z\in \C$. The rational $1$-parameter subgroups of $T$ are those belonging to the 
$\Q$-vector space $\Xgot_*(T)_\Q:=\Xgot^*(T)\otimes_{\Z}\Q$:  $\tau \in \Xgot_*(T)_\Q$ if there exists $n\geq 1$ such that $n\tau\in \Xgot_*(T)$.
The map (\ref{eq:1-PS}) defines an isomorphism between $\Xgot_*(T)_\Q$ and  the $\Q$-vector space $\tgot_{0,\Q}\subset \tgot_{0}$ generated by $\Lambda$.

Let $\Xgot_{*}(G)$ be the set of $1$-PS of $G$. Since all $1$-PS of $G$ are conjugated to an element of $\Xgot_{*}(T)$, the map (\ref{eq:1-PS}) 
defines a bijective map between $\Xgot_{*}(G)$ and the subset $G(\Lambda)=\{g\lambda, \lambda\in\Lambda\}\subset \ggot_{ell}$. The rational $1$-parameter subgroups of $G$ 
belongs to the set $\Xgot_*(G)_\Q$ defined by the relation:  $\tau \in \Xgot_*(G)_\Q$ if there exists $n\geq 1$ such that $n\tau\in \Xgot_*(G)$. 
Here (\ref{eq:1-PS}) defines a bijective map between $\Xgot_*(G)_\Q$ and the subset $\ggot_{ell,\Q}:=G(\tgot_{0,\Q})$ of {\em rational elliptic} element of $\ggot$. So, an element $\mu\in\ggot$ is rational elliptic if there exists $n\geq 1$ such that  $e^{z}\in \C^*\mapsto \exp( -iz n\mu)$ defines a $1$-PS of $G$.

The characters of $T$ form an abelian group under pointwise multiplication, denoted $\Xgot^*(T)$. The map $\chi\mapsto \chi_0$, defined by the relation 
$\chi(\exp(\xi))=e^{i\langle\chi_0,\xi\rangle},\forall \xi\in\tgot_0^*$, realizes a canonical isomorphism between $\Xgot^*(T)$ and the lattice 
$\Lambda^*= \hom(\Lambda, \Z)\subset\tgot_0^*$. The $\Q$-vector space $\Xgot^*(T)_\Q:=\Xgot^*(T)\otimes_{\Z}\Q$ is canonically isomorphic with 
 the $\Q$-vector space $\tgot^*_{0,\Q}\subset \tgot^*_{0}$ generated by $\Lambda^*$.

Let us now choose a $K$-invariant inner product on $\kgot$, that takes integral values on the lattice $\Lambda$. It defines a $K$-invariant isomorphism of 
real vector spaces $\kgot^*\simeq\kgot$, and also an isomorphism $\tgot^*_{0,\Q}\simeq \tgot_{0,\Q}$ of $\Q$-vector spaces, invariant under the Weyl group action.

The quadratic map $\delta:\ggot\to\R$ defined by the relation $\delta(\xi_1\oplus i\xi_2):= \| \xi_1\|^2-\|\xi_2\|^2$ is $G$-invariant. Hence, for any elliptic element 
$\mu\in\ggot_{ell}$, the quantity $\delta(\mu)$ is positive if $\mu\neq 0$. Thus, we can define the norm of $\gamma\in \Xgot_{*}(G)$ by setting
$\|\gamma\|:=\sqrt{\delta(\gamma_0)}$, and we have $\|g \gamma g^{-1}\|=\|\gamma\|$ for any $g\in G$.

Let $\tgot^*_{0,+}$ be a Weyl chamber of $\tgot^*_0$. Let $\Lambda^*_+:=\Lambda^*\cap \tgot^*_{0,+}$ be the set of dominant weights of $G$; it is a 
submonoid of the group $\Lambda^*$ of characters of $T$. For each $\chi\in \Lambda^*_+,$ 
we choose a simple $G$-module $V_\chi$ of highest weight $\chi$.

%%%%%%%%%%%%%%%%%%%%%%%%%%%%%%%%%%%%%%%%%%%%%%%%%%%%
%%%%%%%%%%%%%%%%%%%%%%%%%%%%%%%%%%%%%%%%%%%%%%%%%%%%
\section{HKKN-stratifications in the relative algebraic framework} \label{sec:HKKN-algebraic}
%%%%%%%%%%%%%%%%%%%%%%%%%%%%%%%%%%%%%%%%%%%%%%%%%%%%
%%%%%%%%%%%%%%%%%%%%%%%%%%%%%%%%%%%%%%%%%%%%%%%%%%%%

%%%%%%%%%%%%%%%%%%%%%%%%%%%%%%%%%%%%%%%%
\subsection{GIT-quotient}
\label{sec:defss}
%%%%%%%%%%%%%%%%%%%%%%%%%%%%%%%%%%%%%%%%

Let $E$ et $V$ be two $G$-modules and $X$ be a closed $G$-stable
subvariety of $V\times\Pbb E$.
We adapt the notion of semistability for the action of $G$ on $X$
thinking about $X$ as a relative projective variety.
The situation is closed to that studied in \cite{GHH15,HHZ23}
where semistability for projective varieties over affine ones is
defined. But here, we emphasize on the induced stratification and the
moment polyhedron.
\\

\noindent{\bf Notation.}
The elements of $V\times\Pbb E$ are denoted by $x$. 
For such a $x$, we write $x=(v,m)$, implying that $v\in V$ and $m\in \Pbb E$. 
For $m\in\Pbb E$, $\tilde m$ denote a chosen nonzero point in the line $m$ in $E$.
\\

We associate to the subvariety $X\subset V\times\Pbb E$ the affine variety $\tilde X:=\{(v,\tilde{m}),\ (v,m)\in X\}$ of 
$V\times E$ which is a cone in the $E$-variable.
Consider, on the ring $\C[\tilde X]$ of regular functions, the
graduation induced by the $\C^*$-action on $E$ by homotheties.
For any $h\in \N$, $\C[E]_h$ denotes the space of homogeneous polynomials
of degree $h$.
Then $
\C[V\times E]=\oplus_{h\in\N}\C[V]\otimes \C[E]_h
$
is graded, and $\C[\tilde X]_h$ is the image of $\C[V]\otimes \C[E]_h$
under the restriction map to $\tilde X$. In other words, $X$ is
characterized by a $G$-stable graded ideal $I(X):=\oplus_{k\in\N}
I_k(X)$, where $I_k(X)\subset \C[V]\otimes \C[E]_k$ and  the quotient
$R(X):=\C[\tilde X]=\C[V\times E]/I(X)$ is graded: $R(X):=\oplus_{k\in\N} R_k(X)$ where 
$R_k(X)=\C[V]\otimes \C[E]_k/I_k(X)$.\\

\bigskip
Consider the affine variety $X_0=\Spec(R_0(X))$ that is equal to the image of $X$ under
the projection $V\times\Pbb E\longrightarrow  V$.

\begin{lemdefi}{}\label{def:us}
  Let $x=(v,m)\in X$.
  The following are equivalent:
  \begin{enumerate}
  \item there exist $h>0$ and a $G$-invariant function $f\in
    R_h(X)$ such that $f(v,\tilde m)\neq 0$;
    \item the
  orbit-closure $\overline{G.(v,\tilde m)}$ does not intersect $V\times \{0\}$.
  \end{enumerate}
  Then, the point $x$ is said to be   {\it (relatively) semistable}.
  Otherwise, $x$ is said to be {\it (relatively) unstable}.
The term relatively is omitted hereafter.  
For example,  the set of semistable and unstable points are
  respectively  denoted by ${X}^{\rm ss}$ and ${X}^{\rm
    us}$.
The subset ${X}^{\rm ss}$ is open and $G$-stable. 
\end{lemdefi}

\begin{proof}
  The first assertion implies the second one, since any
  $f\in R_h(X)$, with $h>0$, vanishes along
  $(V\times\{0\})\cap\tilde X$. 

Conversely, $\overline{G.(v,\tilde m)}$ and $V\times\{0\}$ are disjoint closed
$G$-stable subsets of $V\times E$. Thus, there exists $f\in\C[V\times E]^G$
such that $f_{|V\times\{0\}}=0$ and $f(v,\tilde m)\neq 0$. Write
$f=\sum_{h\geq 0}f_h$ with $f_h \in\C[V]\otimes\C[E]_h$.
The assumption $f_{|V\times\{0\}}=0$ means that $f_0=0$. 
Thus, there exists a positive $h$ such that $f_h(v,\tilde m)\neq 0$. 
But $(f_h)_{|V\times\{0\}}=0$. 
  
  The last assertion on $X^{\rm ss}$ is obvious.
\end{proof}

Define the quotient $X^{\rm ss}\quot G$ to be
$$
X^{\rm ss}\quot G:=\Proj(\oplus_{h\in\N}R_h(X)^G)
$$
endowed with the quotient map
$$
\pi\,:\, X^{\rm ss}\longrightarrow  X^{\rm ss}\quot G. 
$$
The inclusion $R_0(X)^G$ in $R(X)^G$
induces a commutative diagram
\begin{center}
\begin{tikzpicture}[>=latex, node distance=1.4cm]

% Définition des nœuds
\node (Xss) {$X^{\rm ss}$};
\node (XssG) [right=3cm of Xss] {$X^{\rm ss} \quot G$};
\node (X) [below of=Xss] {$X$};
\node (X0) [below of=X] {$X_0$};
\node (X0G) [right=3cm of X0] {$X_0 \quot G$};

% Flèches
\draw[->] (Xss) -- node[above] {$\pi$} (XssG);
\draw[->] (Xss) -- (X);
\draw[->] (X) -- (X0);
\draw[->] (X0) -- node[above] {$\pi_0$} (X0G);
\draw[->] (XssG) -- node[right] {$p$} (X0G);

\end{tikzpicture}
\end{center}
where $p$ is projective.

\begin{lemma}
  \label{lem:fibrequot}
  The quotient map $\pi$ is affine and for any affine open subset $U\subset
  X^{\rm ss}\quot G$, the restriction of $\pi$ on $\pi\inv(U)$ is
  the affine quotient $\pi\inv(U)\longrightarrow  \pi\inv(U)\quot G$. 
In particular, each fiber of $\pi$ contains a unique closed $G$-orbit in $X^{\rm ss}$.
\end{lemma}

\begin{proof}
Let $f\in R_h(X)^G$ be nonzero, for some positive $h$.
Consider $D(f)\subset   X^{\rm ss}\quot G$ the associated
principal open subset. 
Then, $D(f)$ is affine with $R(X)^G_{(f)}=\{\frac
g{f^k}\,:\,g\in R_{kh}(X)^G$ with $k>0\}$ as ring of regular
function.

Similarly, thinking about $f$ as an element of  $R(X)_{(f)}$,
we get that $\pi^\inv(D(f))$ identifies with  the principal open subset $\tilde D(f)$ in $\tilde X$. 
Its ring of regular functions is 
$R(X)_{(f)}$. 
Hence, the lemma is true with $U=D(f)$. 
Since the subsets $D(f)$ form a covering the first assertion follows.

The "In particular" assertion being true when $X$ is affine, it is a consequence of the first part. 
\end{proof}

\noindent{\bf Basic Examples.}
\begin{enumerate}
  \item When $V=\{0\}$, $X$ is projective and endowed with the very ample $G$-linearized line bundle $\Ocal(1)$. 
  Then $X^{\rm ss}\quot G=X^{\rm ss}(\Ocal(1))\quot G$. 
  \item If $E=\C$ endowed with the trivial $G$-action, $X$ is affine, $X^{\rm ss}=X$ and 
  $X^{\rm ss}\quot G=X\quot G$ is the affine quotient.
  \item Let $\theta$ be a character of $G$ and $E$ be the associated 1-dimensional representation of $G$. 
  Here, $X$ is affine and $X^{\rm ss}\quot G$ is the quotient denoted by $X\quot_\theta G$ in {\it e.g.} \cite{King94}.
\end{enumerate}

Let
$\tau\,:\,\C^*\longrightarrow  G\in \Xgot_{*}(G)$ and $x\in X$.
We write $\lim_{t\to \infty}\tau(t)x=x_0$, if the map $f\,:\,\C^*\longrightarrow  X, t\longmapsto\tau(t\inv )x$ extends to a
regular function $\bar f\,:\,\C\longrightarrow  X$ with $\bar f(0)=x_0$. 

\begin{theorem}{\bf Hilbert-Mumford}
  \label{th:HM}
  Let $x=(v,m)\in X$.
 The point $x$ is unstable if and only if there exists
 $\tau\in \Xgot_{*}(G)$ such that
 \begin{enumerate}
 \item $\lim_{t\to \infty}\tau(t)v$ exists in $V$; and
 \item $\lim_{t\to \infty}\tau(t)\tilde m=0$. \end{enumerate}
\end{theorem}

\begin{proof}
  If there exists $\tau\in \Xgot_{*}(G)$ as in the theorem, $(\lim_{t\to
    0}\tau(t) x,0)\in \overline{\tau(\C^*)(v,\tilde m)}\subset
  \overline{G(v,\tilde m)}$.
  Hence, $x$ is unstable.

  Conversely, assume that $x$ is unstable.
 By definition, there exists $x'\in X$ such that $(x',0)\in
 \overline{G(v,\tilde m)}$.
 Let $x^0$ be such that $G.(x^0,0)$ is the only closed $G$-orbit in
 $\overline{G(x',0)}$.
 Then $G.(x^0,0)$ is also the only closed $G$-orbit in $G.(v,\tilde m)$.
 The classical Hilbert-Mumford theorem (see {\it e.g.} \cite[Theorem~6.9]{PV94}) implies
 there exists $\tau$ as in the theorem.
\end{proof}

\subsection{Instability degree}

Fix $\tau\in \Xgot_{*}(G)$.
Then $E$ decomposes as $E=\oplus_{n\in\Z}E_n$ where $E_n=\{\tilde m\in
E\,:\,\forall t\in\C^*\quad \tau(t)\tilde m=t^n\tilde m\}$.
Let $\tilde m=\sum_n \tilde m_n$ be a vector written according this decomposition.
Assuming $\tilde m\neq 0$ and setting $m=[\tilde m]\in\Pbb E$, define
\begin{equation}
  \label{eq:defm}
  \varpi_E(m,\tau):=-\max\{n\in\Z\,:\,\tilde m_n\neq 0\}.
\end{equation}
Obviously, $\lim_{t\to \infty}\tau(t)\tilde m=0$ if and only if $\varpi_E(m,\tau)>0$.

We have a well-defined $Ad(G)$-invariant map $\Vert\ \Vert\,:\,\Xgot_{*}(G)-\{0\}\longrightarrow \  \R_{>0}$ (see Section {\bf Notation}).
Set 
$$
\bar \varpi_E(m,\tau)=\frac{\varpi_E(m,\tau)}{\Vert\tau\Vert},\quad \forall \tau\in \Xgot_{*}(G)-\{0\}.
$$
For  $x=(v,m)\in V\times \Pbb E$ and $\tau\in \Xgot_{*}(G)$, we define:
\begin{equation} \label{eq:defvarpibar}
\bar \varpi_{rel}(x,\tau)=
  \begin{cases}
     -\infty&\mbox{if }\lim_{t\to \infty}\tau(t)v\mbox{ does not exist in }V,\\
    \bar \varpi_E(m,\tau)& \mbox{otherwise.}
  \end{cases}
\end{equation}
Hence, $\lim_{t\to \infty}\tau(t)x$ exists in $V\times \Pbb E$ when $ \bar{\varpi}_{rel}(x,\tau)\neq-\infty$.

Set 
\begin{equation}
  \label{eq:defM}
  \mathbf{M}_{rel}(x):=\sup_{
      \tau\in \Xgot_{*}(G)-\{0\}
  }\bar \varpi_{rel}(x,\tau).
\end{equation}

\begin{remark}\label{rem:G-stable}
Notice that $\mathbf{M}_{rel}(x)=-\infty$ if and only if $v\in V$ is $G$-stable, i.e. the orbit $Gv$ is closed and the stabilizer subgroup 
$G_v$ is finite.
\end{remark}

The following result is a direct consequence of Hilbert-Mumford Theorem~\ref{th:HM}.

\begin{proposition}\label{prop:unstable-M-rel}
The point
$x=(v,m)\in V\times \Pbb E$ is unstable if and only if $\mathbf{M}_{rel}(x)>0$.
\end{proposition}

When considering the action of the maximal torus $T\subset G$, we define
\begin{equation}
  \label{eq:defM-T}
  \mathbf{M}^T_{rel}(x):=\sup_{
      \tau\in \Xgot_{*}(T)-\{0\}
  }\bar \varpi_{rel}(x,\tau).
\end{equation}

\subsection{Optimal destabilizing $1$-parameter subgroups}

\subsubsection{The case of a torus action}

In this section, we consider  the action of a maximal torus $T\subset G$ on $V\times \Pbb E$. Under the $T$-action, $E$ decomposes in weight spaces:
\begin{equation}
  \label{eq:1}
  E=\bigoplus_{\chi\in \Xgot^*(T)}E_\chi,
\end{equation}
where $E_\chi=\{\tilde m\in E\,:\,\forall t\in T\quad t\tilde m=\chi(t)\tilde m\}$.
If $\tilde m=\sum \tilde m_\chi\in V$ is written according this decomposition, we set
$\Wt_T(\tilde m)=\{\chi\,:\,\tilde m_\chi\neq 0\}$.
For $v\in V$, one defines $\Wt_T(v)$ similarly.

Recall that $\Xgot^*(T)\otimes_\Z\R\simeq \tgot_0^*$ is equipped with a scalar product $(\ ,\ )$ that is induced by the invariant scalar product on $\kgot^*$. 
If $S\subset \tgot_0^*$ is a finite set, we denote by $\conv(S)$ the convex hull of $S$. We also consider $\cone(S):=\sum_{s\in S}\R^+s$. 

Let  $C\subset \tgot_0^*$ be a non-empty closed convex set that does not contain $0$. It is well known that the distance between $0$ and $C$ is characterized by 
\begin{equation}
  \label{eq:distsup}
  d(0,C)=\sup_{\eta\in \tgot_0^*-\{0\}}\,
\inf_{\xi\in C}\frac{(\eta,\xi)}{\Vert \eta\Vert}.
\end{equation}

Furthermore, the supremum on $\eta$ is attained exactly when $\eta$ is positively proportional to the orthogonal projection of $0$ onto $C$.

\begin{definition}
To any $x=(v,m)\in V\times\Pbb E$, we associate the closed convex polyhedron 
     \begin{equation}
       \label{eq:defPxy}
       \Pcal_T(x):=\conv(\Wt_T(\tilde m))+\cone(\Wt_T(v))\subset
       \tgot_0^*.
     \end{equation} 
\end{definition}

Let $\langle\ ,\ \rangle$ denote the duality bracket between $\Xgot^*(T)\simeq \Lambda^*$
and $\Xgot_{*}(T)\simeq \Lambda$.

     \begin{proposition}
       \label{prop:ssG=T}
       For $x\in V\times\Pbb E$, we have:
       \begin{enumerate}
       \item $x$ is $T$-semistable if and only if $0\in
         \Pcal_T(x)$.
       \item If $x$ is  $T$-unstable then : 
       \begin{enumerate}
       \item $\mathbf{M}_{rel}^T(x)$ is the distance between $0$ and $\Pcal_T(x)$.
       \item There exists a unique rational $1$-parameter subgroup
           $\tau_x\in \Xgot_{*}(T)_\Q\simeq \tgot_{0,\Q}$ such that $\mathbf{M}_{rel}(x)=\bar \varpi_{rel}(x,\tau_x)=\Vert\tau_x\Vert$.
           Moreover, using the identification $\tgot^*_0\simeq\tgot_0$ given by the scalar product, $\tau_x$ is identified as the
           opposite of the orthogonal projection of $0$ onto $\Pcal_T(x)$.
         \item The set $\{\chi\in \Wt_T(\tilde m)\cup\Wt_T(v)\;|\;
           \langle\tau_x,\chi\rangle=-\Vert\tau_x\Vert\}$ is non-empty
           and contained in $\Wt_T(\tilde m)$.
       \end{enumerate}
         \end{enumerate}
         The rational $1$-parameter subgroup $\tau_x$ is said to be 
         {\em optimal destabilizing} for $x$. 
     \end{proposition}

\begin{proof}
Let $x=(v,m)$ and $\tau\in \Xgot_{*}(T)$. Observe first that $\lim_{t\to
\infty}\tau(t)v$ exists if and only if $\langle \chi, \tau\rangle\leq 0,\forall \chi\in \Wt_T(v)$, and $\lim_{t\to
\infty}\tau(t)\tilde m=0$ if and only if $\langle \chi, \tau\rangle< 0,\forall \chi\in \Wt_T(\tilde m)$. Now,
$$
       \begin{array}{ccll}
         0\not\in\Pcal_T(x)&\iff&\exists\tau\in \Xgot_{*}(T),\ \forall
                                \eta\in\Pcal_T(x) \quad
                                \langle\eta,\tau\rangle>0.  \quad \mbox{
                                                            (Hahn-Banach's
                                                         theorem)}\\
         
                         &\iff&\exists\tau\in \Xgot_{*}(T), \quad\left\{
                                \begin{array}{l}
                                \forall
                                \chi\in\Wt_T(v) \quad
                                  \langle\chi,\tau\rangle\geq 0,\  \text{and}\\
                                  \forall \chi\in\Wt_T(\tilde m) \quad
                                  \langle\chi,\tau\rangle> 0.
                                \end{array}
         \right .\\
         
&\iff&\exists\tau\in \Xgot_{*}(T),\quad\left\{
  \begin{array}{l}
    \lim_{t\to \infty}(-\tau)(t)v\ \mbox{exists}, \ \text{and}\\
       \lim_{t\to \infty}(-\tau)(t)\tilde m=0.
      \end{array}
      \right .\\
         &\iff&x\mbox{ is $T$-unstable.}
       \end{array}
 $$

The assertions $(a)$ and $(b)$ are a direct consequence of the
reminder regarding orthogonal projection that precedes the proposition. 

Since $\Pcal_T(x)$ is a rational polyhedron, and the scalar product takes rational values on $\tgot_{0,\Q}^*$, the orthogonal projection of
$0$ on $\Pcal_T(x)$ is a rational element. Thus, $\tau_x$ is rational. 

The last assertion is an easy fact that we leave to the reader.
     \end{proof}

\subsubsection{The general case}\label{subsec:the-general-case}

Let 
$$
P(\tau):=\{g\in G\,:\,\lim_{t\to \infty}\tau(t)g\tau(t\inv)\mbox{ exists in }G\}
$$
be the parabolic subgroup associated to $\tau\in \Xgot_{*}(G)$, and let 
$$
U(\tau):=\{g\in G\,:\,\lim_{t\to \infty}\tau(t)g\tau(t\inv)=e\}.
$$
be the unipotent radical of $P(\tau)$. We have a Levi decomposition $P(\tau)=U(\tau)G^\tau$ where $G^\tau$ denotes the centralizer of the image of $\tau$ in $G$.
Moreover, for $u\in U(\tau)$ and $l\in G^\tau$, we have $\lim_{t\to
  \infty}\tau(t)(ul)\tau(t\inv)=l$.

\begin{lemma}
  \label{lem:minv}
  Let $x\in V\times\Pbb E$ and $\tau\in \Xgot_{*}(G)$.
  \begin{enumerate}
  \item For any $g\in G$,  
    $\bar\varpi_{rel}(x,\tau)=\bar\varpi_{rel}(g x,g\tau g\inv)$.
    \item For any $p\in P(\tau)$, 
    $\bar\varpi_{rel}(x,\tau)=\varpi_{rel}(x,p\tau p\inv)$.
  \end{enumerate}
\end{lemma}

\begin{proof}
The first assertion is obvious. The second one is a consequence of the
fact that, for any point $z$, any $u\in U(\tau)$ and $l\in G^\tau$ and any $t\in\C^*$, we have
$$
p\tau(t) p\inv z=(p\tau(t) p\inv\tau(t\inv))\tau(t) z
$$
and
$$
\lim_{t\to \infty} p\tau(t) p\inv\tau(t\inv)=u,
$$
where $p=ul$.
\end{proof}

Several objects or properties that are associated to $\tau\in \Xgot_{*}(G)$ are left unchanged by multiplying $\tau$ by a positive
integer. Examples of such things are the assertions ``$\lim_{t\to \infty}\tau(t)x$ exists'',
 ``$\tau$ fixes $x$''. Then, we can freely use these properties for rational $1$-parameter subgroups $\tau\in\Xgot_{*}(G)_{\Q}$.

\begin{theorem}{}
  \label{th:Kempf}

  Let $x\in  V\times\Pbb E$ be unstable.
  \begin{enumerate}
  \item There exists nontrivial $\tau \in \Xgot_{*}(G)_{\Q}$ such that 
    $ \mathbf{M}_{rel}(x)=\bar \varpi_{rel}(x,\tau)=\Vert \tau\Vert$.
    Such a $\tau$ is said to be an \emph{optimal destabilizing} $1$-parameter subgroup for $x$. 
    \item Let $\tau_1$, $\tau_2\in \Xgot_{*}(G)_\Q$ be optimal destabilizing $1$-parameter subgroups for $x$. Then
    \begin{enumerate}
    \item $P(\tau_1)=P(\tau_2)$;
    \item $\exists p\in P(\tau_1),\quad \tau_2=p\tau_1p\inv$.
    \end{enumerate}
    We can therefore define the parabolic subgroup $P(x)$ as $P(\tau_1)$.
  \end{enumerate}
    \end{theorem}

    \begin{proof}   
 Since any $1$-parameter subgroup of $G$ is $G$-conjugated to a $1$-parameter subgroup of $T$, we have
 $$
\mathbf{M}_{rel}(x)=\sup_{g\in G}\bigg (\sup_{
    \tau\in \Xgot_{*}(T)-\{0\}
  }\bar \varpi_{rel}(g x,\tau)\bigg )
  =\sup_{g\in G}\mathbf{M}_{rel}^T(g x).
  $$

 Fix $g\in G$ such that $\mathbf{M}_{rel}^T(g x)>0$, i.e. $gx$ is $T$-unstable. By Proposition~\ref{prop:ssG=T}, there exists a 
 unique $\tau_{gx}\in \Xgot_{*}(T)_\Q$ such that
$$
\mathbf{M}_{rel}^T(g x)=\bar \varpi_{rel}(g x,\tau_{gx})=\Vert \tau_{gx}\Vert.
$$

Using the fact that there are only finitely  many possibilities for the set
$\Wt_T(g\tilde m)\times\Wt_T(gv)$ when $g$ varies, we see that $\{\tau_{gx},\, gx\ \mbox{is}\ T\mbox{-unstable}\}$ is finite. The first claim follows.\\

Let $\tau_1,\tau_2\in \Xgot_{*}(G)_\Q$ such that
$$
\bar \varpi_{rel}(x,\tau_1)=\bar \varpi_{rel}(x,\tau_2)=\mathbf{M}_{rel}(x)=\Vert \tau_1\Vert=\Vert \tau_2\Vert.
$$
The intersection $P(\tau_1)\cap P(\tau_2)$ contains a maximal torus
$T'$ of $G$ (see {\it e.g.} \cite[Corollary~28.3]{redbook}). Let $p_1\in P(\tau_1)$ and $p_2\in
P(\tau_2)$ such that $\tau_1':=p_1\tau_1p_1\inv$ and
$\tau_2':=p_2\tau_2p_2\inv$ belong to $\Xgot_{*}(T')$.

By Lemma~\ref{lem:minv}, 
$\varpi_E(x,\tau_1)=\varpi_E(x,\tau_1')$.
And similarly for $\tau_2$ and $\tau_2'$. Hence,
$$
\mathbf{M}_{rel}(x)=\mathbf{M}_{rel}^{T'}(x)=\bar \varpi_{rel}(x,\tau_1')=\bar \varpi_{rel}(x,\tau_2').
$$
The unicity of the optimal destabilizing $1$-PS for $T'$ shows that $\tau_1'=\tau_2'$. It follows that
$P(\tau_1)=p_1P(\tau_1)p_1\inv=P(\tau_1')=P(\tau_2')=p_2P(\tau_2)p_2\inv=P(\tau_2)$.

Finally, the identity $\tau_2= p\tau_1p\inv$, with $p=p_2\inv p_1\in P(\tau_1)$ proves the last
assertion. 
    \end{proof}

Given an unstable element $x\in  V\times\Pbb E$, we denote by $\Lambda(x)\subset  \Xgot_{*}(G)_\Q$ the set of optimal destabilizing $1$-parameter subgroups for $x$.

\begin{lemma}\label{lem:unicity-tau-x}There exists a unique $\tau_x\in \kgot$ such that $\Lambda(x)$ is equal to the orbit
$\{p\tau_xp\inv,\ p\in P(x)\}$.
\end{lemma}
\begin{proof}Let $\tau \in \Xgot_{*}(G)_{\Q}$ that is optimal destabilizing for $x$. Let $g\in G$ such that 
$g\tau g^{-1}\in \Xgot_{*}(T)_{\Q}$. Let $(k,p)\in K\times P(\tau)$ such that $g=kp$. Then $\tau'=p\tau p^{-1}$ 
is optimal destabilizing for $x$, and $\tau'\in Ad(k^{-1})(\Xgot_{*}(T)_{\Q})\subset \kgot$. The unicity is classical (see Theorem D.4 in \cite[Appendix~D]{GRS21}).
\end{proof}

\subsection{Optimal $1$-parameter subgroup and flow}

We now prove three results, due to L. Ness \cite{Ness84} (see also \cite{RR84}) 
in the classical setting on the set of unstable points $x$ such that a given $\tau$ belongs to $\Lambda(x)$.
    
    \begin{theorem}(See \cite[Theorem~8.3]{Ness84})
      \label{th:Ness}
      Let $x\in V\times\Pbb E$ be unstable,
      and $\tau\in \Xgot_{*}(G)_{\Q}$ be an optimal destabilizing $1$-parameter subgroups for $x$. 
      Let\footnote{The limit $\lim_{t\to \infty}\tau(t)x$ exists in $V\times \Pbb E$, because $ \varpi_{rel}(x,\tau)\neq-\infty$} $x'=\lim_{t\to\infty}\tau(t)x$. 
      Then
      \begin{enumerate}
      \item $x'$ is unstable; \label{th:Nessi}
      \item $\mathbf{M}_{rel}(x')=\mathbf{M}_{rel}(x)$;\label{th:Nessii}
       \item $\tau$ is an optimal destabilizing $1$-parameter subgroup for $x'$. \label{th:Nessiii}
      \end{enumerate}
    \end{theorem}

\begin{proof}
Write $x=(v,m)$ and $x'=(v',m')$.
Since $\bar \varpi_E(m',\tau)=\bar \varpi_E(m,\tau)>0$ and  $v'$ is $\tau$-fixed,
the point $(v',m')$ is relatively unstable. Clearly, Assertion~\ref{th:Nessii} implies 
Assertion~\ref{th:Nessiii}.

\bigskip
We first prove Assertion~\ref{th:Nessii} when $G=T$ is a torus.
Observe that
$\Wt_T(\tilde m')=\{\chi\in\Wt_T(\tilde m)\,|\,\langle\tau,\chi\rangle=\Vert\tau\Vert^2\}$.
Hence, the face of $\Pcal_T(x)$ containing the projection of $0$ in its
interior is equal to $\Pcal_T(x')$. It implies that $\mathbf{M}_{rel}^T(x')=\mathbf{M}_{rel}^T(x)$.

\bigskip
We now assume that $G$ is reductive and connected. 
Let $\tau'\in \Xgot_{*}(G)_\Q$ be an optimal destabilizing $1$-PS for $x'$. 
Let  $T'\subset P(\tau)\cap P(\tau')$ be a maximal torus $G$. 
Let $p\in P(\tau)$ and $p'\in
P(\tau')$ such that $\tau_1:=p\tau p\inv$ and
$\tau_1':=p'\tau'p'\inv$ belong to $\Xgot_{*}(T')$.

Then $\tau_1$ is optimal destabilizing for $x_1:=p x$.
Moreover, $\lim_{t\to \infty}\tau_1(t)x_1=x_1'$. 
By the case $G=T'$, we deduce that $\tau_1$ is optimal destabilizing for $x_1$. 

On the other hand, by Lemma~\ref{lem:minv}, $\tau_1'$ is optimal destabilizing for $x$ and $G$. 
Hence, it is for $x$ and $T'$.
In particular, we have
$$
\bar \varpi_{rel}(x,\tau)=\bar \varpi_{rel}(x,\tau_1)=\bar \varpi_{rel}(x,\tau_1')=\bar \varpi_{rel}(x,\tau').
$$ 
\end{proof}

We now prove a converse of Theorem~\ref{th:Ness}.

\begin{theorem}(See \cite[Theorem~9.3]{Ness84})
  \label{th:Nessconv}
Let $x\in V\times\Pbb E$ and $\tau\in \Xgot_{*}(G)_\Q$ such that $\varpi_{rel}(x,\tau)$ is positive. 
Set $x'=\lim_{t\to\infty}\tau(t)x$. 
Assume that $x'$ is unstable and that $\tau$ is an optimal destabilizing $1$-parameter subgroup for $x'$. 
Then 
      \begin{enumerate}
      \item $x$ is unstable; \label{th:Nessci}
      \item $\mathbf{M}_{rel}(x)=\mathbf{M}_{rel}(x')$;\label{th:Nesscii}
        \item $\tau$ is an optimal destabilizing $1$-parameter subgroup for $x$. \label{th:Nessciii}
      \end{enumerate}
\end{theorem}

\begin{proof}
The first assertion is obvious since $\varpi_{rel}(x,\tau)>0$.
The second assertion is a direct consequence of the last one. 
Let us prove the last one. 

Let $\tau_1$ be a rational $1$-parameter subgroup of $G$ optimal destabilizing for $x$. 
Let $p\in U(\tau)$ and $p_1\in P(\tau_1)$ such that $\tau'=u\tau u\inv$
and $\tau_1'=p_1\tau_1 p_1\inv$ are $1$-parameter subgroup of some
maximal torus $T'$ of $G$. 

By Lemma~\ref{lem:minv}, $\tau_1'$ is optimal for $x$.
Moreover, $\lim_{t\to \infty}\tau'(t)x=ux'=:x''$ and $\varpi_{rel}(x'',\tau')=\varpi_{rel}(x,\tau)$. 
In particular, $\tau'$ is optimal destabilizing for $x''$. 

Observe that $\Wt_{T'}(y'')\subset \Wt_{T'}(y)$. Then $\bar
\varpi_E(m'',\tau_1')\geq \bar \varpi_E(m,\tau_1')$. Finally, we have
$$
 \bar \varpi_E(m,\tau_1')\leq  \bar \varpi_E(m'',\tau_1') = \bar \varpi_E(m'',\tau')= \bar \varpi_E(u y,\tau')
$$
and $\tau'$ is optimal for $uy$. So,  $\tau$ is optimal for $x$. 
\end{proof}

\subsection{Algebraic Shifting Trick}

\subsubsection{\`A la Ness}

Fix $\tau\in \Xgot_{*}(G)_\Q$.
Choose $g\in G$ such that  $g\tau g\inv\in \Xgot_{*}(T)$.
Let $T''$ be the subtorus of $T$ such that $\Xgot_{*}(T'')$ spans the
orthogonal of $g\tau g\inv$ in $\Xgot_{*}(T)\otimes \R\simeq \tgot_{0}$.
Let $(G^\tau)'$ be the subgroup of $G^\tau$ generated by the derived group
$[G^\tau,G^\tau]$ and $g\inv T''g$.
Note that $(G^\tau)'$ does not depend on $g$.
Moreover, the map $(G^\tau)'\times \C^*\longrightarrow  G^\tau,
(h,t)\longmapsto h\tau(t)$ is an isogeny. 

\begin{theorem}{}(See~\cite[Theorem~9.4]{Ness84})\label{th:optss}
  Let $x\in V\times\Pbb E$. Let $\tau\in \Xgot_{*}(G)_\Q$ be 
  such that $\bar \varpi_{rel}(x,\tau)=\Vert\tau\Vert>0$.
  Assume that $\tau(\C^*)$ fixes $x$.

  Then, $\tau$ is an optimal destabilizing $1$-parameter subgroup for $x$ if and only if $x$ is
  semistable for the action of $(G^\tau)'$.
\end{theorem}

\begin{proof} 
Write $x=(v,m)$.
Assume that $\tau$ is an optimal destabilizing $1$-PS for $x$.
  Let $\tau_1$ be a $1$-parameter subgroup of $(G^\tau)'$ such that
$\lim_{t\to \infty}\tau_1(t)x$ exists.

  Since $\tau_1$ and $\tau$ are orthogonal, we have, for any
  $\varepsilon\in\R$,  $\Vert \tau+\varepsilon \tau_1\Vert^2=\Vert
  \tau\Vert^2+\varepsilon^2 \Vert\tau_1\Vert^2$. Moreover, since $m$
  is fixed by $\tau$, we have $\varpi_E(m,\tau+\varepsilon\tau_1)=
  \varpi_E(m,\tau)+\varepsilon \varpi_E(m,\tau_1)$.
  A direct computation shows that
  \begin{equation}
\bar \varpi_E(m,\tau+\varepsilon\tau_1)=\bar \varpi_E(m,\tau)+\varepsilon
\frac{\varpi_E(m,\tau_1)}{\Vert\tau\Vert}+o(\varepsilon).\label{eq:DL}
\end{equation}

But, for any rational $\varepsilon>0$, $\lim_{t\to \infty}(\tau+\varepsilon \tau_1)(t)v$ exists.
As $\bar \varpi_E(m,\tau+\varepsilon\tau_1)\leq\bar \varpi_E(m,\tau)$, the identity 
\eqref{eq:DL} shows that $\varpi_E(m,\tau_1)\leq 0$.
Now, Theorem~\ref{th:HM} implies that $x$ is semistable for the
action of $(G^\tau)'$.

\bigskip
Conversely, assume that $x$ is semistable for the action of
$(G^\tau)'$.
 Let $\tau_1$ be a $1$-parameter subgroup of $G$ such that
$\lim_{t\to \infty}\tau_1(t)x$ exists. Let us show that $\bar \varpi_E(m,\tau_1)\leq \bar \varpi_E(m,\tau)$.

Let $u\in U(\tau)$ and $p_1\in P(\tau_1)$ such that $u\tau
u\inv=:\tau'$ and
$p_1\tau_1p_1\inv=:\tau_1'$ are $1$-parameter subgroups of a maximal
torus $T'$ of $G$. 

\bigskip
\noindent\underline{Claim:} \quad $x$ is semistable for the action
$(G^{\tau'})'$.

Consider the map $L\,:\,(x_1,\tilde m_1)=(\lim_{t\to \infty}\tau(t)x_1,\lim_{t\to
  0}\tau(t) t^{\varpi_E(m,\tau)}\tilde m_1)$ defined on a closed subvariety of
$V\times E$ containing $x$ and stable by $G^\tau$. Hence, $L$ is
defined on $\overline{G^\tau(v,\tilde m)}$.
Moreover, since $\lim_{t\to \infty}\tau(t)u\inv\tau(t)\inv=e$, we have 
$
L((G^\tau)' u\inv (v,\tilde m))=(G^\tau)' (v,\tilde m).
$
Now, $x$ being semistable for $(G^\tau)'$, the closure of
$(G^\tau)' u\inv (v,\tilde m)$ does not intersect $V\times\{0\}$.
The claim follows.

\bigskip
The Claim  implies that $x$ is relatively semistable for $T''$, e.g.
$0\in\Pcal_{T'}(x)$. But $\Pcal_{T'}(x)$ is the orthogonal
projection of $\Pcal_T(x)$ in the direction $\tau'$.
It follows that $\tau'$ is optimal of $x$ and the $T'$-action. In
particular
$$
\bar \varpi_E(m,\tau_1)=\bar \varpi_E(m,\tau_1')\leq \bar \varpi_E(m,\tau')=\bar \varpi_E(m,\tau).
$$
\end{proof}

\subsubsection{With a tensor product}

Let us recall two standard operations on embedded varieties.

{\bf Cartesian product}

If $X\subset V\times\Pbb E$ and $X'\subset V'\times\Pbb(E')$, using
the Segre embedding, we embed $X\times X'$ in $V\times V' \times\Pbb(E\otimes E')$ such that 
$$
R_k(X\times X')\simeq R_k(X)\otimes R_k(X').
$$

{\bf Changing the polarization}

Let $\ell\geq 1$. Then $\Pbb E $ embeds in $\Pbb(\mathrm{Sym}^\ell E)$ by the
Veronese embedding, $[\tilde m]\mapsto [\tilde m^\ell]$. We denote by
$X_{[\ell]}$ the image of $X$ by the induced embedding of
$V\times\Pbb E$ in $V\times\Pbb(\mathrm{Sym}^\ell E)$. Then, 
$$
R_k(X_{[\ell]})\simeq R_{k\ell}(X),\quad \forall k\geq 0.
$$

\bigskip
Recall now an easy property of root systems:

\begin{lemma}
  \label{lem:betatau}
  Let $\tau$ be a $1$-parameter subgroup of $G$.

  Choose a maximal torus $T'$ containing the image of $\tau$.  Let
  $\tau^\flat$ be the character of $T'$ dual of $\tau$ for the
  scalar product on $\Xgot^*(T')_\R$. Let $\ell\geq 1$ such that $\ell\tau^\flat\in \Xgot^*(T')$.

Then $\ell\tau^\flat$ is a character of the parabolic subgroup $P(\tau)$ that is independent of the choice of $T'$.
\end{lemma}

Let $T'$ be a maximal torus of $G$ and $\chi$ a character of $T'$.
Then, there exists a unique irreducible representation of $G$ denoted by
$V(\chi)$ having $\chi$ as extremal weight for the $T'$-action. Here, extremal weight
means a weight that is a vertex of $\conv(\Wt_{T'}(V(\chi)))$.
One can construct $V(\chi)$ as follows: choose a Borel subgroup
$B'\supset T'$ such that $\chi$ is dominant relatively to
$B'$. Then, $V(\chi)$ is the irreducible representation with highest
weight $\chi$ for $B'$.

Start now with $\tau\in \Xgot_{*}(G)$. Let $\ell$ and $\tau^\flat$ as in Lemma~\ref{lem:betatau}. 
Thus, the representation $V(\ell\tau^\flat)$ is well-defined.  Let $v_{\ell\tau^\flat}$ be an eigenvector of weight $\ell\tau^\flat$ in $V_{\ell\tau^\flat}$. 

\begin{theorem}
  \label{th:shiftingtrick}
  Let $x=(v,m)\in V\times\Pbb E$ such that $\bar\varpi_{rel}(x,\tau)=\Vert \tau\Vert$. The following statements are equivalent:
  \begin{enumerate}
\item $\tau$ is an optimal destabilizing $1$-parameter subgroup $x$.
\item $(v,[\tilde m^\ell\otimes v_{\ell\tau^\flat}])\in V\times\Pbb(\mathrm{Sym}^\ell E\otimes V(\ell\tau^\flat))$ is semistable.
\end{enumerate}
\end{theorem}

\begin{proof}
First, let us assume that $\tau$ is optimal destabilizing for $x$, and by contradiction, that $(v,[\tilde{m}^\ell\otimes v_{\ell\tau^\flat}])$ is unstable.
Let $\tau_1$ be a $1$-parameter subgroup of $G$ that is optimal destabilizing for 
$(v,[\tilde{m}^\ell\otimes v_{\ell\tau^\flat}])$.

Let $T'$ be a maximal torus of $G$ contained in $P(\tau)\cap P(\tau_1)$. 
Let $p\in P(\tau)$ and $p_1$ in $P(\tau_1)$ such that 
$\tau':=p\tau p\inv$ and $\tau_1':=p_1\tau_1 p_1\inv$ are rational $1$-parameter subgroups of $T'$.

According to Lemma~\ref{lem:minv}, $\tau'$ and $\tau_1'$ are optimal destabilizing for the $T'$-action on $x$ and $(v,[\tilde{m}^\ell\otimes v_{\ell\tau^\flat}])$, respectively.
In particular, $-(\tau')^\flat$ belongs to $\Pcal_{T'}(x)$, by Proposition~\ref{prop:ssG=T}.

On the other hand, $\ell\tau^\flat$ is a character of $T'$ and 
\begin{align*}
\Wt_{T'}([\tilde m^\ell \otimes v_{\ell\tau^\flat}])		&=   \ell\tau^\flat+\Wt_{T'}(\tilde m^\ell)), \\
\conv(\Wt_{T'}(\tilde m^\ell))                                        &=  \ell\,\conv(\Wt_{T'}(\tilde m)). 
\end{align*}

It follows that $\Pcal_{T'}(v,[\tilde m^\ell \otimes v_{\ell\tau^\flat}])=\ell\left(\tau^\flat+\Pcal_{T'}(x) \right)$. 
Using $\tau_1'$, Lemma~\ref{lem:minv} implies that $[\tilde m^\ell\otimes v_{\ell\tau^\flat}]$ is unstable
for $T'$.
Hence,  $0$ does not belong to $\Pcal_{T'}(v, [\tilde m^\ell\otimes
v_{\ell\tau^\flat}])$, that is $-\tau^\flat$ does not belong to 
$\Pcal_{T'}(x)$. 

Since, by Lemma~\ref{lem:betatau}, $\ell\tau^\flat=\ell(\tau')^\flat$ as
characters of $P(\tau)=P(\tau')$ this is a contradiction.

\bigskip
The converse works similarly. Assume that $(v,[\tilde{m}^\ell\otimes v_{\ell\tau^\flat}])
\in V\times\Pbb(\mathrm{Sym}^\ell V\otimes V(\ell\tau^\flat))$ is
semistable.
Let $\tau_1$ be a $1$-parameter subgroup of $G$ that is optimal destabilizing for
$x$.

Let again $T'$ be a maximal torus of $G$ contained in $P(\tau)\cap P(\tau_1)$. 
Let $p\in P(\tau)$ and $p_1$ in $P(\tau_1)$ such that 
$\tau':=p\tau p\inv$ and $\tau_1':=p_1\tau_1 p_1\inv$ are $1$-parameter subgroups of $T'$. 

Since $[\tilde m^\ell\otimes v_{\ell\tau^\flat}]$ is semistable for $T'$,
$-(\tau')^\flat$ belongs to $\Pcal_{T'}(x)$. 
Then, Proposition~\ref{prop:ssG=T} implies that $\tau'$ is optimal destabilizing for
$x$ and the action of $T'$. 
But $\tau_1'$ being a $1$-parameter of $T'$ that is optimal destabilizing for
$x$ and the $G$-action, we deduce that $\tau'$ is optimal destabilizing for
$x$ and $G$, too. 
Lemma~\ref{lem:minv} implies now that $\tau$ is optimal destabilizing for $x$. 
\end{proof}

\subsection{A stratification of the nullcone}
\label{sec:strat}

\subsubsection{A partition of the unstable locus}

In this section, we introduce the HKKN-stratification in the
relative setting of  a closed $G$-stable irreducible 
subvariety $X\subset V\times\Pbb E$.

For any dominant $\tau\in \Xgot_{*}(G)_\Q-\{0\}$, we denote by $\langle\tau\rangle$ the set of
rational $1$-parameter subgroups of $G$ conjugated to $\tau$, and we define 
\begin{equation}
  \label{eq:defSdtau}
  X_{\langle\tau\rangle}:=\left\{x\in X^{\rm us}\,:\, \mathbf{M}_{rel}(x)=\Vert\tau\Vert\mbox{ and } 
  \langle\tau\rangle\cap \Lambda(x)\neq\emptyset\right\}.
\end{equation}
Since there are only finitely many possibilities for the sets
$\Wt_T(v)\times\Wt_T(\tilde m)$ when $x=(v,m)$ varies in $X^{\rm us}$, one easily sees that $X_{\langle\tau\rangle}$
is not empty for finitely many classes $\langle\tau\rangle$. 
Moreover, Theorem~\ref{th:Kempf} shows that  $X^{\rm us}$
is the disjoint union of the sets $X_{\langle\tau\rangle}$.

To describe the geometry of $X_{\langle\tau\rangle}$, define 
\begin{equation}
\label{eq:defSdtau2}
X_{\tau}:=\{x\in X^{\rm us}\,:\, \mathbf{M}_{rel}(x)=\Vert\tau\Vert\mbox{ and } \tau\in \Lambda(x)\},
\end{equation}
and
\begin{equation}
\label{eq:defZdtau}
Z_{\tau}:=\{x\in X_{\tau}\,:\, \tau(\C^*) \mbox{ fixes }x \}.
\end{equation}

Set $V_{\tau\leq 0}=\{v\in V\,:\, \lim_{t\to \infty}\tau(t)v\mbox{ exists in } V\}$. 
Decompose $E=\oplus_{q\in \Q}E_q$ under the action of $\tau$. 
Here, the index set is $\Q$ since $\tau$ is rational. 
Set $E_{\tau\leq -\Vert \tau\Vert^2}=\oplus_{q\leq -\Vert \tau\Vert^2}E_q$ and similarly $E_{\tau < -\Vert \tau\Vert^2}$.

We consider the closed subvarieties
$$
A_\tau:=X\bigcap V_{\tau\leq 0}\times\Pbb(E_{\tau\leq -\Vert \tau\Vert^2})
\quad\mathrm{and}\quad B_\tau:=X\bigcap V_{\tau= 0}\times\Pbb(E_{\tau= -\Vert \tau\Vert^2}).
$$
and the open subset of $A_\tau$ defined by 
$$
\Ucal_\tau:=X\bigcap V_{\tau\leq 0}\times\left(\Pbb(E_{\tau\leq -\Vert \tau\Vert^2})-\Pbb(E_{\tau< -\Vert \tau\Vert^2})\right).
$$
The map $x\longmapsto \lim_{t\to \infty}\tau(t)x$ defines a linear projection $\pi_\tau :\Ucal_\tau\to B_\tau$. %such that $\Ucal_\tau=\pi_\tau^{-1}(B_\tau)$.

By definition, we have $X_\tau\subset \Ucal_\tau$ and $Z_\tau\subset B_\tau$. Thanks to Theorem~\ref{th:optss}, we know that  $Z_\tau$ 
is equal to open subset $B_\tau^{ss}$ of semistable points for the action of $(G^\tau)'$, and by Theorem~\ref{th:Ness}, 
we know that $x\in A_\tau$ belongs to $X_\tau$ if and only if $\pi_\tau(x)\in Z_\tau$. Hence, $X_\tau$ is an open subset of $A_\tau$, 
and thus a locally closed subset of $X$.

\bigskip

\noindent{\bf Fibered product.}
Let $X$ be a $G$-variety and $S\subset X$ be a locally closed subset that is stable under the action of some given parabolic subgroup $P\subset G$. 
Consider the incidence variety
$$
I=\{(gP/P,x)\in G/P\times X\;:\;g\inv x\in S\}
$$
with the two projections $p_{G/P}$ and $p_X$ on $G/P$ and $X$, respectively. 
Observe that $I$ is locally closed as a subset of $G/P\times X$, and hence, is a variety. 
One can immediately check that the fibers of the map
$$
\begin{array}{ccl}
   G\times S&\longrightarrow &I\\   
   (g,s)&\longmapsto&(gP/P,gs)
\end{array}
$$
are the $P$-orbits for the action given by $p.(g,s)=(gp\inv,ps)$. 
When we think about $I$ as a quotient, we denote it by $G\times_P S$. 

Since the map $G\longrightarrow  G/P$ is locally trivial (by Bruhat decomposition), the map $p_{G/P}\,:\,I=G\times_P S\longrightarrow  G/P$
is locally trivial with fiber $S$.

\bigskip
\begin{proposition}\label{prop:Kirwan}
  Let $\tau\in \Xgot_{*}(G)_\Q$ be non-trivial.
  Then:
  \begin{enumerate}
\item \label{ass:K1}$Z_{\tau}$ and $X_{\tau}$ are locally closed in $X$. 
\item \label{ass:K2}$X_{\tau}$ is the set of $x\in X$ such that $\lim_{t\to \infty}\tau(t)x$ exists and belongs to $Z_{\tau}$.
\item \label{ass:K3}$X_{\tau}$ is stable under the action of $P(\tau)$, and $gX_{\tau}\cap X_\tau\neq\emptyset$ if and only if $g\in P(\tau)$.
\item \label{ass:K4} The  map:
$$
\begin{array}{ccl}
  G\times_{P(\tau)}X_{\tau} &\longrightarrow  &X_{\langle\tau\rangle}\\
 {}[g:z]& \longmapsto & gz
\end{array}
$$
is bijective. 
\item \label{ass:K5}The boundary $\overline{X_{\langle\tau\rangle}}-X_{\langle\tau\rangle}$ is contained in the union of $X_{\langle\tau'\rangle}$ with $\Vert\tau'\Vert>\Vert\tau\Vert$.
\item \label{ass:K6}$X_{\langle\tau\rangle}$ is locally closed in $X$.
\end{enumerate}
\end{proposition}

\begin{proof}
Assertions~\ref{ass:K1} and  \ref{ass:K2} have been already explained and Assertion~\ref{ass:K3} is a direct consequence of Lemma~\ref{lem:minv}. Assertion~\ref{ass:K3} 
follows from Lemma~\ref{lem:minv} and Theorem~\ref{th:Kempf}. Assertion~\ref{ass:K4} follows from the previous assertion and the fact that $X_{\langle\tau\rangle}=G.X_\tau$.

Consider a point $x$ in the closed subset $A_\tau - X_{\tau}$ of the subvariety $A_\tau$, and define $x_0=\lim_{t\to\infty}\tau(t)x$. Two cases occur. 

Either $x_0\in B_\tau-Z_\tau$. Then, $x_0$ is unstable for the action of $(G^\tau)'$ on $B_\tau$. Let $\tau_1$ be a $1$-parameter subgroup of $(G^\tau)'$ optimal for $x_0$. 
Then Formula~\eqref{eq:DL} shows that $\bar\varpi(x,\tau+\varepsilon \tau_1)>\bar\varpi(x,\tau)=\Vert \tau\Vert$, which gives $\mathbf{M}_{rel}(x)>\Vert\tau\Vert$. 

Or $x_0\in X\bigcap V_{\tau= 0}\times\Pbb(E_{\tau< -\Vert \tau\Vert^2})$. This is possible only if $x\in V_{\tau\leq 0}\times\Pbb(E_{\tau< -\Vert \tau\Vert^2})$: 
this last inclusion implies that $\varpi_E(x,\tau)>\Vert \tau\Vert^2$ and then $\mathbf{M}_{rel}(x)>\Vert\tau\Vert$.

We have therefore verified that $\mathbf{M}_{rel}(x)>\Vert\tau\Vert,\forall x\in A_\tau - X_{\tau}$.

Since $G/P(\tau)$ is projective, the map $\pi\,:\,G\times_{P(\tau)}A_\tau\longrightarrow  X$ is proper.
Since $X_\tau$ is dense in $A_\tau$ and $X_{\langle\tau\rangle}=G.X_\tau$, the image of $\pi$ is $G.A_\tau=\overline{X_{\langle\tau\rangle}}$. 
Moreover, $\overline{X_{\langle\tau\rangle}}-X_{\langle\tau\rangle}=G.(A_\tau - X_{\tau})$. 
Then, the previous paragraph shows that the points $x$ in $\overline{X_{\langle\tau\rangle}}-X_{\langle\tau\rangle}$ 
satisfy $\mathbf{M}_{rel}(x)>\Vert\tau\Vert$. Assertion~\ref{ass:K5} is proved. 

The properness of $\pi$ shows also that  $\overline{X_{\langle\tau\rangle}}-X_{\langle\tau\rangle}$ is closed. 
Hence, $X_{\langle\tau\rangle}$ is locally closed. 
\end{proof}

\begin{corollary}\label{coro:open-strata} 
Let $X\subset V\times\Pbb E$ be a closed $G$-stable irreducible variety such that $X^{ss}=\emptyset$. 
Then, among the $\tau\in\Xgot_*(T)^+$ such that $X_{\langle\tau\rangle}$ is not empty, a unique one $\tau_0$ has minimal length.
Moreover, $X_{\langle \tau_0\rangle}$ is the unique open strata.
\end{corollary}

\begin{proof}Since $X^{ss}=\emptyset$, we have a finite partition $X=\bigcup_{\tau\neq 0}X_{\langle \tau\rangle}$ into locally closed subvarieties. 
Let $r=\min\{\Vert\tau\Vert,\ X_{\langle \tau\rangle}\neq\emptyset\}>0$. We can write $X=\Ucal\bigcup F$ with 
$\Ucal:= \bigcup_{\tau, \Vert\tau\Vert=r}X_{\langle \tau\rangle}$ and $F:=\bigcup_{\tau, \Vert\tau\Vert > r}X_{\langle \tau\rangle}$. 
Thanks to the Assertion~\ref{ass:K5} of Proposition~\ref{prop:Kirwan}, we know that $F$ is a closed subvariety of $X$, and then $\Ucal$ is a non-empty Zariski open subset. 
Since $X$ is irreducible, there exists only one $\tau_0$ such that  $\Vert\tau_0\Vert=r$, and such that  $\Ucal=X_{\langle \tau_0\rangle}$. 
\end{proof}

%%%%%%%%%%%%%%%%%%%%%%%%%%%%%%%%%%%%%%
\subsection{Moment polyhedron}\label{sec:polyhedron}
%%%%%%%%%%%%%%%%%%%%%%%%%%%%%%%%%%%%%%

%%%%%%%%%%%%%%%%%%%%%%%%%%%%%%%%
\subsubsection{Definition}
%%%%%%%%%%%%%%%%%%%%%%%%%%%%%%%%

We fix a maximal torus $T$ of $G$ and a Borel subgroup $B$ containing $T$. 
Let $\Xgot^*(T)$ be the group of $T$-character and $\Xgot^*(T)^+$ denote the set of dominant characters.
For $\chi\in \Xgot^*(T)^+$, we denote by $V_\chi$ the irreducible representation of highest weight $\chi$.

To a closed $G$-stable {\em irreducible} subvariety of $X\subset V\times\Pbb E$, we associate
$$
\Ccal(X)=\{\lambda\in \Xgot^*(T)_{\Q}\;:\; \exists n>0,\quad n\lambda\in\Xgot^*(T)^+\mbox{ and } V_{n\lambda}\subset R_{n}(X)\}.
$$
\begin{remark}\label{rem:P-X-0}
Definition~\ref{def:us} means that \ $0\in \Ccal(X)\iff X^{ss}\neq\emptyset$. 
\end{remark}

If $X$ is affine (case $E=\C$ with trivial $G$-action), it is well known (see
{\it e.g.} \cite[Proof of Proposition~2.1]{Br87}) that $\Ccal(X)$ is a convex polyhedral cone in
$\Xgot^*(T)_\Q$.
If $C$ is a convex set in a rational affine space, its {\it recession cone}  is defined as the set of vectors $v$ such that for any
$p\in C$, $p+v$ belongs to $C$. This condition is equivalent to requiring that there exists $p\in C$, such that $p+tv\in C,\ \forall t\in \Q^{\geq 0}$.

\begin{lemma}
  \label{lem:CX}
  If $X$ is a $G$-stable irreducible subvariety of $V\times\Pbb E$,
  the set $\Ccal(X)$ is a rational polyhedron with recession cone
  $\Ccal(X_0)$.
  In particular, there is a set $\{\lambda_1,\ldots,\lambda_\ell\}$ of dominant weights and a set $\{\xi_1,\ldots,\xi_p\}$ of rational dominant elements such that 
\begin{equation}\label{eq:P-X-convex}
\Ccal(X)=\mathrm{convex \ hull}_\Q(\{\xi_1,\ldots,\xi_p\})+\sum_{j=1}^\ell\ \Q^{\geq 0} \lambda_j.
\end{equation}
\end{lemma}

\begin{proof}
It is well-know that $S_X=\{(\chi, n)\in \Xgot^*(T)\times \N\;:\; V_\chi\subset  R_n(X)\}$ is the set of weights of
$T\times\C^*$ acting on $R(X)^U=\C[\tilde X]^U$. A theorem of Hadziev and Grosshans (see e.g. \cite[Theorem 2.7]{Br10}) says that the algebra $\C[\tilde X]^U$ is finitely generated.
Since it is also an integral domain, we can conclude that $S_X$ is a finitely generated semigroup of $\Xgot^*(T)\times \N$. 
One deduces easily that $\Ccal(X)$ is a rational polyhedron, since it is the
image of  $S_X\cap \Xgot^*(T)\times \N^*$ under the map $(\chi,n)\mapsto\frac \chi n$.

Since $R(X)^U$ is a module over $\C[X_0]^U$, $\Ccal(X_0)$ is
contained in the recession cone of $\Ccal(X)$. Conversely, $ R(X)$ is a finitely generated $\C[X_0]$-module,
and hence, $ R(X)^U$ is a finitely generated $\C[X_0]^U$-module. 
Let $f_0,\dots,f_s$ in $ R(X)^U$ be a set of generators. 
One may assume that $f_0=1$ and $f_1,\dots,f_s$ are eigenvectors for
$T\times\C^*$ of weight $(\chi_i,n_i)$. Then 
$$
\Ccal(X)={\rm convex\ hull}_\Q(\{\xi_1,\ldots,\xi_s\})+\Ccal(X_0),
$$
with $\xi_i= \frac{\chi_i}{n_i}$. We deduce that the recession cone of $\Ccal(X)$ is exactly $\Ccal(X_0)$. Finally,
since $\Ccal(X_0)$ is a convex polyhedral cone in $\Xgot^*(T)_\Q$, there is a set $\{\lambda_1,\ldots,\lambda_\ell\}$ of dominant weights such that 
$\Ccal(X_0)=\sum_{j=1}^\ell\ \Q^{\geq 0} \lambda_j$.
\end{proof}

\begin{remark}
Conversely, any rational polyhedron in $\Q^n$ is the moment polyhedron of some action of $T=(\C^*)^n$ on a projective over affine variety.
This follows easily from the decomposition's theorem for polyhedron (see \cite[Section~1.C]{BG:poly}).  
\end{remark}

\medskip
Let $\chi$ be a dominant character of $T$. Let $v_\chi$ be an eigenvector of weight $\chi$ in
$V_\chi$. The projective variety $\Ocal_\chi=G.[v_\chi]\subset\Pbb(V_\chi)$ satisfies 
$R(\Ocal_\chi)\simeq \oplus_{n\geq 0} V_{n\chi}^*$.

\begin{lemma}[Shifting Trick]\label{lem:Shifting}
Let $\lambda$ be a rational dominant weight, and $\ell\geq 1$ such that $\ell\lambda\in\Xgot^*(T)^+$. 
The following conditions are equivalent
\begin{enumerate}
\item $\lambda\in\Ccal(X)$.
\item $\exists n\geq 1$, $V_{n\ell\lambda}\subset R_{n\ell}(X)$.
\item $\exists n\geq 1$, $R_n(X_{[\ell]}\times \Ocal_{\ell\lambda})^G\neq 0$.
\item $(X_{[\ell]}\times \Ocal_{\ell\lambda})^{ss}\neq \emptyset$.
\item $0\in\Ccal(X_{[\ell]}\times \Ocal_{\ell\lambda})$.
\end{enumerate}
\end{lemma}

\begin{proof}
The equivalences $(i)\iff (ii)$ and $(iii)\iff (iv)\iff (v)$ follow directly from the definitions.
The equivalence $(ii)\iff (iii)$ is due to the fact that $V_{n\ell\lambda}\subset R_{n\ell}(X)$ if and only if 
$R_n(X_{[\ell]}\times \Ocal_{\ell\lambda})=R_{n\ell}(X)\otimes V_{n\ell\lambda}^*$ contains a nonzero $G$-invariant.
\end{proof}

%%%%%%%%%%%%%%%%%%%%%%%%%%%%%%%%
\subsubsection{Some inequalities}
%%%%%%%%%%%%%%%%%%%%%%%%%%%%%%%%

Let $\langle -,-\rangle$ be the duality bracket between $\tgot^*_0$ and $\tgot_0$.

\begin{lemma}
\label{lem:ineqopenstrata}
Assume that $X^{\rm ss}$ is empty and let $\tau\in \Xgot_{*}(T)_\Q$ be the unique dominant element such that $X_{\langle\tau\rangle}$ 
is the open stratum. Then, for any $\lambda\in\Ccal(X)$, we have
$$
\langle\lambda,\tau\rangle\geq\Vert \tau\Vert^2.
$$
\end{lemma}

\begin{proof}
By definition, there exist a positive $h$ and $f\in (\C[E]\otimes\C[V]_h)^{(B)_{h\lambda}}$ 
such that $f$ is not identically zero on $X$. 
Hence, by assumption, $f$ is not identically zero on the stratum $X_{\langle\tau\rangle}$. 

On the one hand, Proposition~\ref{prop:Kirwan} implies that $G\times_{P(\tau)}X_\tau$ maps surjectively 
on $X_{\langle\tau\rangle}$. On the other hand, $\tau$ being dominant $BP(\tau)$ is open in $G$. 
We deduce that $f$ is not identically zero on $BX_\tau$ and finally on $X_\tau$, by $B$-equivariance. 

Let $x=(v,m)\in X_\tau$ such that $f(v,\tilde m)\neq 0$.  Consider the limit $x':=\lim_{t\to\infty}\tau(t)x=(v',m')$. Since $\varpi_E(m,\tau)=\Vert \tau\Vert^2$, 
$\lim_{t\to\infty}\tau(t)t^{\Vert\tau\Vert^2}\tilde m=\tilde m'$ 
(up to changing the representative $\tilde m'$ if necessary).
Hence,
\begin{equation}
  \label{eq:lim}
  \lim_{t\to\infty}\tau(t)(v,t^{\Vert\tau\Vert^2}\tilde m)=(v',\tilde m').
\end{equation}
Consider now 
$$
\begin{array}{ll}
f(\tau(t)(v,t^{\Vert\tau\Vert^2}\tilde m))&=(\tau(t)\inv f) (v,t^{\Vert\tau\Vert^2}\tilde m) \\
&=t^{\langle -h\lambda,\tau\rangle}f(v,t^{\Vert\tau\Vert^2}\tilde m)\\
&=t^{\langle -h\lambda,\tau\rangle+h\Vert\tau\Vert^2}f(v,\tilde m).
\end{array}
$$
Now, Equation~\eqref{eq:lim} implies that this last quantity tends to $f(v',\tilde m')$, 
when $t$ goes to $\infty$. 
Hence, the exponent of $t$ is nonpositive: $\langle -\lambda,\tau\rangle+\Vert\tau\Vert^2\leq 0$. 
\end{proof}

%%%%%%%%%%%%%%%%%%%%%%%%%%%%%%%%
\subsubsection{The stratum in terms of moment polyhedrons}
%%%%%%%%%%%%%%%%%%%%%%%%%%%%%%%%

\begin{proposition}
\label{prop:stratproj}
Assume that $X^{\rm ss}$ is empty.
Let $\tau\in \Xgot_{*}(T)_\Q$ be dominant such that $X_{\langle \tau\rangle}$ is the open stratum in $X$. 
  
Then, $\tau^\flat$ is the orthogonal  projection of $0$ onto $\Ccal(X)$.
\end{proposition}

\begin{proof}
We first prove that $\tau^\flat$ belongs to $\Ccal(X)$.
Fix a positive integer $\ell$ such that $\ell\tau^\flat\in \Xgot^*(T)$.
Let $x=(v,m)\in X_{\langle \tau\rangle}$. Theorem~\ref{th:shiftingtrick} implies that $(x,[\tilde m^\ell\otimes v_{\ell\tau^\flat}])$ is semistable. 

Set $Y=G.[v_{\ell\tau^\flat}]\subset\Pbb(V_{\ell\tau^\flat})$. Observe that $Y\simeq G/P(-\tau)$.
Consider $V\times\Pbb E \times Y\subset V\times\Pbb(\mathrm{Sym}^\ell(E)\otimes V_{\ell\tau^\flat})$. 
Then, there exist a positive integer $n$ and $f\in(\C[V]\otimes \C[\mathrm{Sym}^\ell(E)\otimes V_{\ell\tau^\flat}]_n)^G$ 
such that $f(v,\tilde m^\ell\otimes v_{\ell\tau^\flat})\neq 0$. 

But, on the one hand, the restriction map on the affine cones induces, for any nonnegative integer $h$ 
$$
\C[V]\otimes \C[\mathrm{Sym}^\ell(E)\otimes V_{\ell\tau^\flat}]_h\longrightarrow  
H^0(V\times\Pbb E \times Y,\Ocal(nh)\otimes\Lcal_{n\tau^\flat}^h).
$$
And, on the other hand, 
$$
H^0(V\times\Pbb E \times Y,\Ocal(nh)\otimes\Lcal_{n\tau^\flat}^h)\simeq
\C[V]\otimes \mathrm{Sym}^{nh}(E)^*\otimes V(nh\tau^\flat)^*.
$$
Hence, $\C[V]\otimes \mathrm{Sym}^{nh}(E)^*\otimes V(nh\tau^\flat)^*$ contains a nonzero $G$-invariant; that is, 
$V(nh\tau^\flat)$ embeds in $\C[V]\otimes \C[E]_{nh}$. 
Thus, $\tau^\flat$ belongs to $\Ccal(X)$.

But, Lemma~\ref{lem:ineqopenstrata} implies that $\Ccal(X)$ is contained in the half space defined by 
$(\tau^\flat,\lambda)\geq\Vert\tau^\flat\Vert^2$. Thus, $\tau^\flat$ is the orthogonal projection of $0$ onto $\Ccal(X)$.
\end{proof}

\medskip

     The next result is the non-abelian analog of Proposition~\ref{prop:ssG=T}.
     
      \begin{proposition}
      Let $x\in V\times\Pbb E$, and let $\Ccal(\overline{Gx})$ be the moment polyhedron of the variety $\overline{Gx}\subset V\times\Pbb E$. 
       \begin{enumerate}
       \item $x$ is semistable if and only if $0\in \Ccal(\overline{Gx})$.
       \item If $x$ is unstable then : 
       \begin{enumerate}
       \item $\mathbf{M}_{rel}(x)$ is the distance between $0$ and $\Ccal(\overline{Gx})$.
       \item The orthogonal projection of $0$ onto $\Ccal(\overline{Gx})$ defines\footnote{Using the identification $\tgot^*_0\simeq\tgot_0$ given by the invariant scalar product.} 
       an optimal destabilizing $1$-parameter subgroup of some $x'\in Gx$.
       \end{enumerate}
          \end{enumerate}
     \end{proposition}

\begin{proof}
Point  (i) follows from the definition of semistability, and point (ii) is a direct consequence of Proposition~\ref{prop:stratproj} applied to $X=\overline{Gx}$.   
\end{proof}

%%%%%%%%%%%%%%%%%%%%%%%%%%%%%
\subsubsection{On the geometry of the strata}
%%%%%%%%%%%%%%%%%%%%%%%%%%%%%

We can now improve Proposition~\ref{prop:Kirwan}:

\begin{proposition}
\label{prop:stratiso}
The  map 
$$
\begin{array}{ccl}
G\times_{P(\tau)}X_{\tau} &\longrightarrow  &X_{\langle\tau\rangle}\\
{}[g:z]& \longmapsto & gz
\end{array}
$$
is an isomorphism. 
\end{proposition}

\begin{proof}
Here, we adapt an argument of \cite{Br87}.
It is sufficient to prove that there exists a $G$-equivariant morphism (see {\it e.g.} \cite[Appendice]{Res04})
$$
\Gamma\,:\,X_{\langle\tau\rangle}\longrightarrow  G/P(\tau).
$$

Up to replacing $X$ by the closure of $X_{\langle\tau\rangle}$, 
we may assume that $X^{\rm ss}=\emptyset$ and that $X_{\langle\tau\rangle}$ is the open stratum. 
Let $\beta$ be the orthogonal projection of $0$ on $\Ccal(X)$. 

Let $x=(v,m)\in X$. 
Since $\Ccal(\overline{Gx})\subset\Ccal(X)$, Proposition~\ref{prop:stratproj} implies that 
$x\in X_{\langle\tau\rangle}$ if and only if $\beta\in\Ccal(\overline{Gx})$. 
Hence, for any $x\in X_{\langle\tau\rangle}$ there exist $g\in G$, $h\in\N_{>0}$ and 
$f\in (\C[V]\otimes\C[E]_h)^{(B)_{h\beta}}$  such that $f(g(v,\tilde m))\neq 0$. 
Since $X_{\langle\tau\rangle}$  is paracompact, one can find a finite number of triples
$(g_i,h_i,f_i)_{1\leq i\leq s}$ such that 
$$
\forall x\in X_{\langle\tau\rangle},\quad \exists i, \quad f_i(g_i(v,\tilde m))\neq 0.
$$
We can assume that all polynomial functions $f_i$ have the same degree $h$ 
(even if this means replacing them with some of their powers)

Let $M_{h\beta}$ be the $G$-module generated by the $f_i$'s in $\C[V]\otimes\C[E]_{h}$. 
We get 
\begin{equation}
\label{eq:openstratcov}
\forall x=(v,m)\in X_{\langle\tau\rangle} \qquad \exists f\in M_{h\beta} \quad f(v,\tilde m)\neq 0.
\end{equation}

The inclusion $M_{h\beta}\subset  R_h(X)$ induces a $G$-equivariant map
$$
\tilde\varphi\,:\,\tilde X\longrightarrow  M_{h\beta}^*,
$$
that is homogeneous of degree $h$ in the variable in $E$.
Moreover, \eqref{eq:openstratcov} implies that the cone over the open stratum is mapped 
on $M_{h\beta}^*-\{0\}$. We get a well-defined morphism
$$
\bar \varphi\,:\,X_{\langle\tau\rangle}\longrightarrow \Pbb(M_{h\beta}^*).
$$

Consider now $\Hom^G(V_{h\beta}^*,M_{h\beta}^*)$ and the canonical isomorphism
$$
\begin{array}{ccl}
  V_{h\beta}^*\otimes\Hom^G(V_{h\beta}^*,M_{h\beta}^*)&\longrightarrow &M_{h\beta}^*\\
  \psi\otimes u&\longmapsto&u(\psi).
\end{array}
$$
A point in the projective space of a representation is said to be primitive if it is 
fixed by some Borel subgroup of $G$. Consider the set $C$ (resp. $C_{h\beta}$) of 
primitive points in $\Pbb(M_{h\beta}^*)$ (resp. in $\Pbb(V_{h\beta}^*)$). Then
\begin{equation}
  \label{eq:isocone}
  C\simeq C_{h\beta}\times\Pbb(\Hom^G(V_{h\beta}^*,M_{h\beta}^*))
  \qquad\mbox{and}\qquad
  C_{h\beta}\simeq G/P(\tau). 
\end{equation}

Similarly, write $M_{h\beta}=V_{h\beta}\otimes \C^p$, where $\C^p=\Hom^G(V_{h\beta},M_{h\beta})$.
Consider the $G$-equivariant map 
$$
\psi\,:\,\mathrm{Sym}^\bullet M_{h\beta}\longrightarrow R(X),
$$
defined by the universal property of the symmetric algebra.
Consider on $\Xgot^*(T)$ the dominance order $\leq$. 
Fix a positive integer $n$. 
The weights of $T$ on $\mathrm{Sym}^n M_{h\beta}$ are sums of $n$ weights of $T$ on $M_{h\beta}$ that is on $V_{h\beta}$.
In particular, any $T$-weight $\chi$ on $\mathrm{Sym}^n M_{h\beta}$ satisfies $\chi\leq nh\beta$.
One deduces that 
$$
\mathrm{Sym}^n M_{h\beta}=\left(V_{nh\beta}\otimes \mathrm{Sym}^n\C^p\right)\oplus S,
$$
where $S$ is a sum of highest weight modules $V_\lambda$ such that $\lambda<nh\beta$. 
On the other hand, by Lemma~\ref{lem:ineqopenstrata}, if $V_\nu$ is contained in $R_{nh}(X)$ then
$(\nu,\beta)\geq nh\Vert\beta\Vert^2$.
As a consequence, $S$ is contained in the kernel of $\psi$. 
This means that $\bar\varphi$ factors by the immersion $C\subset\Pbb(M_{h\beta})$. 
We obtain a morphism 
$$
\bar\varphi_1\,:\,X_{\langle\tau\rangle}\longrightarrow  C.
$$
But, by \eqref{eq:isocone}, $C$ projects on $G/P(\tau)$, and one gets the expected morphism $\Gamma$ by 
composing $\bar\varphi_1$ with this projection. 
\end{proof}

\begin{corollary}
\label{cor:smstrat}
If $X$ is smooth, each strata $X_{\langle\tau\rangle}$ is smooth.
\end{corollary}

\begin{proof}
By \cite{Fogarty} or \cite{Iversen}, the fixed point set $X^\tau$ is smooth. 
Observe that the map $x\mapsto \varpi_{rel}(x,\tau)$ is constant on the irreducible components of $X^\tau$, 
since this fact is true when $X=V\times\Pbb E$.
Let $Y$ be the union of the irreducible components of $X^\tau$ where $\varpi_{rel}(y,\tau)=\Vert\tau\Vert^2$. 
Observe that Theorem~\ref{th:optss} implies that $Z_\tau$ is open in $Y$, and hence, $Z_\tau$ is smooth.\\

Let $Y^+=\{x\in X\,:\,\lim_{t\to \infty}\tau(t)x \mbox{ exists and belongs to }Y\}$, 
the associated Bia\l ynicki-Birula cell.
By \cite{Bia02}, $Y^+$ is a locally closed smooth subvariety of $X$. 
But Proposition~\ref{prop:Kirwan} implies that $X_\tau$ is open in $Y^+$. 
Hence, $X_\tau$ is smooth.

Finally, Proposition~\ref{prop:stratiso} shows that $X_{\langle\tau\rangle}$ is smooth.
\end{proof}

%%%%%%%%%%%%%%%%%%%%%%%%%%%%%%%%%%%%%%%%%
\subsection{The case of a $B$-stable variety}\label{sec:polyhedron:B}
%%%%%%%%%%%%%%%%%%%%%%%%%%%%%%%%%%%%%%%%%%%%%

\subsubsection{The polyhedron $\Ccal(X)$}

Let $U$ be the unipotent radical of $B$.
Let $\lambda$ be a rational dominant weight and $n$ be a positive integer. 
As observed in the proof of Lemma~\ref{lem:CX}, $V_{n\lambda}\subset R_{n}(X)$ if and only if $R_{n}(X)^U$ 
contains an eigenvector of weight $n\lambda$ for $T$. 
As observed in \cite{GSj06}, this last condition only needs $X$ to be $B$-stable to make sense. 
So, for any closed $B$-stable {\em irreducible} subvariety $X\subset V\times\Pbb E$, we define its moment polyhedron by
\begin{equation}
  \label{eq:defmpB}
  \Ccal(X)=\left\{\lambda\in \Xgot^*(T)_\Q\;:\; \exists n\geq 1,\quad n\lambda\in\Xgot^*(T)^+\mbox{ and } \C_{n\lambda}\subset R_{n}(X)^{U}\right\}.
\end{equation}
Here $\C_{n\lambda}$ denotes the $1$-dimensional $T$-module associated to the character $n\lambda$.

In this context, Remark~\ref{rem:P-X-0} remains valid, and we have a weakened version of lemma~\ref{lem:CX}.

\begin{lemma}\label{lem:P-X-B-0}
Let $X\subset V\times\Pbb E$ be a $B$-stable irreducible subvariety.
\begin{enumerate}
\item $0\in \Ccal(X)\iff X^{ss}:=X\cap (V\times\Pbb E)^{ss}\neq\emptyset$.
\item $\Ccal(X)$ is a $\Q$-convex set of $\Xgot^*(T)_\Q$, with recession cone containing $\Ccal(X_0)$.
\end{enumerate}
\end{lemma}
\begin{proof}
Condition $0\in \Ccal(X)$ is equivalent to the existence of $\varphi\in R_n(V\times\Pbb E)^B=R_n(V\times\Pbb E)^G$ such that $\varphi\vert_X\neq 0$, and the later condition means that $X^{ss}$ is non-empty. Therefore, Point $(i)$ is proved. 

For $(ii)$, we adapt the proof of Lemma~\ref{lem:CX} with a few  modifications. $S_X=\{(\chi, n)\in \Xgot^*(T)\times \N\;:\; \C_\chi\subset  R_n(X)^U\}$ is the set of weights of $T\times\C^*$ acting on $R(X)^U=\C[\tilde X]^U$.  Here, since $X$ is invariant only with respect to $B$, the algebra $\C[\tilde X]^U$ is not necessarily finitely generated. Nevertheless, $S_X$ is a semigroup of $\Xgot^*(T)\times \N$ because $R(X)^U$ is an integral domain, and we can deduce that $\Ccal(X)$ is 
a $\Q$-convex set of  $\Xgot^*(T)_\Q$, since it is the image of $S_X\cap \Xgot^*(T)\times \N^*$ under the map $(\chi,n)\mapsto\frac \chi n$.

Since $R(X)^U$ is a module over $\C[X_0]^U$, $\Ccal(X_0)$ is contained in the recession cone of $\Ccal(X)$. 
\end{proof}

Let $\chi$ be a dominant character of $T$. In the orbit $\Ocal_\chi=G.[v_\chi]\subset\Pbb(V_\chi)$, the element $[v_\chi]$ is fixed by the action of the Borel subgroup $B$.

In this context, Lemma~\ref{lem:Shifting} becomes

\begin{lemma}\label{lem:Shifting-2}
Let $X\subset V\times\Pbb E$ be a closed irreducible $B$-stable subvariety. Let $\lambda$ be a rational dominant weight, and $\ell\geq 1$ such that $\ell\lambda\in\Xgot^*(T)^+$. Thus,  
$X_{[\ell]}\times \{[v_{\ell\lambda}]\}$ is a closed irreducible $B$-stable subvariety of $V\times\Pbb(\mathrm{Sym}^\ell E\otimes V_{\ell\lambda})$. The following conditions are equivalent
\begin{enumerate}
\item $\lambda\in\Ccal (X)$.
\item $(X_{[\ell]}\times \{[v_{\ell\lambda}]\})^{ss}\neq\emptyset$.
\item $0 \in \Ccal(X_{[\ell]}\times \{[v_{\ell\lambda}]\})$.
\end{enumerate}
\end{lemma} 
\begin{proof}
Equivalence $(ii)\iff (iii)$ is proved in Lemma~\ref{lem:P-X-B-0}, and $(i)\iff (iii)$ is a consequence of the following equivalences:
\begin{align*}
\exists n\geq 1, \C_{n\lambda}\subset R_{n}(X)^{U}		&\iff  \exists n\geq 1, \C_{n\ell\lambda}\subset R_{n\ell}(X)^{U} \\
  											&\iff  \exists n\geq 1, \left(R_n(X_{[\ell]})\otimes\C_{n\ell\lambda}^*\right)^B\neq 0 \\
											&\iff  \exists n\geq 1, \left(R_n(X_{[\ell]}\times \{[v_{\ell\lambda}]\})\right)^B\neq 0 \\
											&\iff 0 \in \Ccal(X_{[\ell]}\times \{[v_{\ell\lambda}]\}).
\end{align*}
\end{proof}

\begin{remark}
It should be noted that when working with a $B$-stable subvariety, the set $\Ccal(X)$ may be empty. This is the case, for example, if we take $X:=\{[v_\chi]\}\subset \Pbb(V_\chi)$.
\end{remark}

\subsubsection{The weight polyhedron $\Pcal_U(x)$}

%Recall that $U$ is the unipotent radical of the Borel subgroup $B$. 
Following \cite{Frantz}, we associate to $x\in V\times\Pbb E$, the set 
\begin{equation}
  \label{def:PUx}
  \Pcal_U(x):=\bigcap_{u\in U}\Pcal_T(ux)\ \subset \tgot_0^*,
\end{equation}
where the closed convex polyhedrons $\Pcal_T(y), y\in V\times\Pbb E,$ are defined in (\ref{eq:defPxy}). 
Since in this intersection, only finitely many polyhedrons can occur, 
$\Pcal_U(x)$ is a closed convex rational polyhedron. 

\begin{lemma}
  \label{lem:ssPUx}
  The point $x$ is semistable if and only if $0$ belongs to $\Pcal_U(x)$.
\end{lemma}

\begin{proof}
If $x$ is semistable, so is $ux$. 
Hence, $0$ belongs to $\Pcal_T(ux)$. 

Conversely, assume that $x$ is unstable. Let $\tau$ be an optimal destabilizing vector for $x$. 
Since $P(\tau)\cap B$ contains a maximal torus of $G$, $\tau$ is $P(\tau)$-conjugated to a $1$-parameter subgroup of $B$. 
We may assume that $\tau\in\Xgot_*(B)$. Now, $B=UT$ and there exists $u\in U$ such that $u\tau u\inv$ is a $1$-parameter subgroup of $T$. 
Then, $0$ does not belong to $\Pcal_T(ux)$.
\end{proof}

\subsubsection{A description of $\Ccal(X)$}

The following result was obtained by M. Frantz \cite{Frantz}, when $X$ is a $G$-invariant projective variety.

\begin{theorem} \label{th:FGS}
Let $X$ be closed irreducible $B$-stable subvariety of $V\times\Pbb E$. Then, 
$$
\Ccal(X)\supset\Xgot^*(T)^+_\Q\cap -\Pcal_U(x),\qquad \forall x\in X,
$$
with equality for all $x$ in a nonempty Zariski open subset of $X$.
\end{theorem}

\begin{proof}Let $\lambda$ be a rational dominant weight, and $\ell\geq 1$ such that $\ell\lambda\in\Xgot^*(T)^+$. 
Then, by Lemma~\ref{lem:Shifting-2}, $\lambda\in\Ccal(X)$ if and only if there 
exists $x\in X$ such that $(x^\ell,[v_{\ell\lambda}])\in X_{[\ell]}\times \Ocal_{\ell\lambda}$ is semistable. 
By Lemma~\ref{lem:ssPUx}, this is equivalent to $0\in \Pcal_U((x^\ell,[v_{\ell\lambda}]))$. 
But, $[v_{\ell\lambda}]$ being fixed by $B$, we have
$$
\Pcal_U((x^\ell,[v_{\ell\lambda}])=\ell\left(\lambda+\Pcal_U(x)\right).
$$
This proves that 
$$
\lambda\in\Ccal(X)\iff \exists x\in X,\quad \lambda\in -\Pcal_U(x).
$$
The inclusion of the statement follows. Observe that, for every $\lambda$ as before, the above proof provides 
an open set $\Omega_{\lambda}$ of points $x$ such that $-\lambda\in\Pcal_U(x)$.
Let $\Pcal_1,\dots,\Pcal_s$ be the collection of polyhedrons  $\Pcal_U(x)$ that are maximal for the inclusion. 
For each $i$, one can fix $\lambda_i\in\Pcal_i$ that belongs to no other $\Pcal_j$. 
Since $X$ is irreducible, the intersection of the $\Omega_{\lambda_i}$ is nonempty. 
Pick $x$ in it. 
Then $\Pcal_U(x)$ contains any $\lambda_i$; this is a contradiction unless $s=1$. 
The last part of the statement follows.
\end{proof}

\begin{corollary}\label{coro:Bx} For any $x\in V\times \Pbb E$, we have 
$$
\Ccal(\overline{Bx})=\Xgot^*(T)^+_\Q\cap -\Pcal_U(x).
$$
\end{corollary}

\begin{proof}Thanks to the previous Theorem we have $\Pcal(\overline{Bx})=\Xgot^*(T)^+_\Q\cap -\Pcal_U(y)$ for $y$ belonging in a Zariski dense open subset 
$\Omega\subset \overline{Bx}$. The result holds since $\Pcal_U(y)=\Pcal_U(x)$ for any $y\in \Omega\cap Bx$.
\end{proof}

\begin{definition}\label{def:P-T-X} When working with an irreducible $T$-stable subvariety $X\subset V\times\Pbb E$, we define the polyhedron
$\Ccal_T(X)=\{\lambda\in \Xgot^*(T)_\Q\;:\; \exists n\geq 1,\, n\lambda\in\Xgot^*(T)\mbox{ and } \C_{n\lambda}\subset R_{n}(X)\}$.
\end{definition}

Corollary~\ref{coro:Bx}, applied to the case where the reductive group is abelian, says that the polyhedron $\Ccal_T(\overline{Tx})$, which is  
attached to the $T$-variety $\overline{Tx}\subset V\times\Pbb E$, is equal to $\Xgot^*(T)_\Q\cap-\Pcal_T(x)$. 
The identity of Corollary~\ref{coro:Bx} is then equivalent to 
\begin{equation} \label{equation:coro-Bx}
\Ccal(\overline{Bx})=\Xgot^*(T)^+_\Q\cap \bigcap_{u\in U}\Ccal_T\left(\overline{Tux}\right).
\end{equation}

%%%%%%%%%%%%%%%%%%%%%%%%%%%%%%%%%%%%%%%%%%%%%%%%%%%%
\section{Semistability in a non-compact K\"{a}hler framework}
%%%%%%%%%%%%%%%%%%%%%%%%%%%%%%%%%%%%%%%%%%%%%%%%%%%%
\label{sec:semistability-kahler}

In the next chapters, $G$ is a connected complex reductive Lie group, and $K$ is a maximal compact subgroup of $G$.

Throughout the paper, by ``K\"ahler Hamiltonian $G$-manifold'', we mean a complex manifold $N$ equipped with 
an action of $G$ by holomorphic transformations and a $K$-invariant K\"ahler $2$-form $\Omega_N$. 
We suppose furthermore that the $K$-action on $(N,\Omega_N)$ is Hamiltonian. The moment map $\Phi_N: N\to\kgot^*$ 
satisfies the relations 
$$
d\langle \Phi_N, \lambda\rangle\vert_n= \Omega\vert_n(\lambda\cdot n ,-),\qquad \forall (\lambda,n)\in\kgot\times N, 
$$ 
where $\lambda\cdot n=\frac{d}{dt}\vert_0 e^{t\lambda}n$.

In the following sections, we'll look at questions relative to GIT in the framework of non-compact K\"ahler Hamiltonian $G$-manifolds. 
One of the main difference with Section~\ref{sec:HKKN-algebraic} is that we are working with K\"ahler $2$-forms that are not necessarily integral. For example, 
a coadjoint orbit $K\mu$ is a K\"ahler Hamiltonian $G$-manifold for any $\mu\in\tgot^*_{0}$. Even if $K\mu$ is always a projective variety, the 
K\"ahler $2$-form $\Omega_\mu$ we have on it is integral  only if $\mu\in\Lambda^*$.

%%%%%%%%%%%%%%%%%%%%%%%%%%%%%%%%%%%%%%%%%%%%%%%%%%%%
\subsection{$\Phi$-semistability}\label{sec:phi-semistability}
%%%%%%%%%%%%%%%%%%%%%%%%%%%%%%%%%%%%%%%%%%%%%%%%%%%%

Let $(M,\Omega)$ be a compact K\"{a}hler Hamiltonian $G$-manifold, with moment map $\Phi_M: M\to \kgot^*$. 
Let $V$ be a $G$-module, that is equipped with a $K$-invariant Hermitian structure $\langle .,.\rangle_V$. The moment map 
$\Phi_V: V\to \kgot^*$ relative to the symplectic structure $\Omega_V:= - {\rm Im}(\langle .,.\rangle_V)$ is defined by the relations 
$\langle\Phi_V(v),\lambda\rangle=\tfrac{i}{2}\langle \lambda v,v\rangle_V$.

\begin{definition}
For any $r>0$, we consider the moment map $\Phi_r: V\times M\to \kgot^*$ defined by the relations : 
$\Phi_r(v,m)=r\,\Phi_V(v)+\Phi_M(m)$,\  $\forall (v,m)\in V\times M$. 
When $r=1$, the moment map $\Phi_r$ is denoted by $\Phi$.
\end{definition}

\begin{definition}
\begin{itemize}
\item $(v,m)\in V\times M$ is $\Phi_r$-semistable if $\overline{G(v,m)}\cap \{\Phi_r=0\}\neq \emptyset$.
\item $(v,m)\in V\times M$ is $\Phi_r$-unstable if $\overline{G(v,m)}\cap \{\Phi_r=0\}= \emptyset$.
\item $(V\times M)^{\Phi_r-ss}$ is the subset of $\Phi_r$-semistable elements. 
\end{itemize}
\end{definition}

For any $r>0$, the dilatation $\delta_r:(v,m)\mapsto (\sqrt{r}v,m)$ on $V\times M$ is $G$-equivariant and 
satisfies $\Phi\circ \delta_r=\Phi_r$. Thus, 
$\delta_r\left((V\times M)^{\Phi_r-ss}\right)$ is equal to $(V\times M)^{\Phi-ss}$.

The following fact is proved in Section~\ref{sec:numerical-invariant-phi}.

\begin{proposition}
$(V\times M)^{\Phi_r-ss}$ is equal to $(V\times M)^{\Phi-ss}$ for any $r>0$. 
\end{proposition}

The last result shows that $(V\times M)^{\Phi-ss}$ is stable under the dilatations $\delta_r, r>0$.

%%%%%%%%%%%%%%%%%%%%%%%%%%%%%%%%%%%%%%%%%%%%%%%%%%%%
\subsection{Kemp-Ness Theorem in $V\times \Pbb E$}
%%%%%%%%%%%%%%%%%%%%%%%%%%%%%%%%%%%%%%%%%%%%%%%%%%%%

We work with two $G$-modules $E$ and $V$, both equipped with $K$-invariant Hermitian structures $\langle .,.\rangle_E$ and $\langle .,.\rangle_V$. 
The definition of {\em relative semistability} was given in Section~\ref{sec:defss}.

In the next section, we characterize relative semistability in terms of the function $\Psi: V\times(E-\{0\})\to \R$ that is defined by the relation
\begin{equation}\label{eq:psi-relative}
\Psi(v,w)=\tfrac{1}{4}\|v\|^2 + \log(\|w\|).
\end{equation}

%%%%%%%%%%%%%%%%%%%%%%%%%%%%%%
\subsubsection{A.D. King's lemma}
%%%%%%%%%%%%%%%%%%%%%%%%%%%%%%

The following lemma was proved by King \cite{King94} when $\dim V=1$. 

\begin{lemma}\label{lem:king-proof}
\begin{enumerate}
\item \label{lem:king-proof-ass1} Let $\Xgot$ be any affine subvariety of $V\times E$ disjoint from
$V\times\{0\}$. Then the function $\exp(\Psi)$ restricted to $\Xgot$ is proper (and thus achieves its infimum).
\item Let $\Ocal$ be a $G$-orbit in $V\times E$ disjoint from $V\times\{0\}$. Then $\Ocal$ is closed if and only if the restriction of $\Psi$ to $\Ocal$ achieves its infimum.
\end{enumerate}
\end{lemma}

\begin{proof}
For \ref{lem:king-proof-ass1}, 
  it is sufficient to prove that the set $\Xgot\cap \{\Psi\leq R\}$ is
  bounded for any real $R$.

  Let $P(v,w)=\sum_{i=1}^N P_i(v)Q_i(w)$ be a polynomial equal to $1$
  on $\Xgot$ and that vanishes on $V\times\{0\}$.  We can suppose
  that each polynomial $Q_i(w)$ is homogeneous of degree $k_i\geq 1$:
  there exists $C\geq 0$, such that
  $|Q_i(w)|\leq C \|w\|^{k_i},\forall w\in V$, for all
  $i\in \{1,\ldots, N\}$.

  The condition $\Psi(v,w)\leq R$ means that
  $\|w\|\leq e^R e^{-\frac{\|v\|^2}{4}}\leq e^R$. If
  $(v,w)\in \Xgot\cap \{\Psi\leq R\}$ we have
  \begin{align*}
    1=P(v,w)&\leq \sum_{j=1}^N |P_i(v)|\, |Q_i(w)|\\
            &\leq C\sum_{j=1}^N |P_i(v)| e^{k_i R} e^{-\frac{k_i}{4}\|v\|^2}\leq C'\left(\sum_{j=1}^N |P_i(v)| \right) e^{-\frac{1}{4}\|v\|^2},
  \end{align*}
  with $C'=Ce^{\sup(k_i) R}$. The inequalities $\|w\|\leq e^R$ and
  $e^{\frac{1}{4}\|v\|^2}\leq C'\left(\sum_{j=1}^N |P_i(v)| \right)$
  shows that $\Xgot\cap \{\Psi\leq R\}$ is bounded in $V\times E$.

\bigskip  Let us prove the second point. If the orbit $\Ocal$ is closed, it
  is an affine subvariety disjoint from $V\times\{0\}$. From the
  first point, we see that the function $\Psi$ restricted to $\Ocal$
  achieves its infimum.

  Suppose now that the restriction of $\Psi$ to $\Ocal$ achieves its
  infimum at $x\in \Ocal$. If $\Ocal$ is not closed, the
  Hilbert-Mumford criterium tells us that there exists a $1$-parameter
  subgroup $\lambda: \C^*\to G$ such that the limit
  $\lim_{t\to \infty} \lambda(t) x$ exists but does not belong to
  $\Ocal$. Let $g\in G$ such that $\lambda_0:= {\rm Ad} (g)(\lambda)$, when restricted to $S^1\subset \C^*$, 
  takes value in $K$. There exists
  $(k,p)\in K\times P({\lambda_0})$ such that $g=pk$. We see then that
  the restriction of $\Psi$ to $\Ocal$ achieves its infimum at $y=kx$,
  and that the limit $\lim_{t\to \infty} \lambda_0(t) y$ exists but does
  not belong to $\Ocal$.  If $x=(v,w)$, we get
  $\lambda_0(t) v=\sum_{k\leq 0} t^k v_k$,
  $\lambda_0(t) w=\sum_{\ell\leq 0} t^\ell w_\ell$ and
$$
\Psi( \lambda_0(t) x)=\frac{1}{4}\sum_{k\leq 0} |t|^{2k} \|v_k\|^2+
\log\left(\sum_{\ell\leq 0} |t|^{2\ell} \|w_\ell\|^2\right),\quad
\forall t\in\C^*.
$$
The condition $\Psi( \lambda_0(t) x)\geq \Psi(x), \forall t\in\C^*$
imposes that $\|v_k\|=\|w_\ell\|=0$ if $k,\ell\leq -1$, so
$\lambda_0(t) x = x, \forall t\in\C^*$.  It contradicts the hypothesis
that $\lim_{t\to \infty} \lambda_0(t) x$ does not belong to $\Ocal$.
\end{proof}

\begin{proposition}\label{prop:kempf-ness-relatif}
Let $(v,m)\in V\times \Pbb E$. The following conditions are equivalent:
\begin{enumerate}
\item  \label{prop:kempf-ness-relatif-ass1}$(v,m)\in V\times \Pbb E$ is relatively semistable.
\item \label{prop:kempf-ness-relatif-ass2} 
The function $\Psi$ restricted to $\overline{G(v,\tilde m)}$ achieves its infimum.
\item \label{prop:kempf-ness-relatif-ass3} 
The function $\Psi$ is bounded from below when restricted to $G(v,\tilde m)$.
\end{enumerate}
\end{proposition}

\begin{proof}
  Implication \ref{prop:kempf-ness-relatif-ass1} $\Longrightarrow$ \ref{prop:kempf-ness-relatif-ass2} is a consequence of
  the first point of Lemma~\ref{lem:king-proof} applied to
  $\Xgot=\overline{G(v,\tilde m)}$, and Implication~\ref{prop:kempf-ness-relatif-ass2}
  $\Longrightarrow$~\ref{prop:kempf-ness-relatif-ass3} is immediate.

  Implication~\ref{prop:kempf-ness-relatif-ass3} $\Longrightarrow$~\ref{prop:kempf-ness-relatif-ass1} 
  is easy. Suppose that there exists $c\in\R$ such that
  $\Psi(gv,g\tilde m)\geq c,\forall g\in G$. In other words,
  $\|gv\|^2\geq c- \log(\|g\tilde m\|^2),\quad \forall g\in G$.  It shows
  that if a sequence $g_n\tilde m$ tends to $0$, then $\|g_nv\|^2$ tends to
  $+\infty$. Hence, $\overline{G(v,\tilde m)}\cap V\times\{0\}=\emptyset $.
\end{proof}

%%%%%%%%%%%%%%%%%%%%%%%%%%%%%%%%%%%%%%
\subsubsection{Characterization of semistability \`a la Kempf-Ness}
%%%%%%%%%%%%%%%%%%%%%%%%%%%%%%%%%%%%%%
We take the following convention for the moment map $\Phi_V: V\to \kgot^*$, and $\Phi_{\Pbb E }: \Pbb E \to \kgot^*$:
$$
\langle\Phi_V(v),\lambda\rangle=\tfrac{i}{2}\langle\lambda v,v\rangle_V\quad \mathrm{and}\quad 
\langle\Phi_{\Pbb E }(m),\lambda\rangle= i\,\frac{\langle \lambda
  \tilde m,\tilde m\rangle_E}{\langle \tilde m,\tilde m\rangle_E},\quad \forall \lambda\in\kgot.
$$

The moment map $\Phi:V\times \Pbb E\to \kgot^*$ is defined by the relation
$$
\Phi(v,m)= \Phi_V(v) + \Phi_{\Pbb E }(m).
$$

The relation between the moment map $\Phi$ and the function $\Psi$
defined in (\ref{eq:psi-relative}) is given by the following
computation: $\forall\lambda\in\kgot$, $\forall (v,m)\in V\times \Pbb E$, we have
\begin{align}\label{eq:psi-relative-1}
\frac{d}{dt}\Psi\Big(e^{it\lambda}\cdot (v,\tilde m)\Big)&= \, \, \langle\Phi\Big(e^{it\lambda}\cdot(v,m)\Big), \lambda \rangle.\\
\frac{d^2}{dt^2}\Psi\Big(e^{it\lambda}\cdot (v,\tilde m)\Big)&= \|\lambda\cdot(e^{it\lambda}v)\|^2_V+ \|\lambda\cdot(e^{it\lambda}m)\|^2_{FS}\ \geq 0,
\label{eq:psi-relative-2}
\end{align}
where $\|-\|_{FS}$ is the Fubini-Study metric on $\Pbb E$.

In Section~\ref{sec:kempf-ness-functions}, we  explain the link between $\Psi$ and Kempf-Ness functions.

\medskip

\medskip

We obtain the following characterization of semistable elements in terms of the moment map.

\begin{theorem}\label{theo:kempf-ness-relatif}
Let $(v,m)\in V\times \Pbb E$. The following conditions are equivalent:
\begin{enumerate}
\item  \label{theo:kempf-ness-relatif-ass1} 
$(v,m)\in V\times \Pbb E$ is relatively semistable.
\item \label{theo:kempf-ness-relatif-ass2}  
The function $\Psi$ restricted to $\overline{G(v,\tilde m)}$ achieves its infimum.
\item \label{theo:kempf-ness-relatif-ass3}  
The function $\Psi$ is bounded from below when restricted to
$G(v,\tilde m)$.
\item \label{theo:kempf-ness-relatif-ass4} 
$(v,m)\in V\times \Pbb E$ is  $\Phi$-semistable.
\end{enumerate}
\end{theorem}

\medskip

\begin{proof}
  Equivalences~\ref{theo:kempf-ness-relatif-ass1} $\Longleftrightarrow$~\ref{theo:kempf-ness-relatif-ass2}  
  $\Longleftrightarrow$~\ref{theo:kempf-ness-relatif-ass3}   are proved in Proposition~\ref{prop:kempf-ness-relatif}.  Let us prove Implication~\ref{theo:kempf-ness-relatif-ass2}  
  $\Longrightarrow$~\ref{theo:kempf-ness-relatif-ass4}. Suppose that there exists
  $(v_0,\tilde m_0)\in \overline{G(v,\tilde m)}\subset V\times E-\{0\}$ such that
  $\Psi(x)\geq \Psi(v_0,\tilde m_0)$, $\forall x\in
  \overline{G(v,\tilde m)}$. In
  particular, we have
  $\Psi(g(v_0,\tilde m_0))\geq \Psi(v_0,\tilde m_0),\forall g\in G$, and then
  $\frac{d}{dt}|_{t=0}\Psi(e^{it\lambda}(v_0,\tilde m_0))=0$,
  $\forall \lambda\in\kgot$.  Thanks to (\ref{eq:psi-relative-1}), it
  shows that $(v_0,m_0)\in \overline{G(v,m)}\cap \Phi^{-1}(0)$,
  i.e. $(v,m)$ is $\Phi$-semistable.

  The last implication~\ref{theo:kempf-ness-relatif-ass4}  $\Longrightarrow$~\ref{theo:kempf-ness-relatif-ass3}  is proved
  in Section~\ref{sec:kempf-ness-functions} (see Proposition~\ref{prop:phi-minimal-orbit}).
\end{proof}

%%%%%%%%%%%%%%%%%%%%%%%%%%%%%%%%%%%%%%%%%%%%%%%%%%%%
\subsection{The moment map squared}\label{sec:squared}
%%%%%%%%%%%%%%%%%%%%%%%%%%%%%%%%%%%%%%%%%%%%%%%%%%%%

Let us recall how the stability behaviour under the $G$-action is determined by the negative gradient flow lines of the moment map squared,
\begin{align}
\begin{split}
f :  V\times M & \longrightarrow  \R,\label{eq:squared}\\ 
x&\longmapsto  \tfrac{1}{2} \|\Phi(x)\|^2.
\end{split}
\end{align}

Using the identification $\kgot\simeq \kgot^*$ given by our rational invariant scalar product on $\kgot$ (see Section {\bf Notation}), we see that 
the gradient of $f$ is given by 
\begin{equation}\label{eq:gradient-f}
\nabla f(x)=J(\Phi(x)\cdot x),\ \forall x\in V\times M,
\end{equation}
 where $J$ is the complex structure on the tangent space of  $V\times M$. 
Hence, $x\in V\times M$ is a critical point of $f$ if and only if $\Phi(x)\cdot x=0$.

The negative gradient flow line of $f$ through $x\in V\times M$ is the solution $x(t), t\geq 0$ of the differential equation
\begin{equation}\label{eq:flow-f}
\frac{d}{dt}x(t) = - J\Big(\Phi(x(t))\cdot x(t)\Big),\quad {\rm and} \quad x(0) = x.
\end{equation}

\begin{lemma}
The moment map $\Phi$ is proper if and only if $\C[V]^G=\C$.
\end{lemma}

\begin{proof}
  Since $M$ is compact, the map $\Phi$ is proper if and only if the
  moment map $\Phi_V: V\to \kgot^*$ is proper.  Now we use that the
  following statements are equivalent: (a) $\Phi_V$ is proper, (b)
  $\Phi_V^{-1}(0)=\{0\}$ and (c) $\C[V]^G=\C$ (see Lemma 5.2 in \cite{pep09}).
\end{proof}

\medskip

As we will not assume that $\C[V]^G=\C$,  the moment map $\Phi$ is not necessarily proper. But we have the following useful counterpart.

\begin{proposition}\label{prop:phi-proper-orbit}
Let $x\in V\times M$. The moment map $\Phi$ is proper when restricted to $\overline{Gx}$.
\end{proposition}

\begin{proof}
  Let $x=(v,m)$, $C=\sup\|\Phi_{M}\|$ and $R\geq 0$.  The set
  $\overline{Gx}\cap \{\|\Phi\|\leq R\}$ is contained in
$$
\Big(\overline{Gv}\cap \big\{\|\Phi_V\|\leq R+C\big\}\Big)\times M.
$$
In \cite[Lemma~4.10]{Sjamaar98}, Sjamaar proved that the restriction
of $\Phi_V$ to the affine variety $\overline{Gv}$ is a proper map. It
shows that $\overline{Gx}\cap \{\|\Phi\|\leq R\}$ is compact for every
$R\geq 0$.
\end{proof}

\medskip

The following facts are well-known when one works with a compact manifold, or when the moment map is proper \cite{Lerman05}. 
In \cite{Sjamaar98}, Sjamaar proved it for manifolds of the type $V\times\Ocal$, where $\Ocal$ is a coadjoint orbit. 

\begin{proposition}\label{prop:flow-moment-map-square}
Let $x\in V\times M$.
\begin{enumerate}
\item The negative flow $x(t)$ is defined for any $t\in [0,+\infty[$, and  $\{x(t),\ t\geq 0\}$ is contained in the orbit $Gx$.
\item There exists a ball $B\subset V$ such that $\{x(t),\ t\geq 0\}\subset B\times M$.
\item The limit of the negative gradient flow line,  $x_{\infty}:=\lim_{t\to+\infty} x(t)$, exists and is contained in the critical set of $f$.
\end{enumerate}
\end{proposition}

\begin{proof}
  Let $t\in [0,\epsilon[ \mapsto x(t)$ be a solution of
  $(\ref{eq:flow-f})$. Let $t\mapsto g(t)\in G$ be the solution of the
  differential equation
  \begin{equation}\label{eq:g(t)}
    g(t)^{-1}g'(t)= \iu \Phi(x(t))\quad \quad {\rm and}\quad g(0)=e,
  \end{equation}
  that is defined on an interval $[0,\epsilon'[ \subset
  [0,\epsilon[$. Now it is easy to check that $g(t)x(t)=x$,
  $\forall t\in[0,\epsilon'[$. In other words, $x(t)$ belongs to the
  compact set $\overline{Gx}\cap \{y\,:\,\Vert\Phi(y)\Vert\leq \Vert\Phi(x)\Vert\}$ for any
  $t\in[0,\epsilon'[$. This shows that the maximal solutions of
  (\ref{eq:flow-f}) and (\ref{eq:g(t)}) are defined on $[0,+\infty[$.

  The first two points are proved and, the last point follows from the
  first two points and the Lojasiewicz gradient inequality 
  (see \cite{Lerman05} or \cite[Chapter~3]{GRS21}).
\end{proof}

The following property, which is due to Kirwan and Ness in the compact setting, is proved in Proposition~\ref{prop:phi-minimal-orbit}.

\begin{proposition}\label{prop:phi-minimal-orbit-1}
Let $x\in V\times M$. Then $\|\Phi(x_\infty)\|=\inf_{g\in G}\|\Phi(gx)\|$.
\end{proposition}

The previous result permits to see that 
$$
(V\times M)^{\Phi-ss}=\left\{x\in V\times M,\ \Phi(x_\infty)=0\right\}.
$$

\begin{remark}\label{rem:g(t)-lambda}
Suppose that $\lambda\in\kgot$ satisfies $\lambda\cdot x=0$. Hence, $\Phi(x)\in\kgot^*_\lambda$, and the 
solution $t\mapsto g(t)$  of the differential equation (\ref{eq:g(t)}) belongs to the subgroup $G_\lambda$. In particular,
if $x\in(V\times M)^{\Phi-ss}\cap (V\times M)^\lambda$, there exists a sequence $\ell_n\in G_\lambda$ such that $\lim_{n\to\infty} \Phi(\ell_nx)=0$.
\end{remark}

%%%%%%%%%%%%%%%%%%%%%%%%%%%%%%%%%%%%%%%%%%%%%%%%%%%%
\subsection{Kempf-Ness functions}\label{sec:kempf-ness-functions}
%%%%%%%%%%%%%%%%%%%%%%%%%%%%%%%%%%%%%%%%%%%%%%%%%%%%

We equip $G$ with the unique left invariant Riemannian metric which agrees with the
inner product on $\ggot=T_eG$ defined by $\langle \xi_1+ \iu\eta_1,
\xi_2+ \iu\eta_2\rangle_\ggot=\langle \xi_1,
\xi_2\rangle_\kgot+\langle \eta_1, \eta_2\rangle_\kgot$, for any
$\xi_1,\xi_2,\eta_1,\eta_2$ in $\kgot$.

Consider the homogeneous space $\mathbb{X}=G/K$ : it is a complete Riemannian manifold with non-positive sectional curvature. 
We denote by $\pi: G\to \mathbb{X}=G/K$ the projection. Recall that the geodesics on $\mathbb{X}$ are 
of the form $t\mapsto \pi(g\exp(it\lambda))$, with $g\in G$ and $\lambda\in \kgot$.

Let $(N,\Omega_N,\Phi_N)$ be a K\"{a}hler Hamiltonian $G$-manifold. In this section, we introduce the Kempf-Ness 
function $\Psi_n : \mathbb{X}\to \R$ associated to an element $n\in N$.

The following fact is proved in \cite{Mun00} (see also \cite[Chapter~4]{GRS21}).

\begin{proposition}\label{prop:psi-n} Let $n\in N$.
\begin{enumerate}
\item There exists a unique smooth function $\Psi_n:G\to\R$ satisfying $\Psi_n(gk)=\Psi_n(g)$, $\Psi_n(e)=0$, and 
\begin{equation}\label{eq:psi-n-definition}
\frac{d}{dt}\Psi_n(ge^{it\lambda})=-\langle \Phi_N(e^{-it\lambda}g^{-1}n),\lambda\rangle
\end{equation}
for all $(g,k,\lambda,t)\in G\times K\times\kgot\times\R$.
\item For any $ g,h\in G$, we have $$\Psi_{h^{-1}n}(g)=\Psi_n(hg)-\Psi_n(h).$$
\end{enumerate}
\end{proposition}

The function $\Psi_n:G\to\R$ is called the lifted Kempf-Ness function based at $n$. It is $K$-invariant and hence descends to a function 
$\Psi_{n}:\mathbb{X}\to\R$ denoted by the same symbol and called the Kempf-Ness function.

Let's recall the classic examples of Kempf-Ness functions.

\begin{example} \label{example:psi-v}Let $V$ be a $G$-module equipped with a $K$-invariant Hermitian metric $\langle.,.\rangle_V$. 
To any $v\in V$, the Kempf-Ness function $\Psi_v: G\to \R$ is defined by 
$\Psi_v(g)=\tfrac{1}{4}\|g^{-1}v\|_V^2 -\tfrac{1}{4}\|v\|_V^2$.
\end{example}

\begin{example} Let $E$ be a $G$-module equipped with a $K$-invariant Hermitian metric $\langle.,.\rangle_E$. 
The projective space $\Pbb E $ is equipped with the Fubini-Study symplectic form 
$\Omega_{\Pbb E }$, and the moment map $\Phi_{\Pbb E }: \Pbb E \to \kgot^*$ is defined by 
$$
\langle\Phi_{\Pbb E }([y]),\lambda \rangle= i\,\frac{\langle \lambda y,y\rangle_E}{\langle y,y\rangle_E}.
$$
The Kempf-Ness function associated to $[y]\in \Pbb E $ is defined by the relation
$$
\Psi_{[y]}(g)=\log(\| g^{-1}y\|)-\log(\| y\|).
$$
\end{example}

\begin{example}\label{example:psi-mu} Let $K\mu$ be the coadjoint orbit of $K$ passing through $\mu\in\tgot^*_{0,+}$. 
The map $k\mapsto k\mu$ defines an isomorphism $K/K_\mu\simeq K\mu$  where $K_\mu=\{k, k\mu=\mu\}$ 
is the stabilizer subgroup. Let $B^-\subset G$ be the Borel subgroup with Lie algebra equal to $\tgot\oplus \sum_{\alpha<0}\ggot_\alpha$. 
We denote by $\phi_\mu : G\to K\mu$ the map $G\to G/B^{-}\simeq K/T\to K/K_\mu\simeq K\mu$. 

The Kempf-Ness function $\Psi_\mu:G\to\R$ associated to $\mu\in K\mu$ is characterized by the relations
\begin{equation}\label{eq:psi-mu-definition}
\Psi_\mu(gk)=\Psi_\mu(g),\quad \Psi_\mu(e)=0,\quad {\rm and}\quad  \frac{d}{dt}\Psi_\mu(ge^{it\lambda})
=-\langle \phi_\mu(e^{-it\lambda}g^{-1}),\lambda\rangle.
\end{equation}
for all $(g,k,\lambda,t)\in G\times K\times\kgot\times\R$.

Notice that for any $\mu_0,\mu_1\in\tgot^*_{0,+}$, and any $r_0,r_1\geq 0$, we have  
\begin{equation}\label{eq:psi-mu-barycentre}
\Psi_{\mu}=r_1\Psi_{\mu_1}+r_0\Psi_{\mu_0},
\end{equation} 
for $\mu=r_1\mu_1+r_0\mu_0\in\tgot^*_{0,+}$.
\end{example}

\begin{example} Suppose that $\mu$ is a dominant weight. 
Let $V_\mu$ be the irreducible representation with highest weight $\mu$. 
We equip $V_\mu$ with a $K$-invariant Hermitian structure. 
Let $v_\mu\in V_\mu$ be a vector of norm $1$, such that $b\, v_\mu=\chi_\mu(b)v_\mu,\ \forall b\in B$. 
Here the Kempf-Ness function $\Psi_\mu$, characterized by (\ref{eq:psi-mu-definition}), is defined by 
$\Psi_\mu(g)=-\log(\|g^{-1}v_\mu\|)$.
\end{example}

We return to the case of the K\"ahler Hamiltonian  $G$-manifold $N=V\times M$, with $M$ compact. 
In the next theorem, we summarize the properties satisfied by the Kempf-Ness functions $\Psi_x: \mathbb{X}\to \R$ 
associated to $x\in V\times M$. The proof given in \cite[Theorem~4.3]{GRS21} is done for compact manifolds. 
We explain below how to adapt the arguments for 
$V\times M$. We denote by $d_{\mathbb{X}}$ the Riemannian distance on $\mathbb{X}$. Recall that $f= \tfrac{1}{2} \|\Phi\|^2$.

\begin{theorem}\label{theo:kempf-ness-kahler}Let $x\in V\times M$.
\begin{enumerate}
\item \label{theo:kempf-ness-kahler-ass1}
The Kempf-Ness function $\Psi_x:\mathbb{X}\to\R$ is convex along geodesics.
\item \label{theo:kempf-ness-kahler-ass2} 
The map $\varphi_x: G\to V\times M$, defined  by $\varphi_x(g)=g^{-1}x$, intertwines the gradient vector field $\nabla \Psi_x \in {\rm Vect}(G)$
and the gradient vector field $\nabla f \in {\rm Vect}(V\times M)$ : 
\begin{equation}\label{eq:psi-x-theo}
\T\varphi_x\vert_g\left(\nabla \Psi_x\vert_g\right)=\nabla f\vert_{\varphi_x(g)}.
\end{equation}
\item\label{theo:kempf-ness-kahler-ass3} 
Let $g\in G$. There exists a negative gradient flow line $\theta:\R^{\geq 0}\to \mathbb{X}$ of $\Psi_x$ such that $\theta(0)=\pi(g)$.
\item \label{theo:kempf-ness-kahler-ass4} 
Let $\theta_0,\theta_1:\R^{\geq 0}\to \mathbb{X}$ be negative gradient flow lines of the Kempf-Ness function $\Psi_x$:
\begin{enumerate}
\item The function $t\geq 0\mapsto d_{\mathbb{X}}(\theta_1(t),\theta_0(t))$ is non-increasing.
\item The function $t\geq 0\mapsto \Psi_x(\theta_0(t))-\Psi_x(\theta_1(t))$ is bounded.
\end{enumerate}
\item \label{theo:kempf-ness-kahler-ass5} 
Every negative gradient flow line $\theta:\R^{\geq 0}\to \mathbb{X}$ of $\Psi_x$ satisfies 
\begin{equation}\label{eq:psi-x-flow}
\lim_{t\to\infty}\Psi_x(\theta(t))=\inf_{\mathbb{X}}\Psi_x.
\end{equation}
\item \label{theo:kempf-ness-kahler-ass6}
$\Psi_x$ is bounded from below if and only if $\Phi(x_\infty)=0$.
\item \label{theo:kempf-ness-kahler-ass7} 
The limit $\lim_{t\to+\infty}\frac{\Psi_x(\theta(t))}{t}$
exists and does not depend on the negative gradient flow line $\theta$.
\end{enumerate}
\end{theorem}

\begin{remark}
Note that the Kempf-Ness function $\Psi_x$ is not globally Lipschitz when the $V$-component of $x$ is not fixed by the $G$-action.
\end{remark}

\begin{proof}
  The geodesics on $\mathbb{X}$ are of the form
  $\gamma(t)=\pi(g\exp(it\lambda))$, with $g\in G$ and
  $\lambda\in \kgot$. Thanks to (\ref{eq:psi-n-definition}), we see
  that
  $\frac{d^2}{dt^2}\vert_0\Psi_x(\gamma(t))=\|\lambda\cdot(g^{-1}x)\|^2\geq
  0$, where $\|-\|^2$ denotes the Riemannian metric on $V\times
  M$. The first point follows.

  The second point follows from (\ref{eq:gradient-f}) and
  (\ref{eq:psi-n-definition}).

  Let $g\in G$. Let us explain the structure of the negative gradient
  flow line $\theta:\R^{\geq 0}\to \mathbb{X}$ of $\Psi_x$ satisfying
  $\theta(0)=\pi(g)$. Set $y=g^{-1}x\in V\times M$. The negative
  gradient flow line $y(t)$ of $f$ through $y$ is $y(t)=h(t)^{-1}y$
  where $h:[0,+\infty[\to G$ is the solution of the differential
  equation : $h(t)^{-1}h'(t)= i \Phi(y(t)),\ \forall t\geq 0$, and
  $h(0)=e$. Thanks to (\ref{eq:psi-x-theo}), we see that
  $t\geq 0\mapsto \pi(g h(t))$ is the negative gradient flow line of
  $\Psi_x$ passing through $\pi(g)$.

  The first point of Assertion~\ref{theo:kempf-ness-kahler-ass4} is proven in Lemma A.2 of
  \cite{GRS21}. Consider the solutions $g_0,g_1:\R^{\geq 0}\to G$ of
  the differential equation
  $g_\ell^{-1}\frac{d}{dt}g_\ell= \iu \Phi(g^{-1}_\ell x)$ such that
  $\theta_0 = \pi\circ g_0$ and $\theta_1 = \pi\circ g_1$.  In the
  proof of the part {\em iv)} of Theorem 4.3 in \cite{GRS21}, the
  authors obtain the following inequalities:
  \begin{equation}\label{eq:psi-theta-1-2}
    C\|\Phi(g_1^{-1}(t)x)\|\geq \Psi_x(\theta_1(t))-\Psi_x(\theta_0(t))\geq - C \|\Phi(g_0^{-1}(t)x)\|,
  \end{equation}
  with $C=\sup_{t\geq 0}
  d_{\mathbb{X}}(\theta_1(t),\theta_0(t))$. Since the functions
  $t\mapsto \|\Phi(g_\ell^{-1}(t)x)\|$ are non-increasing, the second point
 of \ref{theo:kempf-ness-kahler-ass4} is proved.

  Thanks to (\ref{eq:psi-theta-1-2}), we see that
  $\Psi_x\circ\, \theta_1$ is bounded from below if and only if
  $\Psi_x\circ\, \theta_0$ is bounded from below. 
  Thus,
  $\lim_{t\to\infty}\Psi_x\circ \theta_1(t)=-\infty$ if and only if
  $\lim_{t\to\infty}\Psi_x\circ \theta_0(t)=-\infty$.

  By definition of the Kempf-Ness function, we have
  \begin{equation}\label{eq:psi-phi}
    \frac{d}{dt}\Psi_x\left(\theta_\ell(t)\right)=-\|\Phi(g_\ell(t)^{-1}x)\|^2,\quad \forall t\geq 0.
  \end{equation}
  Thus, $\lim_{t\to\infty} \|\Phi(g^{-1}_\ell(t)x)\|=0$ if the
  non-increasing function $\Psi_x\circ \,\theta_\ell$ is bounded from
  below. Identity (\ref{eq:psi-theta-1-2}) shows that
  $\lim_{t\to\infty}\Psi_x\circ
  \theta_1(t)=\lim_{t\to\infty}\Psi_x\circ \theta_0(t)\in \R$ when the
  functions $\Psi_x\circ\, \theta_\ell$ are bounded from below.

  We have checked that 
  $\lim_{t\to+\infty}\Psi_x(\theta(t))\in \R\cup \{-\infty\}$ does not
  depend on the negative gradient flow line $\theta$. As a result, we
  get (\ref{eq:psi-x-flow}).

  Let $\theta=\pi\circ g$ be the negative gradient flow line of
  $\Psi_x$ such that $g(0)=e$. We have already explained that
  $\Phi(x_\infty)=\lim_{t\to\infty} \Phi(g(t)^{-1}x)=0$ if the
  function $\Psi_x\circ \theta$ is bounded from below. Reciprocally,
  suppose that $\Phi(x_\infty)=0$. By the Lojasiewicz gradient
  inequality, we know that there exists $t_o>0$, and $a,b >0$ such
  that $\|\Phi(g(t)^{-1}x)\|\leq a e^{-bt}$ for $t\geq t_o$. Thus,
  $\frac{d}{dt}\left(\Psi_x\circ
    \theta\right)=-\|\Phi(g(t)^{-1}x)\|^2\geq -a e^{-bt}$ for
  $t\geq t_o$. This proves that
  $\lim_{t\to\infty}\Psi_x\circ \theta(t)=\inf_{\mathbb{X}}\Psi_x$ is
  finite. Thus, $\Psi_x$ is bounded from below. Assertion~\ref{theo:kempf-ness-kahler-ass6} is
  proved.

  Let $\theta=\pi\circ g$ be a negative gradient flow line of
  $\Psi_x$. Let $y=g(0)^{-1}x$. Identity (\ref{eq:psi-phi}) gives that
  $\lim_{t\to\infty}\frac{d}{dt}\Psi_x\left(\theta(t)\right)=-\|\Phi(y_\infty)\|^2$. This
  implies that
  $\lim_{t\to+\infty}\frac{\Psi_x(\theta(t))}{t}=-\|\Phi(y_\infty)\|^2$.
  The fact that the limit
  $\lim_{t\to+\infty}\frac{\Psi_x(\theta(t))}{t}$ does not depend on
  the negative gradient flow line $\theta$ is a consequence of \ref{theo:kempf-ness-kahler-ass4}.
\end{proof}

\medskip

The following important result is a consequence of Theorem~\ref{theo:kempf-ness-kahler}.

\begin{proposition}\label{prop:phi-minimal-orbit}
Let $x\in V\times M$.
\begin{enumerate}
\item $\|\Phi(x_\infty)\|=\inf_{g\in G}\|\Phi(gx)\|$.
\item $\Psi_x$ is bounded from below if and only if $x$ is $\Phi$-semistable.
\end{enumerate}
\end{proposition}

\begin{proof}
 In  Assertion~\ref{theo:kempf-ness-kahler-ass7} of Theorem~\ref{theo:kempf-ness-kahler}, we have
  proved that $\|\Phi(y)\|\geq \|\Phi(y_\infty)\|=\|\Phi(x_\infty)\|$
  for any $y\in Gx$. The first point follows and it proves that $x$ is
  $\Phi$-semistable iff $\Phi(x_\infty)=0$. Thus, the second point is
  a consequence of Assertion~\ref{theo:kempf-ness-kahler-ass6} of Theorem~\ref{theo:kempf-ness-kahler}.
\end{proof}

%%%%%%%%%%%%%%%%%%%%%%%%%%%%%%%%%%%%%%%%%%%%%%%%%%%%
\subsection{Numerical invariant for $\Phi$-semistability}\label{sec:numerical-invariant-phi}
%%%%%%%%%%%%%%%%%%%%%%%%%%%%%%%%%%%%%%%%%%%%%%%%%%%%

In this section, we introduce a $G$-invariant function $\mathbf{M}_{\Phi}: V\times M\to \R\cup\{-\infty\}$ that will characterize $\Phi$-semistability. 

First, we define a map $\varpi_V: V\times\kgot\to\R$, which is the
analytical counterpart of (\ref{eq:defm}) with $V$ in place of $E$. 
Any $\lambda\in\kgot$ defines an orthogonal decomposition $V=\oplus_a V_{\lambda,a}$, where 
$\lambda v= i a v, \forall v\in V_{\lambda,a}$. Any $v\in V$ admits a decomposition
$v=\sum_a v_a$, with $v_a\in V_{\lambda,a}$, and we define 
\begin{equation} \label{eq:varpi-E}
  \varpi_V(v,\lambda):=-\max\{a\,:\, v_a\neq 0\},\quad \forall v\in V-\{0\},
\end{equation}
and $\varpi_V(0,\lambda):=0$.

Let $x\in V\times M$. For any $\lambda\in\kgot$, the function $t\mapsto \langle\Phi(e^{-it\lambda}x),\lambda\rangle$ 
is non-increasing, so we can define
$$
\varpi_{\Phi}(x,\lambda) := \lim_{t\to+\infty} \langle\Phi(e^{-it\lambda}x),\lambda\rangle \in \R\cup\{-\infty\}.
$$
We first notice that $\varpi_{\Phi}(x,\lambda)\leq \langle\Phi(x),\lambda\rangle\leq \| \Phi(x)\|\, \|\lambda\|$.

If $x=(v,m)$ then $e^{-it\lambda}x=(e^{-it\lambda}v,e^{-it\lambda}m)$. Since the trajectory $e^{-it\lambda}m$ is 
the negative gradient flow of the Morse function $\langle \Phi_M,\lambda\rangle$ on the compact manifold $M$, 
it converges to a point $m_\lambda\in M^\lambda$.

\begin{lemma}\label{lem:varphi-finite} Let $x=(v,m)$. 
\begin{enumerate} 
\item $\varpi_{\Phi}(x,\lambda)= -\infty$ if and only if $\varpi_V(v,\lambda)<0$.
\item Suppose that $\varpi_{\Phi}(x,\lambda)\neq -\infty$.  Then, the trajectory $e^{-it\lambda}x$ converges when $t\to +\infty$, and 
$\varpi_{\Phi}(x,\lambda)=\langle\Phi_M(m_\lambda),\lambda\rangle$.
\item Suppose that $\varpi_{\Phi}(x,\lambda)=-\infty$. Then there are $a,c >0$ such that 
$\|e^{-it\lambda}v\| \underset{+\infty}{\sim}c\, e^{at}$.
\end{enumerate}
\end{lemma}

\begin{proof}
  Let us write $v=\sum_a v_a$ with $\lambda v_a= i a v_a$,
  $\forall a$. We have $e^{-it\lambda}v=\sum_{a} e^{ta} v_a$ and
  $\langle\Phi_V(e^{-it\lambda}v),\lambda\rangle= \frac{-1}{2}\sum_a a
  e^{2ta}\|v_a\|^2$.  Two situations may arise: 
  \begin{itemize}
  \item if $\varpi_V(v,\lambda)<0$, then  $\lim_{t\to+\infty}\langle\Phi_V(e^{-it\lambda}v),\lambda\rangle=-\infty$ 
  and $\exists a,c >0$ such that
  $\|e^{-it\lambda}v\| \underset{+\infty}{\sim}c\, e^{at}$.
  \item if $\varpi_V(v,\lambda)\geq 0$, then $\lim_{t\to+\infty}\langle\Phi_V(e^{-it\lambda}v),\lambda\rangle=0$ 
  and the trajectory $e^{-it\lambda}v$ converge in $V$ when
  $t\to +\infty$.
  \end{itemize}
  There considerations proves all the points of the lemma.
\end{proof}

\begin{remark} \label{rem:Phi-r}
Two facts follow from the preceding lemma. First, the quantity $\varpi_{\Phi_r}(x,\lambda)=\varpi_{\Phi}(\delta_r(x),\lambda)$ defined 
with the moment map $\Phi_r$ (see Section~\ref{sec:phi-semistability}) does not depend on $r>0$. Second, $\varpi_{\Phi}(x,\lambda)$ 
is finite if and only if the trajectory $e^{-it\lambda}x$ converge when $t\to +\infty$: the Hamiltonian K\"ahler $G$-manifolds that satisfy 
this property are called ``energy complete'' in \cite{Teleman04}.
\end{remark}

\begin{remark} \label{rem:varpi-phi-infty}
We know from the previous Lemma that $\varpi_{\Phi}(x,\lambda)= -\infty$, $\forall \lambda\neq 0$, if and only if $\varpi_{V}(v,\lambda)<0$, 
$\forall \lambda\neq 0$. This is the case when $v\in V$ is $G$-{\em stable}, i.e. the orbit $Gv$ is closed and the stabilizer subgroup $G_v=\{g\in G, gv=v\}$ is finite.
\end{remark}

We associate  the parabolic subgroup $P(\lambda):=$ $\{g\in G, \lim_{t\to\infty}e^{-it\lambda}g e^{it\lambda} \ {\rm exists}\}$ to any $\lambda\in\kgot$.
Notice that, when $p\in P(\lambda)$, the limit $\lim_{t\to\infty}e^{-it\lambda} p e^{it\lambda}$ belongs to the stabilizer subgroup $G_\lambda=\{g\in G, g\lambda=\lambda\}$. 
Recall that $G=P(\lambda) K$, so that $\forall g\in G$, $\exists k\in K$ such that $gk^{-1}\in P(\lambda)$.

\begin{lemma}\label{lem:varphi-G-invariant} Let $\lambda\in\kgot$ and let $(g,k)\in G\times K$ such that $gk^{-1}\in P(\lambda)$. Then 
$\varpi_{\Phi}(gx,\lambda)=\varpi_{\Phi}(x,k^{-1}\lambda)$, \ $\forall x\in V\times M$.
\end{lemma}

\begin{proof}
  Let $p=gk^{-1}\in P(\lambda)$ and
  $g_\lambda:=\lim_{t\to\infty}e^{-it\lambda} p e^{it\lambda} \in
  G_\lambda$.  Since
  $e^{-it\lambda}gx=(e^{-it\lambda} p e^{it\lambda})e^{-it\lambda} kx$,
  we see that the following statements are equivalent: {\em (a)}
  $e^{-it\lambda}gx$ converge  when $t\to +\infty$, {\em (b)}
  $e^{-it\lambda}kx$ converge  when $t\to +\infty$ and {\em (c)}
  $e^{-itk^{-1}\lambda}x$ converge when $t\to +\infty$. So, we have first proved that $\varpi_{\Phi}(gx,\lambda)=-\infty$ if
  and only if $\varpi_{\Phi}(x,k^{-1}\lambda)=-\infty$. 
  
  Suppose now that $\varpi_{\Phi}(gx,\lambda)$ is finite. As 
  $(gm)_\lambda= g_\lambda (km)_\lambda$, we obtain
  \begin{align*}
    \varpi_{\Phi}(gx,\lambda)=\langle\Phi_{M}((gm)_\lambda),\lambda\rangle &\stackrel{(1)}{=}\langle\Phi_{M}((km)_\lambda),\lambda\rangle\\
                                                                           & \stackrel{(2)}{=}\langle\Phi_{M}(m_{k^{-1}\lambda}),k^{-1}\lambda\rangle=\varpi_{\Phi}(x,k^{-1}\lambda).
  \end{align*}
  The equality $(1)$ is due to the fact that
  $n\in M^\lambda\mapsto \langle\Phi_{M}(n),\lambda\rangle$ is
  constant on the connected component of $M^\lambda$ containing
  $(km)_\lambda$. And $(2)$ comes from the relation
  $\langle\Phi_{M}(e^{-it\lambda}km),\lambda\rangle=\langle\Phi_{M}(e^{-itk^{-1}\lambda}m),k^{-1}\lambda\rangle$.
\end{proof}

\begin{definition}
For any $x\in V\times M$, we consider 
$$
\mathbf{M}_{\Phi}(x)=\sup_{\lambda\neq 0}\frac{\varpi_{\Phi}(x,\lambda)}{\|\lambda\|}.
$$
\end{definition}

In Section~\ref{sec:M-phi-rel}, we verify that the function $\mathbf{M}_{\Phi}$ coincides with the function $\mathbf{M}_{rel}$ 
defined in Section~\ref{sec:HKKN-stratification}, when we work with the $G$-variety $V\times \Pbb E$.

Since $\varpi_{\Phi}(x,\lambda)\leq \langle\Phi(x),\lambda\rangle\leq \| \Phi(x)\|\, \|\lambda\|$, we see that 
$\mathbf{M}_{\Phi}(x)\leq \|\Phi(x)\|$. We have a more precise minoration.

\begin{proposition} \label{prop:M-G-invariant}
The function $\mathbf{M}_{\Phi}$ is $G$-invariant, and $\forall x\in V\times M$ we have 
\begin{equation}\label{eq:M-phi}
\mathbf{M}_{\Phi}(x)\leq \inf_{g\in G}\|\Phi(gx)\| = \|\Phi(x_\infty)\|.
\end{equation}
\end{proposition}

\begin{proof}
  Let $x\in V\times M$ and $g\in G$. From Lemma~\ref{lem:varphi-G-invariant}, it is immediate that for any
  $\lambda\in\kgot-\{0\}$, there exists $k \in K$ such that
  $\frac{\varpi_{\Phi}(gx,\lambda)}{\|\lambda\|}=\frac{\varpi_{\Phi}(x,k^{-1}
    \lambda)}{\| k^{-1}\lambda\|}\leq \mathbf{M}_{\Phi}(x)$.  This proves
  that $\mathbf{M}_{\Phi}(gx)\leq \mathbf{M}_{\Phi}(x)$ for any
  $(x,g)$. Hence, by symmetry we must have
  $\mathbf{M}_{\Phi}(gx)=\mathbf{M}_{\Phi}(x)$.

  Inequality (\ref{eq:M-phi}) is a consequence of the $G$-invariance
  of $\mathbf{M}_{\Phi}$, and the relation
  $\mathbf{M}_{\Phi}(x)\leq \|\Phi(x)\|$.
\end{proof}

\medskip

Thanks to the previous proposition, we see that $\mathbf{M}_{\Phi}(x)\leq 0$ if $x$ is $\Phi$-semistable. 
A. Teleman has proved that the converse statement is true when the K\"ahler Hamiltonian $G$-manifold 
is  ``energy complete'' \cite{Teleman04}. We have checked in Remark~\ref{rem:Phi-r}, that the manifold $V\times M$  
is energy complete when $M$ is compact. Hence, we have the following result.

\begin{theorem}[A. Teleman]\label{theo:teleman}
$x\in V\times M$ is $\Phi$-semistable if and only if $\mathbf{M}_{\Phi}(x)\leq 0$.
\end{theorem}

Thanks to Remark~\ref{rem:Phi-r}, we know that $\mathbf{M}_{\Phi_r}=\mathbf{M}_{\Phi}$ for any $r>0$. 
Theorem~\ref{theo:teleman} shows that 
$(V\times M)^{\Phi_r-ss}$ is equal to $(V\times M)^{\Phi-ss}$ for any $r>0$.

%%%%%%%%%%%%%%%%%%%%%%%%%%%%%%%%%%%%%%%%%%%%%%%%%%%%
\subsection{Optimal destabilizing vector}\label{sec:numerical invariant-phi}
%%%%%%%%%%%%%%%%%%%%%%%%%%%%%%%%%%%%%%%%%%%%%%%%%%%%

Let us recall the following important result that refines Theorem~\ref{theo:teleman}, and which is obtained by  
Bruasse-Teleman \cite{BT05} in the context of ``energy complete" K\"ahler Hamiltonian $G$-manifolds.
 
 \begin{theorem}[Bruasse-Teleman]\label{theo:maximal-destabilizing-1} Assume that $x\in V\times M$ is $\Phi$-unstable. 
 Then there exists a unique $\lambda_x\in \kgot-\{0\}$ such that 
 \begin{equation}\label{eq:lambda-x}
 \mathbf{M}_{\Phi}(x)=\|\lambda_x\|=\frac{\varpi_{\Phi}(x,\lambda_x)}{\|\lambda_x\|}.
 \end{equation}
\end{theorem}

In the compact algebraic framework, the existence of such $\lambda_x$ was conjectured by J. Tits and was proven by 
G. Kempf \cite{Kem78} and G. Rousseau \cite{Rou78}.  See Section~\ref{sec:HKKN-algebraic} for a proof in the non-compact algebraic framework.

If one uses Theorem~\ref{theo:teleman},  the {\em existence} in Theorem~\ref{theo:maximal-destabilizing-1} is the easy part: 
it follows from the fact that for any sequence 
$\lambda_n$ converging to $\lambda_\infty$, we have $\varpi_{\Phi}(x,\lambda_\infty)\geq \limsup \varpi_{\Phi}(x,\lambda_n)$. 
For the sake of completeness, 
we give a proof of Theorem~\ref{theo:maximal-destabilizing-1}  in the next sections.

\medskip

Our main contribution here is the following fact, which is proved in Section~\ref{sec:existence-lambda-x}. 

\begin{theorem}\label{theo:maximal-destabilizing-2} If $x\in V\times M$ is $\Phi$-unstable, then $\Phi(x_\infty)\in K\lambda_x$.
\end{theorem}

\medskip

The next lemma explains the behavior of the map $x\mapsto \lambda_x$ relatively to the $G$-action 
and the dilatations $\delta_r$.

\begin{lemma}\label{lem:lambda-x-G-action}
Let  $x\in V\times M$ be a $\Phi$-unstable element.
\begin{enumerate}
\item We have $\lambda_{gx}=k\lambda_x$ for any $(g,k)\in G\times K$ such that $k^{-1}g\in P({\lambda_x})$.
\item We have $\lambda_{\delta_r(x)}=\lambda_x$ for any $r>0$.
\end{enumerate}
\end{lemma}

\begin{proof}
  Let $(k,p)\in K\times P({\lambda_x})$ such that $g=kp$. Then $g=p'k$
  with $p'=kpk^{-1}\in P(k\lambda_x)$.  Thanks to Lemma
 ~\ref{lem:varphi-G-invariant}, we have
$$
\frac{\varpi_{\Phi}(gx,k\lambda_x)}{\|k\lambda_x\|}=\frac{\varpi_{\Phi}(p'kx,k\lambda_x)}{\|k\lambda_x\|}=\frac{\varpi_{\Phi}(x,\lambda_x)}{\|\lambda_x\|}={\bf
  M}_{\Phi}(x)=\mathbf{M}_{\Phi}(gx).
$$
This proves that $k\lambda_x$ is the optimal destabilizing vector of
$gx$.

The second point is a consequence of the relations
$\varpi_{\Phi}(x,\lambda)=\varpi_{\Phi}(\delta_r(x),\lambda), \forall
r>0$ (see Lemma~\ref{rem:Phi-r}).
\end{proof}

%%%%%%%%%%%%%%%%%%%%%%%%%%%%%%%%%%%%%%%%%%%%%
\subsubsection{Existence of an optimal destabilizing vector}\label{sec:existence-lambda-x}
%%%%%%%%%%%%%%%%%%%%%%%%%%%%%%%%%%%%%%%%%%%%%

The existence of an optimal destabilizing vector satisfying the condition of Theorem~\ref{theo:maximal-destabilizing-2} 
follows from a result essentially due
to Chen-Sun \cite[Theorems~4.4 and~4.5]{ChenSun}.  A detailed
proof of this ``Generalized Kempf Existence Theorem" is given in
\cite[Chapter 10]{GRS21} when the K\"ahler manifold is compact. 

\begin{theorem}\label{theo:chen-sun} Assume that $x\in V\times M$ is $\Phi$-unstable, so that $\Phi(x_\infty)\neq 0$.

Let $x(t)$ be the negative gradient flow of $f$ passing through $x$. Let $g:\R^{\geq 0}\to G$ be the solution of (\ref{eq:g(t)}) so that $x(t)=g^{-1}(t)x$. Defines the curves 
$\xi:\R^{\geq 0}\to \kgot$ and $u:\R^{\geq 0}\to K$ by $g(t)=:\exp(i\xi(t))u(t)$.

Then the limit
$$
\xi_\infty:=\lim_{t\to +\infty}\frac{\xi(t)}{t}
$$
exits and satisfies 
$$
\mathbf{M}_{\Phi}(x)=\frac{\varpi_{\Phi}(x,\xi_\infty)}{\|\xi_\infty\|}=\|\xi_\infty\|\quad {\rm and}\quad \Phi(x_\infty)\in K\xi_\infty.
$$
\end{theorem}

{\em Proof:} We take the steps of the proof given in \cite[Chapter~10]{GRS21} and adapt it to our non-compact framework. 
Let $\gamma(t)$ denote the image of $g(t)$ in $\mathbb{X}=G/K$:
$$
\gamma(t)=\pi(g(t))=\pi(\exp(i\xi(t))).
$$
The geodesic connecting $\gamma(0)$ and $\gamma(t)$ is given by 
$$
\gamma_t(r)=\pi\left(\exp(i\,r\,\tfrac{\xi(t)}{t})\right),\quad 0\leq r\leq t.
$$ 
For $t>0$, define the function $\rho_t :[0,t]\to \R$ by $\rho_t(r) := d_{\mathbb{X}}(\gamma_t(r),\gamma(r))$. 
In the first five steps of \cite[Chapter~10]{GRS21}, the authors show that 
there exist $C>0$ and $\epsilon\in ]0,1[$ such that
\begin{equation}\label{eq:five-step}
\rho_t(r)\leq C \, r^{1-\epsilon}\quad {\rm and}\quad \left| \frac{\xi(t)}{t}-\frac{\xi(t')}{t'}\right|\leq \frac{C}{t^\epsilon}, 
\end{equation}
for all $t'\geq t\geq 1$ and $t\geq r\geq 0$. Their proof of (\ref{eq:five-step}) does not require the manifold to be compact, 
but only use that the negative gradient flow $x(t)$ converge to 
$x_\infty$ with $\Phi(x_\infty)\neq 0$. So, at this stage, we have proved that the 
limit $\xi_\infty:=\lim_{t\to +\infty}\frac{\xi(t)}{t}$ exits.

Define the geodesic $\gamma_\infty: [0,+\infty[\to \mathbb{X}$ by 
$$
\gamma_\infty(r)=\pi\left(\exp(i\,r\,\xi_\infty)\right)=\lim_{t\to \infty} \gamma_t(r).
$$
Taking the limit $t\to \infty$ in (\ref{eq:five-step}), we obtain 
\begin{equation}\label{eq:distance-gamma-r}
d_{\mathbb{X}}(\gamma_\infty(r),\gamma(r))\leq C \, r^{1-\epsilon},\quad\forall r\geq 0.
\end{equation}
We can now prove the next steps 6 $\&$ 7. Since the function $\Psi_x: \mathbb{X}\to \R$ is not globally Lipschitz, 
we slightly modify the arguments of \cite{GRS21}.

\begin{lemma}\label{lem:psi-x-gamma}
There exists $C'>0$ such that $|\Psi_x(\gamma_\infty(r))|\leq C' r,\ \forall r\geq 1$.
\end{lemma} 

\begin{proof}
  The function $F_t(r):=\Psi_x(\gamma_t(r)), r\in [0,t]$ is convex,
  hence
$$
-r\langle \Phi(e^{-i\,\xi_t}x),\tfrac{\xi_t}{t}\rangle = r F'_t(t)\geq
F_t(r)-F_t(0)\geq r\,F'_t(0)=- r \langle
\Phi(x),\tfrac{\xi_t}{t}\rangle
$$
Since
$\|\Phi(e^{-i\,\xi_t}x)\|=\|\Phi(x(t))\| \leq \|\Phi(x)\|,\forall
t\geq 0$, we obtain
$\|\Psi_x(\gamma_t(r))\|\leq C' r, \forall 1\leq r\leq t$ with
$C'=\|\Phi(x)\| \sup_{t\geq 1}\|\tfrac{\xi_t}{t}\|$.  We prove our
lemma by taking the limit $t\to\infty$ in the previous inequality.
\end{proof}

\begin{lemma}
We have $\varpi_{\Phi}(x,\xi_\infty)=\|\Phi(x_\infty)\|^2$ and 
\begin{equation}\label{eq:M-xi-x-infty}
\mathbf{M}_{\Phi}(x)=\frac{\varpi_{\Phi}(x,\xi_\infty)}{\|\xi_\infty\|}=\|\xi_\infty\|=\|\Phi(x_\infty)\|.
\end{equation}
\end{lemma} 

\begin{proof}
  We have
  $\Psi_x(\gamma_\infty(r))=-\int_0^{r}\langle
  \Phi(e^{-i\,s\,\xi_\infty}x),\xi_\infty\rangle ds$, so
  $\lim_{r\to\infty}\langle
  \Phi(e^{-i\,r\,\xi_\infty}x),\xi_\infty\rangle$
  $=:\varpi_{\Phi}(x,\xi_\infty)$ is equal to
  $-\lim_{r\to\infty}\frac{\Psi_x(\gamma_\infty(r))}{r}$.

  First we compute
  $\lim_{r\to\infty}\frac{\Psi_x(\gamma(r))}{r}=-\lim_{r\to\infty}\frac{1}{r}\int_{0}^r
  \|\Phi(x(t))\|^2 = -\|\Phi(x_\infty)\|^2$. Next we show that
  \begin{equation}\label{eq:limite-0}
    \lim_{r\to\infty}\frac{\Psi_x(\gamma(r))-\Psi_x(\gamma_\infty(r))}{r}=0.
  \end{equation} 
  Let $x=(v,m)\in V\times M$. Since $\Psi_x=\Psi_v+\Psi_m$, with
  $\Psi_v$ defined in Example~\ref{example:psi-v}, we have
$$
\frac{\Psi_x(\gamma(r))-\Psi_x(\gamma_\infty(r))}{r}=\frac{\Psi_m(\gamma(r))-\Psi_m(\gamma_\infty(r))}{r}+\frac{\|g(r)^{-1}v\|^2-\|e^{-ir\xi_\infty}v\|^2}{4r}
$$
The Kempf–Ness function $\Psi_m$ is globally Lipschitz continuous with
Lipschitz constant $L := \sup_{g\in G} \|\Phi_M(gm)\|$. Hence, it
follows from (\ref{eq:distance-gamma-r}) that
$$
\left|\frac{\Psi_m(\gamma(r))-\Psi_m(\gamma_\infty(r))}{r}\right|\leq
\frac{C\, L}{r^\epsilon}, \quad {\rm for\ all}\quad r> 0.
$$

The curve $r\geq 0\mapsto g(r)^{-1}v\in V$ is bounded since the limit
$\lim_{r\to\infty}g(r)^{-1}v\in V$ exists. Thanks to Lemma
\ref{lem:psi-x-gamma}, we see that the curve
$r\geq 0\mapsto e^{-ir\xi_\infty}v\in V$ is also bounded (see Remark
\ref{lem:varphi-finite}). Finally, we have proved (\ref{eq:limite-0}),
and it shows that $\varpi_{\Phi}(x,\xi_\infty)=\|\Phi(x_\infty)\|^2$.

Thanks to Proposition~\ref{prop:M-G-invariant}, we have
$$
\frac{\|\Phi(x_\infty)\|^2}{\|\xi_\infty\|}=
\frac{\varpi_{\Phi}(x,\xi_\infty)}{\|\xi_\infty\|}\leq {\bf
  M}_{\Phi}(x)\leq \|\Phi(x_\infty)\|.
$$
The identities (\ref{eq:M-xi-x-infty}) will then follow from the
inequality $\|\xi_\infty\|\leq\|\Phi(x_\infty)\|$. By definition of
$\xi_t$, we have
\begin{eqnarray*}
  \left\Vert\frac{\xi_t}{t}\right\Vert &=& \frac{d_{\mathbb{X}}(\gamma(0),\gamma(t))}{t}\\
                                       &\leq & \frac{1}{t}\int_0^t \|\gamma'(s)\| ds= \frac{1}{t}\int_{0}^t \|\Phi(x(s))\| ds
\end{eqnarray*}
If we take the limit $t\to\infty$, we obtain
$\|\xi_\infty\|\leq\|\Phi(x_\infty)\|$.

It remains to prove that $\Phi(x_\infty)\in K\xi_\infty$. The proof
that is given in the final steps 8 $\&$ 9 of \cite[Chapter 10]{GRS21}
works here as they do not use the compactness of the manifold. 
\end{proof}

%%%%%%%%%%%%%%%%%%%%%%%%%%%%%%%%%%%%%%%%%%%%%%%%%%%%
\subsubsection{Uniqueness of the optimal destabilizing vector}\label{sec:unique-lambda-x}
%%%%%%%%%%%%%%%%%%%%%%%%%%%%%%%%%%%%%%%%%%%%%%%%%%%%

In the compact setting, the unicity of the optimal destabilizing vector can be proved in the projective framework 
using the convexity properties of Kempf-Ness functions (see \cite[Section 5.4]{Woodward11} or \cite[Theorem 10.2]{GRS21}). 
This type of proof could be adapted to our non-compact setting, but we prefer to give a proof close to that 
given by Kempf \cite{Kem78} in the compact algebraic setting (see Theorem~\ref{th:Kempf}). 

For any $\mu\in \ggot$, we denote by $[\mu]_\kgot$ its $\kgot$-part. Recall that the quadratic map 
$\delta:\ggot\to\R$ defined by the relation
$\delta(\mu):= \| [\mu]_\kgot\|^2-\|[i\mu]_\kgot\|^2$ is $G$-invariant. 

Here are some classical properties of the subset $\ggot_{ell}:=\{g\lambda;\  (g,\lambda)\in G\times \kgot\}$ 
formed by the elliptic elements of the Lie algebra $\ggot$. 
\begin{lemma}\label{lem:g-ell}
\begin{enumerate}
\item $\mu\in\ggot$ belongs to $\ggot_{ell}$ iff the subgroup $\overline{\{exp(t\mu), t\in\R\}}\subset G$ is compact.
\item If $\mu\in \ggot_{ell}$, we have $\delta(\mu)\geq 0$, and $\delta(\mu)=0$ iff $\mu=0$.
\item Let $\mu_1,\mu_2\in  \ggot_{ell}$ such that $[\mu_1,\mu_2]=0$. Then $\mu_1+\mu_2\in\ggot_{ell}$ and 
$\sqrt{\delta(\mu_1+\mu_2)}\geq \sqrt{\delta(\mu_1)}+\sqrt{\delta(\mu_2)}$ with equality iff $\mu_1=0$ or $\mu_2\in \R^{\geq 0}\mu_1$.
\item Let $\mu_1,\mu_2\in  \ggot_{ell}$. There are $p_i\in P({\mu_i})$ such that the elements $\mu_i':=p_i\mu_i$ satisfy $[\mu_1',\mu_2']=0$.
\end{enumerate}
\end{lemma}

\begin{proof}
  The first point follows from the fact that all maximal compact
  subgroups of $G$ are conjugate to $K$.  The second point is due to
  the $G$-invariance of $\delta$. Let $\mu_1,\mu_2\in \ggot_{ell}$
  such that $[\mu_1,\mu_2]=0$. Let $(g_1,\lambda_1)\in G\times\kgot$
  such that $\mu_1=g_1\lambda_1$. Then $\mu_2=g_1\tilde{\mu}_2$ with
  $\tilde{\mu}_2$ belonging to the Lie algebra $\ggot_{\lambda_1}$ of
  the reductive subgroup $G_{\lambda_1}$. Thanks to the first point,
  we see that
  $\tilde{\mu}_2\in \ggot_{\lambda_1}\cap
  \ggot_{ell}=(\ggot_{\lambda_1})_{ell}$: hence, there exists
  $(g_2,\lambda_2)\in G_{\lambda_1}\times\kgot_{\lambda_1}$ such that
  $\tilde{\mu}_2=g_2\lambda_2$. Finally, if one takes $g=g_1g_2$, we
  have $\mu_i=g\lambda_i$ and then
  $\mu_1+\mu_2=g(\lambda_1+\lambda_2)\in \ggot_{ell}$. The third point
  stems from this and the $G$-invariance of $\delta$. The intersection
  $P({\mu_1})\cap P({\mu_2})$ of the parabolic subgroups associated to
  $\mu_1,\mu_2\in \ggot_{ell}$ contains a maximal torus $T_1$ of
  $G$. Since all the maximal torus of a parabolic subgroup $P({\mu_i})$
  are conjugated, there exists $p_i\in P({\mu_i})$ such that
  $\mu_i':=p_i\mu_i\in {\rm Lie}(T_1)$. Hence, $[\mu_1',\mu_2']=0$. 
\end{proof}

\medskip

Following the approach in \cite[Chapter~5]{GRS21} (see also \cite[Section~2]{Teleman04}), we extend the map $\varpi_{\Phi}$ to a $G$-invariant map 
$\varpi_{\Phi}^\C: V\times M\times \ggot_{ell}\to \R \cup\{-\infty\}$. To do so, we need the following result.

\begin{lemma}\label{lem:trajectoire-mu} Let $(v,m,\mu)\in V\times M\times \ggot_{ell}$.
\begin{enumerate}
\item The trajectory $e^{-it\mu}m$ has a limit, denoted $m_\mu$, when $t\to+\infty$. 
\item Either $\lim_{t\to+\infty} \langle\Phi_V(e^{-it\mu}v),[\mu]_\kgot\rangle=0$ and 
the trajectory $e^{-it\mu}v$ has a limit when $t\to+\infty$, or $\lim_{t\to+\infty} \langle\Phi_V(e^{-it\mu}v),[\mu]_\kgot\rangle=-\infty$ 
(exponentially fast).
\end{enumerate}
\end{lemma}

\begin{proof}
  Let $(g,\lambda)\in G\times\kgot$ such that $\mu=g\lambda$. So the
  trajectory $e^{-it\mu}m= ge^{-it\lambda}g^{-1}m$ converges to
  $m_\mu:=g(g^{-1}m)_{\lambda}$ (see Section~\ref{sec:numerical-invariant-phi}).

  For every $v\in V$, we have\footnote{The decomposition $v=\sum_{a} v_a$ is not necessarily orthogonal.} $e^{-it\mu}v=\sum_{a\in\R} e^{ta} v_a$
  where $v_a\in\{w\in V, \mu \cdot w= i a w\}$.  So,
  $\langle\Phi_V(e^{-it\mu}v),[\mu]_\kgot\rangle= \frac{i}{2}
  \sum_{a,b} e^{t(a+b)} \langle [\mu]_\kgot\cdot v_a, v_b\rangle_V$. A
  direct computation gives that
  $\langle [\mu]_\kgot\cdot v_a, v_b\rangle_V= i\frac{a+b}{2}\langle
  v_a, v_b\rangle_V$. We finally obtain
  \begin{equation}\label{eq:lem-trajectoire-mu}
    \langle\Phi_V(e^{-it\mu}v),[\mu]_\kgot\rangle= \frac{-1}{4} \sum_{a,b} (a+b)e^{t(a+b)}\langle v_a, v_b\rangle_V.
  \end{equation}
  Let $a_0=\max\{a, v_a\neq 0\}$. Thanks to
  (\ref{eq:lem-trajectoire-mu}), we have two options:
  \begin{itemize}
  \item[-] If $a_0>0$, then $\exists c>0$ such that
    $\langle\Phi_V(e^{-it\mu}v),[\mu]_\kgot\rangle\underset{+\infty}{\sim}-c\,
    e^{2a_0t}$.
  \item[-] If $a_0\leq 0$, then
    $\lim_{t\to+\infty}
    \langle\Phi_V(e^{-it\mu}v),[\mu]_\kgot\rangle=0$ and the
    trajectory $e^{-it\mu}v$ has a limit when $t\to+\infty$.
  \end{itemize}
  The second point of Lemma~\ref{lem:trajectoire-mu} is settled.
\end{proof}

\medskip

We can now define the map
$$
\varpi_{\Phi}^\C(x,\mu)=\lim_{t\to+\infty}\langle\Phi(e^{-it\mu}x),[\mu]_\kgot\rangle \in \R\cup\{-\infty\},\qquad \forall (x,\mu)\in V\times M\times \ggot_{ell}.
$$
Let us summarize some of its properties. The next Lemma is the analytical analog of Lemma~\ref{lem:minv}.

\begin{lemma} \label{lem:varphi-C-property}Let $x\in V\times M$.
\begin{enumerate}
\item $\varpi_{\Phi}^\C(gx,g\mu)=\varpi_{\Phi}^\C(x,\mu)$ for every $g\in G$.
\item $\varpi_{\Phi}^\C(x,p\mu)=\varpi_{\Phi}^\C(x,\mu)$ if $p\in P(\mu)$.
\item $\varpi_{\Phi}^\C(x,\mu_1+\mu_2)\geq \varpi_{\Phi}^\C(x,\mu_1)+\varpi_{\Phi}^\C(x,\mu_2)$ if $[\mu_1,\mu_2]=0$.
\item $\varpi_{\Phi}^\C(x,-\mu)<0$ if $\varpi_{\Phi}^\C(x,\mu)>0$.
\end{enumerate}
\end{lemma}

\begin{proof}
  Identity (\ref{eq:lem-trajectoire-mu}) shows that
  $\lim_{t\to+\infty}
  \langle\Phi_V(e^{-itg\mu}gv),[g\mu]_\kgot\rangle=-\infty$ if and
  only if
  \break $\lim_{t\to+\infty}
  \langle\Phi_V(e^{-it\mu}v),[\mu]_\kgot\rangle=-\infty$. Hence,
  $\varpi_{\Phi}^\C(gx,g\mu)=-\infty$ if and only if
  $\varpi_{\Phi}^\C(x,\mu)=-\infty$.  Suppose now that
  $\varpi_{\Phi}^\C(x,\mu)$ and $\varpi_{\Phi}^\C(gx,g\mu)$ are
  finite: then we have
  $\varpi_{\Phi}^\C(x,\mu)=\langle\Phi_{M}(m_\mu),
  [\mu]_{\kgot}\rangle$ and
  $\varpi_{\Phi}^\C(gx,g\mu)=\langle\Phi_{M}((gm)_{g\mu}),
  [g\mu]_{\kgot}\rangle=\langle\Phi_{M}(gm_{\mu}),
  [g\mu]_{\kgot}\rangle$. The first point follows then from the fact
  that the map
  $g\in G\mapsto \langle\Phi_{M}(g m_0), [g\mu]_{\kgot}\rangle$ is
  constant for any $m_0\in M^\mu$ (see \cite[Lemma 5.8]{GRS21}). The second point is a consequence of the first point
  and Lemma~\ref{lem:varphi-G-invariant}.

  Let us prove the third point. Let $\mu_1,\mu_2\in \ggot_{ell}$ such
  that $[\mu_1,\mu_2]=0$. Thanks to the third point of Lemma~\ref{lem:g-ell}, there are $g\in G$ and
  $\lambda_1,\lambda_2\in\kgot$ such that $\mu_i=g\lambda_i$. 
Thus, we have to show that
  $\varpi_{\Phi}(y,\lambda_1+\lambda_2)\geq
  \varpi_{\Phi}(y,\lambda_1)+\varpi_{\Phi}(y,\lambda_2)$ for
  $y=g^{-1}x$. Consider the function
$$
F(s,t)=
\langle\Phi(e^{-is\lambda_1}e^{-it\lambda_2}y),\lambda_1+\lambda_2\rangle=\underbrace{\langle\Phi(e^{-is\lambda_1}e^{-it\lambda_2}y),\lambda_1\rangle}_{F_1(s,t)}+\underbrace{\langle\Phi(e^{-it\lambda_2}e^{-is\lambda_1}y),\lambda_2\rangle}_{F_2(s,t)}.
$$
Notice that
$e^{-is\lambda_1}e^{-it\lambda_2}=e^{-it\lambda_2}e^{-is\lambda_1}$,
since $[\lambda_1,\lambda_2]=0$. As the function $s\mapsto F_1(s,t)$
is non-increasing, we have
$$
F_1(s,t)\geq
\lim_{s\to+\infty}F_1(s,t)=\varpi_{\Phi}(e^{-it\lambda_2}y,\lambda_1)=\varpi_{\Phi}(y,\lambda_1),
$$
because $e^{-it\lambda_2}\in P({\lambda_1})$. Similarly, we have
$F_2(s,t)\geq
\lim_{t\to+\infty}F_2(s,t)=\varpi_{\Phi}(y,\lambda_2)$. So, we have
proved that
$F(t,t)=F_1(t,t)+ F_2(t,t)\geq
\varpi_{\Phi}(y,\lambda_1)+\varpi_{\Phi}(y,\lambda_2)$ for any
$t\in\R$. Since
$\lim_{t\to+\infty}F(t,t)=\varpi_{\Phi}(y,\lambda_1+\lambda_2)$, we
obtain the desired inequality.

For the last point, it is sufficient to do it for
$\mu=\lambda\in \kgot$. The function
$\langle\Phi(e^{-it\lambda}x),\lambda\rangle$ is non-increasing and by
taking its limits in $\pm\infty$, we obtain
$\lim_{t\to+\infty}\langle\Phi(e^{-it\lambda}x),\lambda\rangle=
\varpi_{\Phi}(x,\lambda)$ and
$\lim_{t\to-\infty}\langle\Phi(e^{-it\lambda}x),\lambda\rangle=-\varpi_{\Phi}(x,-\lambda)$. This
demonstrates that $\varpi_{\Phi}(x,-\lambda)<0$ if
$\varpi_{\Phi}(x,\lambda)>0$.
\end{proof}

\medskip

\begin{proposition}\label{prop:M-phi-etendu}
For any $x\in V\times M$, the quantity
$$
\mathbf{M}^{\C}_{\Phi}(x)=\sup_{\mu\in\ggot_{ell}-\{0\}}\frac{\varpi^\C_{\Phi}(x,\mu)}{\sqrt{\delta(\mu)}}.
$$
is equal to $\mathbf{M}_{\Phi}(x)$.
\end{proposition}

\begin{proof}
  We have $\mathbf{M}^{\C}_{\Phi}(x)\geq \mathbf{M}_{\Phi}(x)$, because the
  map
  $\mu\in\ggot_{ell}-\{0\}\mapsto
  \frac{\varpi^\C_{\Phi}(x,\mu)}{\sqrt{\delta(\mu)}}$ is an extension
  of the map
  $\lambda\in\kgot-\{0\}\mapsto
  \frac{\varpi_{\Phi}(x,\lambda)}{\|\lambda\|}$. On the other hand,
  every $\mu\in\ggot_{ell}$ can be written $\mu=g\lambda$ and
  $\frac{\varpi^\C_{\Phi}(x,\mu)}{\sqrt{\delta(\mu)}}=\frac{\varpi_{\Phi}(g^{-1}x,\lambda)}{\|\lambda\|}\leq
  \mathbf{M}_{\Phi}(g^{-1}x)=\mathbf{M}_{\Phi}(x)$.  This shows that
  $\mathbf{M}^{\C}_{\Phi}(x)\leq \mathbf{M}_{\Phi}(x)$.
\end{proof}

\bigskip

 We now have all the tools we need to show the uniqueness of the optimal 
destabilizing vector. Let us consider an unstable element $x\in V\times M$ and suppose that $\lambda_1,\lambda_2\in \kgot-\{0\}$ satisfy 
$$
\mathbf{M}_{\Phi}(x)=\frac{\varpi_{\Phi}(x,\lambda_1)}{\|\lambda_1\|}=\|\lambda_1\|=\|\lambda_2\|=\frac{\varpi_{\Phi}(x,\lambda_2)}{\|\lambda_2\|}.
$$

The last point of Lemma~\ref{lem:g-ell} tells us that there are $p_i\in P({\lambda_i})$ such that the elements $\mu_i:=p_i\lambda_i$ commute. 
Thanks to the second point of Lemma~\ref{lem:varphi-C-property}, we have 
$$
\mathbf{M}^\C_{\Phi}(x)=\frac{\varpi^\C_{\Phi}(x,\mu_1)}{\sqrt{\delta(\mu_1)}}=\frac{\varpi^\C_{\Phi}(x,\mu_2)}{\sqrt{\delta(\mu_2)}}.
$$
As $\varpi^\C_{\Phi}(x,\mu_1)>0$ and $\varpi^\C_{\Phi}(x,\mu_2)>0$ we have $\mu_1\neq -\mu_2$.  Since $[\mu_1,\mu_2]=0$, 
we have $\varpi_{\Phi}^\C(x,\mu_1+\mu_2)\geq \varpi_{\Phi}^\C(x,\mu_1)+\varpi_{\Phi}^\C(x,\mu_2)$ (see Lemma~\ref{lem:varphi-C-property}). 
We then obtain 
$$
\mathbf{M}^\C_{\Phi}(x)\geq\frac{\varpi^\C_{\Phi}(x,\mu_1+\mu_2)}{\sqrt{\delta(\mu_1+\mu_2)}}\geq \frac{\varpi^\C_{\Phi}(x,\mu_1)+
\varpi^\C_{\Phi}(x,\mu_2)}{\sqrt{\delta(\mu_1+\mu_2)}}
=\frac{\sqrt{\delta(\mu_1)}+\sqrt{\delta(\mu_2)}}{\sqrt{\delta(\mu_1+\mu_2)}}\mathbf{M}^\C_{\Phi}(x), 
$$
i.e. $\sqrt{\delta(\mu_1+\mu_2)}\geq \sqrt{\delta(\mu_1)}+\sqrt{\delta(\mu_2)}$. The third part of Lemma~\ref{lem:g-ell}, 
together with $\delta(\mu_1)=\delta(\mu_2)$, leads to $\mu_1=\mu_2$. The relation $p_1\lambda_1=p_2\lambda_2$ 
implies first that $P({\lambda_1})=P({\lambda_2})$, and then that $\exists p\in P(\lambda_1), \lambda_2=p\lambda_1$. 
Finally, this last relation gives that $\lambda_1=\lambda_2$ (see Theorem D.4 in \cite[Appendix~D]{GRS21}).

%%%%%%%%%%%%%%%%%%%%%%%%%%%%%%%%%%%%%%%%%%%%%%%%%%%%
\subsection{$\mathbf{M}_{\Phi}=\mathbf{M}_{rel}$ on $V\times \Pbb E$}\label{sec:M-phi-rel}
%%%%%%%%%%%%%%%%%%%%%%%%%%%%%%%%%%%%%%%%%%%%%%%%%%%%

We start with a first result.

\begin{proposition}
Let $\mu\in\ggot_{ell}$ be a non-zero rational vector. Then $\varpi^\C_{\Phi}(x,\mu)=\varpi_{rel}(x,\mu)$, $\forall x\in V\times \Pbb E$.
\end{proposition}

\begin{proof}Let $(\lambda,g)\in \kgot_\Q\times G$ such that $\mu=g\lambda$. Since $\varpi_{rel}(x,\mu)=\varpi_{rel}(g^{-1}x,\lambda)$ 
and $\varpi^\C_{\Phi}(x,\mu)=\varpi^\C_{\Phi}(g^{-1}x,\lambda)$, we need to prove that 
$\varpi_{rel}(g^{-1}x,\lambda)=\varpi^\C_{\Phi}(g^{-1}x,\lambda)$.

Let $(v,w)\in V\times E$ such that $g^{-1}x=(v,[w])$.
We have $e^{-it\lambda}v=\sum_{a}e^{ta}v_a$ where $\varpi_V(v,\lambda)=- \max\{a, v_a\neq 0\}$. 
Hence, the limit $\lim_{t\to\infty}e^{-it\lambda}v$ exists if and only if 
$\varpi_V(v,\lambda)\geq 0$. Thus, 
$\varpi_{rel}(g^{-1}x,\lambda)=\varpi^\C_{\Phi}(g^{-1}x,\lambda)=-\infty$ if $\varpi_V(v,\lambda)<0$.  

In the other case, if $\varpi_V(v,\lambda)\geq 0$, we have 
$$
\varpi^\C_{\Phi}(g^{-1}x,\lambda)=\lim_{t\to\infty}\langle\Phi_{\Pbb E }(e^{-it\lambda}[w]), \lambda\rangle= \varpi_E(w,\lambda)=\varpi_{rel}(g^{-1}x,\lambda).
$$
\end{proof}

The preceding result gives the following inequality:
\begin{equation}\label{eq:M-rel-phi}
\mathbf{M}_{\Phi}(x)=\mathbf{M}^\C_{\Phi}(x)=\sup_{\mu\in\ggot_{ell}-\{0\}}\frac{\varpi^\C_{\Phi}(x,\mu)}{\sqrt{\delta(\mu)}}\geq 
\sup_{\stackrel{\mu\in\ggot_{ell}-\{0\}}{\mu\ rational}}\frac{\varpi_{rel}(x,\mu)}{\sqrt{\delta(\mu)}}=: \mathbf{M}_{rel}(x).
\end{equation}

Thanks to Remarks~\ref{rem:G-stable} and \ref{rem:varpi-phi-infty}, we know already that, for $x=(v,[w])$, 
the following statements are equivalent: 
$\mathbf{M}_{\Phi}(x)=-\infty$, $\mathbf{M}_{rel}(x)=-\infty$, and $v\in V$ is $G$-{\em stable}.

\medskip

We consider now the case of $x\in V\times \Pbb E$ such that $\mathbf{M}_{\Phi}(x)\neq -\infty$. 
The equality $\mathbf{M}_{\Phi}(x)=\mathbf{M}_{rel}(x)$ is a consequence of (\ref{eq:M-rel-phi}) and the following lemma.

\begin{lemma}
Let $\mu\in\ggot_{ell}-\{0\}$ such that $\varpi^\C_{\Phi}(x,\mu)\neq -\infty$. There exists a sequence of {\em rational} elements
$\mu_j\in\ggot_{ell}-\{0\}$ such that $\mu=\lim_{j\to\infty}\mu_j$ and $\varpi^\C_{\Phi}(x,\mu)=\lim_{j\to\infty}\varpi_{rel}(x,\mu_j)$.
\end{lemma}

\begin{proof}Let $\tgot_0$ be a maximal abelian subalgebra of $\kgot$. Let $(\lambda,g)\in \tgot_0\times G$ such that $\mu=g\lambda$, and let 
$(v,w)\in V\times E$ such that $g^{-1}x=(v,[w])$. We have $\varpi_V(v,\lambda)\geq 0$, and 
$$
\varpi^\C_{\Phi}(x,\mu)=\varpi^\C_{\Phi}(g^{-1}x,\lambda)=\langle\Phi_{\Pbb E }([w]_\lambda),\lambda\rangle=\varpi_E(w,\lambda).
$$
There exists a finite set $\Rgot_v\subset\tgot^*_0$ such that $e^{-it\xi}v=\sum_{\alpha\in\Rgot_v}e^{t\langle\alpha,\xi\rangle}v_\alpha$, for all 
$(\xi,t)\in \tgot_0\times\R$. We see then that $\forall \xi_0\in\tgot_0$, $\varpi_V(v,\xi_0)\geq 0$ iff $\xi_0$ belongs to the rational cone 
$C_v:=\{\xi\in\tgot_0, \langle\alpha,\xi\rangle\leq 0,\forall \alpha\in\Rgot_v\}$. Consider now a sequence $\lambda_j\in C_v$ 
of rational elements converging to $\lambda$, and let $\mu_j=g\lambda_j$. We can conclude that the sequence 
$\varpi^\C_{\Phi}(x,\mu_j)=\langle\Phi_{\Pbb E }([w]_{\lambda_j}),\lambda_j\rangle=\varpi_E(w,\lambda_j)=\varpi_{rel}(x,\mu_j)$ converges to 
$\varpi^\C_{\Phi}(x,\mu)$, because the map  $\xi\in\tgot_0\mapsto \varpi_E(w,\xi)$ is continuous. 
\end{proof}

Given an unstable element $x\in  V\times\Pbb E$, we denote, in Section \ref{subsec:the-general-case}, 
by $\Lambda(x)\subset  \Xgot_{*}(G)_\Q\simeq \ggot_{ell,\Q}$ the set of optimal destabilizing $1$-parameter subgroups for $x$.

\begin{lemma}\label{lem:unicity-lambda-x}
If $x\in  V\times\Pbb E$ is unstable, $\Lambda(x)\cap \kgot$ is equal $\{\lambda_x\}$.
\end{lemma}
\begin{proof}By definition $\tau \in \Lambda(x)$ if and only if $\mathbf{M}_{rel}(x)=\Vert \tau\Vert=\frac{\varpi_{rel}(x,\tau)}{\Vert \tau\Vert}$. We 
have checked in Lemma \ref{lem:unicity-tau-x} that there exists a unique $\tau_x\in \kgot_\Q$ contained in $\Lambda(x)$. Since 
$\mathbf{M}_{rel}=\mathbf{M}_{\Phi}$, we see that $\tau_x$ satisfies $\mathbf{M}_{\Phi}(x)=\Vert \tau_x\Vert=\frac{\varpi_{\Phi}(x,\tau_x)}{\Vert \tau_x\Vert}$.
Thanks to Theorem \ref{theo:maximal-destabilizing-1} , we can conclude that $\tau_x$ is equal to $\lambda_x$.
\end{proof}

%%%%%%%%%%%%%%%%%%%%%%%%%%%%%%%%%%%%%%%%%%%%%%%%%%%%
\section{HKKN-stratifications in the K\"{a}hler framework}\label{sec:KN-stratification}
%%%%%%%%%%%%%%%%%%%%%%%%%%%%%%%%%%%%%%%%%%%%%%%%%%%%

Consider the K\"ahler Hamiltonian $G$-manifold $V\times M$ where $M$ is compact. 

%%%%%%%%%%%%%%%%%%%%%%%%%%%%%%
\subsection{Critical set}
%%%%%%%%%%%%%%%%%%%%%%%%%%%%%%

The critical set $\mathrm{Crit}(f)$ of $f=\frac{1}{2}\|\Phi\|^2$ is characterized by the relation : $x\in\mathrm{Crit}(f) \Longleftrightarrow \Phi(x)\cdot x=0$.
Let $\tgot^*_{0,+}\subset \tgot^*_0$ be a Weyl chamber.

\begin{definition}
Let $\Bcal:=\Phi(\mathrm{Crit}(f))\cap \tgot_{0,+}^*$, so that 
$$
\mathrm{Crit}(f)=\bigcup_{\beta\in\Bcal} \Ccal_\beta\quad {\rm with}\quad \Ccal_\beta:=K\Big((V^{\beta}\times M^\beta)\cap\Phi^{-1}(\beta)\Big).
$$
\end{definition}

\begin{lemma}\label{lem:crit-f}
The set $\Bcal$ is finite.
\end{lemma}

\begin{proof}
  If $\hgot\subset \kgot$ is a subalgebra, we define the locally
  closed submanifolds
  $(V\times M)_\hgot:=\{x\in V \times M, \kgot_x=\hgot\}$ and
  $(V\times M)_{[\hgot]}=K(V\times M)_\hgot$. There exists a finite
  number of subalgebra $\hgot_1,\ldots,\hgot_r$ such that
  $V\times M=\bigcup_{i=1}^r(V\times M)_{[\hgot_i]}$ (see \cite{DK12}). Let us analyze
  the set
$$
\Phi\left(\mathrm{Crit}(f)\bigcap (V\times M)_{[\hgot_i]}\right)=
K\Phi\left(\mathrm{Crit}(f)\bigcap (V\times M)_{\hgot_i}\right).
$$
Take a decomposition of the moment map
$\Phi=\Phi_{\hgot_i}+\Phi_{\qgot_i}$ relative to an orthogonal
splitting $\kgot=\hgot_i\oplus \qgot_i$. For any connected component
$\Zcal\subset (V\times M)_{\hgot_i}$, the function
$\Phi_{\hgot_i}\vert_{\Zcal}$ is constant, equal to
$\lambda_{i,\Zcal}$, and
$\mathrm{Crit}(f)\bigcap \Zcal \neq \emptyset$ iff
$\Phi_{\qgot_i}(z)=0, z\in\Zcal$ admits a solution. We have proven
that
$$
K\Phi\left(\mathrm{Crit}(f)\bigcap (V\times M)_{\hgot_i}\right)=\bigcup_{\Zcal} K\lambda_{i,\Zcal}
$$
where the union is on the connected components
$\Zcal\subset (V\times M)_{\hgot_i}$ such that
$\Phi_{\qgot_i}(z)=0, z\in\Zcal$ admits a solution. This proves that
$\Bcal$ is finite.
\end{proof}

%%%%%%%%%%%%%%%%%%%%%%%%%%%%%%
\subsection{Strata $\Scal_{\langle\beta\rangle}$ when $\beta\neq 0$}
%%%%%%%%%%%%%%%%%%%%%%%%%%%%%%

Let $\beta\in\Bcal-\{0\}$. In this section, we study the stratum
\begin{equation}\label{def:S-beta-phi}
\Scal_{\langle\beta\rangle}:=\{x\in V\times M, x_\infty\in \Ccal_\beta\}.
\end{equation}

First, we notice that $x\in  \Scal_{\langle\beta\rangle} \Longleftrightarrow \Phi(x_\infty)\in K\beta \Longleftrightarrow \lambda_x \in K\beta$.

\begin{remark}When working with the variety $V\times \Pbb E$, the condition $\lambda_x \in K\beta$ is equivalent to 
$\mathbf{M}_{rel}(x)=\Vert\beta\Vert$ and $\langle\beta\rangle\cap \Lambda(x)\neq\emptyset$ (see Lemma \ref{lem:unicity-lambda-x}).
This shows that the strata $\Scal_{\langle\beta\rangle}$ defined in (\ref{def:S-beta-phi}) is equal to the strata 
$(V\times \Pbb E)_{\langle\beta\rangle}$ defined in (\ref{eq:defSdtau}).
\end{remark}

Thanks to Lemma~\ref{lem:lambda-x-G-action}, we have a first result.

\begin{lemma}
The stratum $\Scal_{\langle\beta\rangle}$ is invariant under the $G$-action, and under the dilatations $\delta_r(v,m)=(\sqrt{r}v,m)$.
\end{lemma}

The next result precises the definition of $\Scal_{\langle\beta\rangle}$ in terms of
the moment map (see \cite[Section~6]{Kir84a} for the compact setting). 

\begin{proposition}\label{prop:S-beta}
A point $x\in V\times M$ lies in $\Scal_{\langle\beta\rangle}$ iff $\beta$ is the unique closest point to $0$ of 
$\Phi(\overline{Gx})\cap \tgot_{0,+}^*$.
\end{proposition}

\begin{remark}
In Section~\ref{sec:convexity}, we show that $\Delta(\overline{Gx}):=\Phi(\overline{Gx})\cap \tgot^*_{0,+}$ is a closed convex set. 
Knowing this, we see that $x\in V\times M$ is contained in $\Scal_{\langle\beta\rangle}$ 
iff  $\beta$ is the orthogonal projection of $0$ onto $\Delta(\overline{Gx})$.
\end{remark}

\medskip

The proof of Proposition~\ref{prop:S-beta},  which works like in the compact framework, is divided into several steps.

\medskip

\begin{lemma}\label{lem:S-beta-1}
Let $x\in V\times M$ and $\beta=\Phi(x)$.
\begin{enumerate}
\item Let $\xi\in \kgot$, such that $\Phi(e^{i\xi}x)=\beta$. Then $\xi\cdot x=0$.
\item Let $g\in G_\beta$ such that $\Phi(gx)=\beta$. Then $gx\in K_\beta x$.
\item If $x\in {\rm Crit}(f)$, we have : $\forall g\in G$ 
\begin{itemize}
\item $\|\Phi(gx)\|\geq \|\beta\|$, 
\item $\|\Phi(gx)\|= \|\beta\|$ if and only if $gx\in K x$.
\end{itemize}
\end{enumerate}
\end{lemma}

\begin{proof}
  The derivative of the function
  $h(t):=\langle\Phi(e^{it \xi} x),\xi\rangle$ is
  $h'(t):=\|\xi\cdot(e^{it \xi} x)\|^2$. Hence, the relation
  $\Phi(e^{i\xi}x)=\Phi(x)$ implies that $h(1)=h(0)$, and then
  $h'(t)=0,\forall t\in [0,1]$. The first point follows.

  For the second point, we write $g=k e^{i\xi}$ with $k\in K_\beta$
  and $\xi\in\kgot_\beta$. The identity $\Phi(gm)=\beta$ gives
  $\Phi(e^{i\xi}m)=\beta$, and we see that $gm=km$ thanks to the first
  point.

  When $\beta=0$, the third point is a consequence of the second
  point. Suppose now that $\beta\neq 0$, and write $g=k p$ with
  $k\in K$ and $p\in P(\beta)$. The function
  $f(t):=\langle\Phi(e^{-it\beta} px),\beta\rangle$ is non-increasing
  since $f'(t)=-\| \beta\cdot(e^{-it\beta}px)\|^2$. As
  $\beta\cdot x=0$, the curve
  $e^{-it\beta}px=e^{-it\beta}pe^{it\beta}x$ tends to $g_0 x$ when
  $t\to+\infty$, with
  $g_0=\lim_{t\to\infty}e^{-it\beta}pe^{it\beta}\in G_\beta$.  We have
  \begin{equation}\label{eq:lem-S-beta-1}
    \|\Phi(gx)\|=\|\Phi(px)\|\geq \frac{\langle\Phi(px),\beta\rangle}{\|\beta\|}= 
    \frac{f(0)}{\|\beta\|}\geq \frac{\lim_{t\to\infty}f(t)}{\|\beta\|}= \frac{\langle\Phi(g_0x),\beta\rangle}{\|\beta\|}\stackrel{(1)}{=}\|\beta\|.
  \end{equation}
  The equality $(1)$ is due to the fact that
  $y\in (V\times M)^\beta\mapsto \langle\Phi(y),\beta\rangle$ is
  constant on the connected component of $(V\times M)^\beta$
  containing $x$.

  Thanks to (\ref{eq:lem-S-beta-1}), we have
  $\|\Phi(gx)\|\geq \|\beta\|$ with equality iff $(i)$ $f'(t)=0$,
  $\forall t\geq 0$, and $(ii)$ $\Phi(px)=\beta$. Condition $(i)$
  means that $e^{-it\beta}px=px$, $\forall t\geq 0$, i.e. $px=g_0x$.
  With condition $(ii)$, we can conclude that $px=g_0x\in K_\beta x$,
  thanks to the second point. We have proved that
  $\|\Phi(gx)\|= \|\beta\|$ if and only if $gx\in K x$.
\end{proof}

\medskip

\begin{lemma}\label{lem:S-beta-2}
If $x,y\in {\rm Crit}(f)$ such that $y\in G x$. Then $y\in K x$.
\end{lemma}

\begin{proof}
  If $x,y\in {\rm Crit}(f)$ with $\Ocal=Gx= G y$, we have $x=x_\infty$
  and $y=y_\infty$: then
  $\|\Phi(x)\|=\|\Phi(y)\|=\inf_{z\in\Ocal}\|\Phi(z)\|$ (see
  Proposition~\ref{prop:phi-minimal-orbit}). Let $g\in G$ such that
  $y=gx$.  The identity $\|\Phi(x)\|=\|\Phi(gx)\|$ implies that
  $gx\in Kx$ (see Lemma~\ref{lem:S-beta-1}).
\end{proof}

\medskip

\begin{lemma}\label{lem:S-beta-3}
For any $y\in Gx$, we have $y_\infty\in K x_\infty$. 
\end{lemma}

\begin{proof}
  Let $g_0\in G$ such that $y=g_0^{-1}x$. The negative gradient flow
  lines of $f$ through $y$ and $x$ are respectively $y(t)=g(t)^{-1}y$
  and $x(t)=h(t)^{-1}x$.  The curves
  $\theta_0,\theta_1:\R^{\geq 0}\to \mathbb{X}$ defined by the
  relation $\theta_0(t)=\pi(g_0g(t))$ and $\theta_1(t)=\pi(h(t))$ are
  negative gradient flow lines of the Kempf-Ness function $\Psi_x$.

  Let $\eta(t)\in \kgot$ and $k(t)\in K$ be the curves such that
  $g_0 g(t)=: h(t)\exp(i\eta(t))k(t)$. Then
  $\|\eta(t)\|=d_{\mathbb{X}}(\theta_0(t),\theta_1(t))$ where
  $d_{\mathbb{X}}$ is the Riemannian distance on $\mathbb{X}$. The
  fourth point of Theorem~\ref{theo:kempf-ness-kahler} says that
  $t\mapsto \|\eta(t)\|$ is non-increasing. Thus, there exists a
  sequence of positive real number $(t_n)$ such that
  $\lim_{n\to \infty} t_n=+\infty$ and
  $\lim_{n\to \infty}\exp(i\eta(t_n))k(t_n)= g_1\in G$. Since we have
  $x(t)=\exp(i\eta(t))k(t) y(t)$, we obtain $x_\infty=g_1 y_\infty$,
  i.e. $x_\infty\in G y_\infty$.  As $x_\infty$ and $y_\infty$ belongs
  to ${\rm Crit}(f)$, we can conclude that $x_\infty\in K y_\infty$
  thanks to the previous lemma.
\end{proof}

\medskip

The next lemma completes the proof of Proposition~\ref{prop:S-beta}.

\begin{lemma}\label{lem:S-beta-4}
Let $y\in\overline{Gx}$ such that $\|\Phi(y)\|=\inf_{g\in G}\|\Phi(gx)\|$. Then $y\in K x_\infty$. 
\end{lemma}

\begin{proof}
  Let $a=\inf_{g\in G}f(gx)$, where $f=\frac{1}{2}\|\Phi\|^2$.  For
  any $\epsilon>0$, we consider the set
  $\Kcal_\epsilon:=\overline{Gx}\cap \{f\leq a+\epsilon\}$ which is
  compact by Proposition~\ref{prop:phi-proper-orbit}. Since $f$ has a
  finite number of critical values (see Lemma~\ref{lem:crit-f}), there
  exists $\epsilon_o>0$, such that
$$
\nabla f(z)=0\Longrightarrow f(z)=a,\quad \forall z\in
\Kcal_{\epsilon_o}.
$$
By the Lojasiewicz gradient inequality there exist constants $C>0$,
and $r\in ]0,1[$, such that $C\|\nabla f(z)\|\geq |f(z)-a|^{r}$ for
any $z\in \Kcal_{\epsilon_o}$. Denote by $d$ the distance on $V\times M$
defined by the Riemannian metric. The former inequality implies that
$d(z,z_\infty)\leq \frac{C}{1-r}|f(z)-a|^{1-r}$ for any
$z\in \Kcal_{\epsilon_o}$.

We consider now $y\in \overline{Gx}$ such that $f(y)=a$. Let $(z^k)$
be a sequence in $Gx\cap \Kcal_{\epsilon_o}$ that converges to $y$.
Thanks to Lemma~\ref{lem:S-beta-3}, we know that the sequence
$(z^k_{\infty})$ belongs to the orbit $Kx_\infty$. Finally,
$$
d(z^{k}_\infty, y)\leq d(z^{k}_\infty, z^k)+ d(z^{k}, y)\leq
\tfrac{C}{1-r}|f(z^k)-a|^{1-r} + d(z^k,y).
$$
We see that $\lim_{k\to\infty}d(z^{k}_\infty, y)=0$, i.e. the sequence
$(z^{k}_\infty)$ converges to $y$. This proves that $y\in K x_\infty$.
\end{proof}

%%%%%%%%%%%%%%%%%%%%%%%%%%%%%%%%%%%%%%%%%%%%%%%%%%%%
\subsection{Substrata $\Scal_{\beta}$}\label{sec:decomposition-Y-beta}
%%%%%%%%%%%%%%%%%%%%%%%%%%%%%%%%%%%%%%%%%%%%%%%%%%%%

For $\beta\in\Bcal-\{0\}$, we define
$$
\Scal_{\beta}:=\Big\{x \in V \times M, \  \lambda_x=\beta\Big\}\ \subset \Scal_{\langle\beta\rangle}.
$$

If $x\in \Scal_{\beta}$, then $\varpi_{\Phi}(x,\beta) := \lim_{t\to+\infty} \langle\Phi(e^{-it\beta} x),\beta\rangle$ is equal to $\|\beta\|^2$.
It implies that the limit $x':=\lim_{t\to+\infty} e^{-it\beta} x$ exists and belongs to $V^\beta \times M^\beta$ (see Lemma~\ref{lem:varphi-finite}). 

We see $V^\beta \times M^\beta$ as a K\"ahler Hamiltonian $G_\beta$-manifold with moment map $\Phi_\beta=\Phi-\beta: V^\beta \times M^\beta\to\kgot_\beta^*$. 

\medskip

The following result is the key point to understand the structure of the stratum $\Scal_{\langle\beta\rangle}$ when $\beta\neq 0$. It is due to 
Kirwan \cite{Kir84a} and Ness \cite{Ness84} in the projective case (see also the work of Hesselink \cite{Hes78,Hes79}, and Ramanan-Ramanathan \cite{RR84}).

\begin{theorem}\label{theo:RR-structure} Let $x\in V \times M$ and $\beta\in\Bcal-\{0\}$. Suppose that the limit $x':=\lim_{t\to+\infty} e^{-it\beta} x$ exists in $V^\beta \times M^\beta$. 
The following statements are equivalent:
\begin{enumerate}
\item \label{prop:RR-structure-ass1} 
$\beta$ is the unique closest point to $0$ of $\Phi(\overline{G_\beta x'})\cap \tgot_{0,+}^*$.
\item \label{prop:RR-structure-ass2} 
$x'$ is $\Phi_\beta$-semistable.
\item \label{prop:RR-structure-ass3} 
$\lambda_x=\beta$.
\end{enumerate}
All these conditions imply that $\lambda_{x'}=\beta$.
\end{theorem}

In the sequence, we denote by $Z\subset V^\beta\times M^\beta$ the connected component containing $x'$. The function 
$z\in Z\mapsto \langle \Phi(z),\beta\rangle$ is constant, equal to $\langle \Phi(x'),\beta\rangle=\varpi_{\Phi}(x',\beta)=\varpi_{\Phi}(x,\beta)$.
\begin{proof}
Let us start by proving that \ref{prop:RR-structure-ass1}
$\Longrightarrow$ \ref{prop:RR-structure-ass2}.

Assuming~\ref{prop:RR-structure-ass1}, there is a sequence $\ell_n\in G_\beta$ such that
  $\beta=\lim_{t\to+\infty}\Phi(\ell_n x')$. Thus, \break
  $\lim_{t\to+\infty}\Phi_\beta(\ell_n x')=0$, i.e. $x'$ is
  $\Phi_\beta$-semistable.

\bigskip 
  Let us prove now that \ref{prop:RR-structure-ass2} $\Longrightarrow$
  \ref{prop:RR-structure-ass3}  and
    \ref{prop:RR-structure-ass1}.

  Assuming~\ref{prop:RR-structure-ass2}, $\beta$ belongs to
  $\Phi(\overline{G_\beta x'})\subset \Phi(Z)$. This implies that
  $\langle \Phi(z),\beta\rangle=\|\beta\|^2, \forall z\in Z$.  Hence,
$$
\|\beta\|=\frac{\varpi_{\Phi}(x,\beta)}{\|\beta\|}\leq {\bf
  M}_\Phi(x)\leq \inf_{g\in
  G}\|\Phi(gx)\|\stackrel{(1)}{\leq}\inf_{g\in G}\|\Phi(gx')\|
\stackrel{(2)}{\leq} \inf_{\ell\in G_\beta}\|\Phi(\ell x')\|\leq
\|\beta\|.
$$
The inequalities $(1)$ and $(2)$ follow from the inclusions
$\overline{G_\beta x'}\subset \overline{G x'}\subset
\overline{Gx}$. We obtain that
\begin{itemize}
\item
  $\frac{\varpi_{\Phi}(x,\beta)}{\|\beta\|}=\|\beta\|={\bf
    M}_\Phi(x)$. Hence, $\lambda_x=\beta$.
\item
  $\frac{\varpi_{\Phi}(x',\beta)}{\|\beta\|}=\|\beta\|=\inf_{g\in
    G}\|\Phi(gx')\|=\mathbf{M}_\Phi(x')$. Hence, $\lambda_{x'}=\beta$.
\end{itemize}

The last point shows that $x'\in \Scal_{\langle\beta\rangle}$, and so, by Proposition
\ref{prop:S-beta}, $\beta$ is the unique closest point to $0$ of
$\Phi(\overline{Gx'})\cap \tgot_{0,+}^*$. Since
$\beta\in \Phi(\overline{G_\beta x'})\subset \Phi(\overline{Gx'})$,
condition~\ref{prop:RR-structure-ass1} holds.

\bigskip
Finally, we prove that \ref{prop:RR-structure-ass3} $\Longrightarrow$ \ref{prop:RR-structure-ass2}. 

Suppose that $\lambda_x=\beta$. As $\Phi(x_\infty)\in K\beta$ there
exists a sequence $g_n\in G$ such that $\Phi(g_nx)$ converge to
$\beta$ as $n\to\infty$.  We write $g_n=k_n p_n$ with $k_n\in K$ and
$p_n\in P(\beta)$. We can assume that $(k_n)$ converges to $k_\infty$,
so that $\lim_{n\to\infty}\Phi(p_n x)=k_\infty^{-1}\beta$.

\begin{lemma}
We have $k_\infty^{-1}\beta=\beta$.
\end{lemma}

\begin{proof}
  $\lim_{t\to\infty}e^{-it\beta}p_n x=\ell_n x'$ where
  $\ell_n:=\lim_{t\to\infty}e^{-it\beta}p_n e^{it\beta} \in
  G_\beta$. For all $t\geq 0$, we have
$$
\langle\Phi(p_n x),\beta\rangle\geq \langle\Phi(e^{-it\beta}p_n
x),\beta\rangle\geq \langle\Phi(\ell_n x'),\beta\rangle=
\langle\Phi(x'),\beta\rangle=\varpi_{\Phi}(x,\beta)=\|\beta\|^2.
$$
If we take the limit $n\to \infty$, we obtain
$\|\beta\|^2\geq \langle k_\infty^{-1}\beta,\beta\rangle\geq
\|\beta\|^2$. This shows that $k_\infty^{-1}\beta=\beta$.
\end{proof}

\medskip

We consider the K\"ahler Hamiltonian $G$-manifold
$N:= V\times M \times (K\beta)^\circ$ with moment map $\Phi_N$. Here,
$(K\beta)^\circ$ denotes the adjoint orbit $K\beta$ equipped with the
opposite symplectic structure: the complex structure is defined
through the identifications
$(K\beta)^\circ \simeq K/K_{\beta}\simeq G/P({-\beta})$.

We have shown that $\lim_{n\to\infty} \Phi_N(p_nx,\beta)=0$: this
implies that, for $n_o$ large enough, $\widehat{x}:=(p_{n_o} x,\beta)$
belongs to the open subset $N^{\Phi-ss}$. Let
$\widehat{x}'=\lim_{t\to+\infty}e^{-it\beta}\widehat{x} :=(\ell_{n_o}
x',\beta)\in N^\beta$. Consider the Kemp-Ness functions
$\Psi_{e^{-it\beta} \widehat{x}}$. We have
$$
\Psi_{e^{-it\beta} \widehat{x}}(g)=\Psi_{ \widehat{x}}(e^{it\beta}g)
-\Psi_{\widehat{x}}(e^{it\beta})=\Psi_{
  \widehat{x}}(e^{it\beta}g)+\int_{0}^t \langle \Phi_N(e^{-is\beta}
\widehat{x}),\beta\rangle ds.
$$
See Proposition~\ref{prop:psi-n}. As $\widehat{x}\in N^{\Phi-ss}$,
there exists $c\in \R$ such that
$\Psi_{ \widehat{x}}(e^{it\beta}g)\geq c$,
$\forall (t,g)\in \R\times G$. Since
$$
\lim_{t\to+\infty} \langle\Phi_N(e^{-is\beta}
\widehat{x}),\beta\rangle=\langle\Phi_N(
\widehat{x}'),\beta\rangle=\langle\Phi(l_{n_o}x'),\beta\rangle-\|\beta\|^2=0
$$
the integral
$\int_{0}^t \langle \Phi_N(e^{-is\beta} \widehat{x}),\beta\rangle ds$
is nonnegative $\forall t\geq 0$. We have proven that
$\Psi_{e^{-it\beta} \widehat{x}}(g)\geq c$,
$\forall (t,g)\in \R^{\geq 0}\times G$, so by taking the limit
$t\to \infty$, we obtain that $\Psi_{\widehat{x}'}$ is bounded from
below. Hence, $\widehat{x}'\in N^{\Phi-ss}\cap N^\beta$: this means
that there exists a sequence $h_n\in G_\beta$ such that
$\lim_{n\to\infty} \Phi_N(h_n\widehat{x}')=0$ (see Remark
\ref{rem:g(t)-lambda}).  We have proven that
$\lim_{n\to\infty} \Phi(h_n\ell_{n_o}x')=\beta$, i.e. $x'$ is
$\Phi_\beta$-semistable.
\end{proof}

\begin{remark}
In the compact K\"ahler setting, Theorem~\ref{theo:RR-structure} seems to be well known among experts. However, we can only find proof for this case in 
Theorem 7.2.2 and Theorem 7.1.7 of \cite{Woodward11}. Unfortunately, the arguments presented there are approximate and difficult to follow.
 
In any case, our method, which differs from the one proposed in \cite{Woodward11}, also works in the compact setting.
\end{remark}

%%%%%%%%%%%%%%%%%%%%%%%%%%%%%%%%%%%%%%%%%%%%%%%%%%%%
\subsection{Decomposition of the stratum $\Scal_{\langle\beta\rangle}$ when $\beta\neq 0$}\label{sec:decomposition-S-beta}
%%%%%%%%%%%%%%%%%%%%%%%%%%%%%%%%%%%%%%%%%%%%%%%%%%%%

Thanks to Lemma~\ref{lem:lambda-x-G-action}, the substratum $\Scal_{\beta}$ is stable under the action of the parabolic subgroup $P(\beta)$, and 
$g \Scal_{\beta} \cap \Scal_{\beta}\neq \emptyset$ only if $g\in P(\beta)$.

Since $G\Scal_{\beta}=\Scal_{\langle\beta\rangle}$, we see that the map $G\times \Scal_{\beta}\to \Scal_{\langle\beta\rangle}, (g,x)\mapsto gx$  induces a bijection
$$
G\times_{P({\beta})} \Scal_{\beta}\simeq \Scal_{\langle\beta\rangle}.
$$

Let $Z_\beta$ be the union of the connected components of $V^\beta\times M^\beta$ that intersect $\Phi^{-1}(\beta)$: 
it is a K\"ahler Hamiltonian $G_\beta$-manifold with moment map $\Phi_\beta=\Phi-\beta$.
We define the Bia\l ynicki-Birula's complex submanifold
$$
Y_\beta:=\Big\{x\in V\times M, \ \lim_{t\to+\infty} e^{-it\beta} x\in Z_\beta\Big\}.
$$
The projection $\pi_\beta: Y_\beta\to Z_\beta$, that sends $x\in Y_\beta$ to $\lim_{t\to+\infty} e^{-it\beta} x\in Z_\beta$, is an holomorphic map.

Let $Z_\beta^{^{ss}}\subset Z_\beta$ be the open subset formed by the $\Phi_\beta$-semistable elements. 
Theorem~\ref{theo:RR-structure} tells us that $\Scal_{\beta}$ is the open subset of 
$Y_\beta$ equal to the pullback $\pi_\beta^{-1}(Z_\beta^{^{ss}})$.

Finally, $G\times_{P(\beta)} \Scal_{\beta}$ has a canonical structure of a complex manifold, and a direct computation 
shows that the map $G\times_{P({\beta})} \Scal_{\beta}\to M, [g,x]\mapsto gx$ is an holomorphic embedding (see \cite[Section~6]{Kir84a}).  
These facts show that $\Scal_{\langle\beta\rangle}$ is a locally closed complex  submanifold of $V\times M$.

%%%%%%%%%%%%%%%%%%%%%%%%%%%%%%%%%%%%%%%%%%%%%%%%%%%%
\subsection{HKKN stratifications}\label{sec:HKKN-stratification}
%%%%%%%%%%%%%%%%%%%%%%%%%%%%%%%%%%%%%%%%%%%%%%%%%%%%

For any $x\in V\times M$, the element $x_\infty$ belongs to the critical set $\mathrm{Crit}(f)=\bigcup_{\beta\in\Bcal} \Ccal_\beta$ 
with $\Ccal_\beta:=K((V^{\beta}\times M^\beta)\cap\Phi^{-1}(\beta))$. Notice that $x_\infty\in \Ccal_\beta$ if and only if $x_\infty\in \Phi^{-1}(K\beta)$.

The aim of this section is the study of the HKKN stratification:
\begin{equation}\label{eq:stratification}
V\times M:=\bigcup_{\beta\in\Bcal} \Scal_{\langle\beta\rangle}\quad {\rm where}\quad \Scal_{\langle\beta\rangle}:=\{x\in V\times M, x_\infty\in \Phi^{-1}(K\beta)\}.
\end{equation}

%%%%%%%%%%%%%%%%%%%%%%%%%%%%%%%%%%%%%%%%%%%%%%%%%%%%
\subsubsection{HKKN stratifications of $V\times M$}\label{sec:HKKN-stratification-0}
%%%%%%%%%%%%%%%%%%%%%%%%%%%%%%%%%%%%%%%%%%%%%%%%%%%%

The stratum $\Scal_{\langle 0\rangle}$, which is the set of $\Phi$-semistable elements, is also denoted by \break
$(V\times M)^{\Phi-ss}:=\{x\in V\times M,\  \overline{Gx}\cap \Phi^{-1}(0)\neq \emptyset\}$.

\begin{lemma}\label{lem:bar-S-beta}
For any $\beta\in \Bcal$, we have 
$$
\overline{\Scal_{\langle\beta\rangle}}\setminus \Scal_{\langle\beta\rangle}\subset \bigcup_{\|\beta'\|>\|\beta\|}\Scal_{\langle\beta'\rangle}.
$$
\end{lemma}

\begin{proof}
  Let $\beta\neq \beta'\in \Bcal$ such that
  $\overline{\Scal_{\langle\beta\rangle}}\cap \Scal_{\langle\beta'\rangle}\neq\emptyset$. Let us use
  that the set $\overline{\Scal_{\langle\beta\rangle}}$ is closed and
  $G$-invariant. As $\Scal_{\langle\beta'\rangle}=G\Scal_{\beta'}$, we have
  $\overline{\Scal_{\langle\beta\rangle}}\cap\Scal_{\beta'}\neq\emptyset$. Since
  the map $x\mapsto \lim_{t\to\infty}e^{-it\beta'}x$ sends
  $\Scal_{\beta'}$ to $Z_{\beta'}^{^{ss}}$, we see that
  $\overline{\Scal_{\langle\beta\rangle}}\cap Z_{\beta'}^{^{ss}}\neq\emptyset$. Since,
  for any $x\in Z_{\beta'}^{^{ss}}$ the limit $x_\infty$ of the
  negative gradient flow line of $f$ belongs to $\Phi^{-1}(\beta')$,
  we have $\overline{\Scal_{\langle\beta\rangle}}\cap
  \Phi^{-1}(\beta')\neq\emptyset$. So, there exists a sequence
  $x_n\in \Scal_{\langle\beta\rangle}$ converging to $x\in \Phi^{-1}(\beta')$. We
  write $x_n=k_n y_n$ with $k_n\in K$ and $y_n\in \Scal_{\beta}$,
  and we can assume that $k_n$ converge to $k_\infty$, so that
  $y:=\lim_{n\to\infty} y_n\in \Phi^{-1}(K\beta')$. For any $n$, and
  any $t\geq 0$, we have
$$
\|\Phi(y_n)\|\geq \frac{\langle\Phi(y_n),\beta\rangle}{\|\beta\|}\geq
\frac{\langle\Phi(e^{-it\beta}y_n),\beta\rangle}{\|\beta\|}\geq
\frac{\varpi(y_n,\beta)}{\|\beta\|}=\|\beta\|.
$$
By taking the limit $n\to\infty$, we obtain
$$
\|\beta'\|=\|\Phi(y)\|\geq
\frac{\langle\Phi(y),\beta\rangle}{\|\beta\|}\geq \|\beta\|.
$$
We obtain that $\|\beta'\|\geq \|\beta\|$, and that the equality
$\|\beta'\|=\|\beta\|$ would give that
$\|\beta\|=\|\Phi(y)\|=\frac{\langle\Phi(y),\beta\rangle}{\|\beta\|}$,
i.e. $\Phi(y)=\beta$. This last relationship is impossible because
$\beta\notin K\beta'$. We have finally shown that
$\|\beta'\|>\|\beta\|$ if
$\overline{\Scal_{\langle\beta\rangle}}\cap\Scal_{\langle\beta'\rangle}\neq\emptyset$.
\end{proof}

\medskip

The main properties of the HKKN stratification are set out  in the following proposition.

\begin{proposition}\label{prop:HKKN-E-M-strata}
Suppose that $M$ is a connected manifold. Let $\rho=\min\{\|\beta\|, \beta\in\Bcal\}$.
\begin{enumerate}
\item If $\|\beta\|>\rho$, then none connected component of $\Scal_{\langle\beta\rangle}$ is open.

\item There exists a unique $\beta_{o}\in \Bcal$ such that $\|\beta_{o}\|=\rho$. 

\item The stratum $\Scal_{\langle\beta_o\rangle}$ is an open subset of $V\times M$, {\em connected} and {\em dense}.
\end{enumerate}
\end{proposition}

\begin{proof}
  The proof uses the following classical fact.
  \begin{lemma}\label{lem:classic-connected}
    Let $N$ be a connected complex manifold and let $X\subset N$ be a
    closed complex submanifold, distinct from $N$. Then, none
    connected component of $X$ is open, and the open subset
    $N\setminus X$ is connected.
  \end{lemma}

  Let us write
  $\Bcal=\{\beta_1,\ldots,\beta_p\}\bigcup\{\beta_{p+1}\cdots,\beta_\ell\}$
  with
$$
\rho=\|\beta_1\|=\cdots=\|\beta_p\|<\|\beta_{p+1}\|\leq \cdots\leq
\|\beta_\ell\|.
$$

We define a sequence of subsets $N_{k}\subset V\times M$,
$k\in\{1,\ldots,\ell\}$ by a decreasing recurrence:
$N_{\ell}= V\times M$ and $N_{k-1}=N_{k}\setminus \Scal_{\langle\beta_k\rangle}$ for
any $k\in\{2,\ldots,\ell\}$.

Thanks to Lemma~\ref{lem:bar-S-beta} and Lemma
\ref{lem:classic-connected}, we can prove the following facts by
top-down recurrence:
\begin{itemize}
\item $\Scal_{\langle\beta_\ell\rangle}$ is closed submanifold of $N_{\ell}$. Hence,
  none connected component of $\Scal_{\langle\beta_\ell\rangle}$ is open and
  $N_{\ell-1}=N_{\ell}\setminus\Scal_{\langle\beta_\ell\rangle}$ is a connected
  open subset of $V\times M$.
\item For $k\in \{p+1,\ldots,\ell\}$, $\Scal_{\langle\beta_k\rangle}$ is closed
  submanifold of $N_{k}$. Hence, none connected component of
  $\Scal_{\langle\beta_k\rangle}$ is open and
  $N_{k-1}=N_{k}\setminus \Scal_{\langle\beta_k\rangle}$ is a connected open subset
  of $V\times M$.
\end{itemize}
So the first point is settled. The open subset $N_p\subset V\times M$
is connected and satisfies
\begin{equation}\label{eq:union-N-p}
  N_p=\Scal_{\langle\beta_1\rangle}\bigcup\cdots\bigcup\Scal_{\langle\beta_p\rangle}
\end{equation}
For a subset $X$ of $N_p$, let us denote by
$\overline{X}^{N_p}:=\overline{X}\cap N_p$ the closure of $X$ in
$N_p$. Lemma~\ref{lem:bar-S-beta} shows that
$\overline{\Scal_{\langle\beta_i\rangle}}^{N_p}=\Scal_{\langle\beta_i\rangle}$,
i.e. $\Scal_{\langle\beta_i\rangle}$ is closed in $N_p$,
$\forall i\in \{1,\ldots,p\}$. As $N_p$ is connected, we must have
$p=1$ in (\ref{eq:union-N-p}), and then $\Scal_{\langle\beta_1\rangle}=N_p$ is a
connected open subset of $V\times M$. Finally, $\Scal_{\langle\beta_1\rangle}$ is
dense since its complement is equal to a union of (a finite number
of) complex submanifolds of dimension strictly less than $V\times M$.
\end{proof}

%%%%%%%%%%%%%%%%%%%%%%%%%%%%%%%%%%%%%%%%%%%%%%%%%%%%
\subsubsection{HKKN stratifications for complex $G$-submanifolds}\label{sec:HKKN-stratification-2}
%%%%%%%%%%%%%%%%%%%%%%%%%%%%%%%%%%%%%%%%%%%%%%%%%%%%

In this section, we suppose that $\Xcal$ is a $G$-invariant complex submanifold of $V\times M$, {\em closed} and {\em connected}.  
Hence, $\Xcal$ is a K\"ahler Hamiltonian $G$-manifold stable under the negative gradient flows of $f$. 
If $x\in \Xcal$, then $x_t\in \Xcal, \forall t\geq 0$, and the limit $x_\infty:=\lim_{t\to\infty} x_t$ belongs to $\Xcal$ since $\Xcal$ is closed.

 If we restrict the stratification (\ref{eq:stratification}) to $\Xcal$, we obtain 
\begin{equation}\label{eq:stratification-submanifold}
\Xcal:=\bigcup_{\beta\in\Bcal_\Xcal} \Xcal_{\langle\beta\rangle}\quad {\rm where}\quad \Xcal_{\langle\beta\rangle}:=
\Xcal\cap \Scal_{\langle\beta\rangle}=\{x\in \Xcal, x_\infty\in \Phi^{-1}(K\beta)\}.
\end{equation}
where $\Bcal_\Xcal:=\{\beta\in\Bcal,\ \Xcal^\beta\cap\Phi^{-1}(\beta)\neq\emptyset\}$.

Let us explain why $\Xcal_{\langle\beta\rangle}$ is a complex submanifold of $\Xcal$, for any $\beta\in\Bcal_\Xcal$. 
The intersection $Z_\beta\cap \Xcal$ is the union of the connected components of $\Xcal^\beta$ that intersect $\Phi^{-1}(\beta)$: 
it is a K\"ahler Hamiltonian $G_\beta$-manifold with moment map $\Phi_\beta=\Phi-\beta$. We define the Bia\l ynicki-Birula's 
complex submanifold of $\Xcal$:
$$
Y_\beta\cap \Xcal:=\Big\{x\in \Xcal, \ \lim_{t\to+\infty} e^{-it\beta} x\in Z_\beta\cap\Xcal\Big\}.
$$
The projection $\pi_\beta: Y_\beta\cap \Xcal\to Z_\beta\cap \Xcal$, that sends $x\in Y_\beta\cap \Xcal$ to $\lim_{t\to+\infty} e^{-it\beta} x$, is an holomorphic map.

Let $Z_\beta^{^{ss}}\cap \Xcal$ is the open subset of $Z_\beta\cap \Xcal$ formed by the $\Phi_\beta$-semistable elements. 
Let $\Scal_{\beta}\cap \Xcal$ be the open subset of 
$Y_\beta\cap \Xcal$ equal to the pullback $\pi_\beta^{-1}(Z_\beta^{^{ss}}\cap \Xcal)$.

Finally, $G\times_{P({\beta})} (\Scal_{\beta}\cap \Xcal)$ has a canonical structure of complex manifold, and the map 
$G\times_{P({\beta})} (\Scal_{\beta}\cap \Xcal)\to \Xcal, [g,x]\mapsto gx$ is an holomorphic embedding, 
with image equal to $\Xcal_{\langle\beta\rangle}\subset \Xcal$.

Since $\overline{\Xcal_{\langle\beta\rangle}}=\overline{\Scal_{\langle\beta\rangle}}\cap \Xcal$, 
the following result is a corollary of Lemma~\ref{lem:bar-S-beta}. 

\begin{lemma}\label{lem:bar-X-beta}
For any $\beta\in \Bcal_\Xcal$, we have 
$$
\overline{\Xcal_{\langle\beta\rangle}}\setminus \Xcal_{\langle\beta\rangle}\subset \bigcup_{\|\beta'\|>\|\beta\|}\Xcal_{\langle\beta'\rangle}.
$$
\end{lemma}

The following result is an analogue of Proposition~\ref{prop:HKKN-E-M-strata}, with the same proof.

\begin{proposition}\label{prop:HKKN-Xcal-G-invariant}
Let $\Xcal$ be a $G$-invariant complex submanifold of $V\times M$ that is {\em closed} and {\em connected}. Let $\rho_\Xcal=\min\{\|\beta\|, \beta\in\Bcal_\Xcal\}$.
\begin{enumerate}
\item Let $\beta\in \Bcal_\Xcal$ such that $\|\beta\|>\rho_\Xcal$. Then none connected component of $\Xcal_{\langle\beta\rangle}$ is open.
\item There exists a unique $\beta_{o}\in \Bcal_\Xcal$ such that $\|\beta_{o}\|=\rho_\Xcal$.
\item The stratum $\Xcal_{\langle\beta_o\rangle}$ is a {\em connected} and {\em dense} open subset of $\Xcal$.
\end{enumerate}
\end{proposition}

%%%%%%%%%%%%%%%%%%%%%%%%%%%%%%%%%%%%%%%%%%%%%%%%%%%%
\subsubsection{HKKN stratifications for analytic subsets}\label{sec:HKKN-stratification-3}
%%%%%%%%%%%%%%%%%%%%%%%%%%%%%%%%%%%%%%%%%%%%%%%%%%%%

Let $N$ be a complex  manifold. A subset $\Zcal\subset N$ is analytic, if $\Zcal$ is closed, and for every $x\in \Zcal$, there is an open neighborhood $\Vcal$ of $x$ in $N$ 
and a finite collection of holomorphic functions $h_1,\ldots, h_s\in \Ocal(\Vcal)$ such that $\Zcal\cap \Vcal= \{z\in \Vcal, h_1(z) =\cdots = h_s(z)=0\}$. 

\begin{definition} Let $\Zcal\subset N$ be an analytic subset.  We say that $x\in \Zcal$  is a regular point of $\Zcal$ if 
$\Zcal\cap \Vcal$ is a complex submanifold of $\Vcal$ for some neighborhood $\Vcal$ of $x$. \end{definition}

\begin{example}
According to Lemma~\ref{lem:bar-X-beta}, we know that for any $r\geq 0$, 
$\bigcup_{\|\beta\|\geq r} \Scal_{\langle\beta\rangle}$
is closed, and therefore it is a $G$-invariant analytic subset of $V\times M$.
\end{example}

The subset of an analytic subset $\Zcal$ formed by its regular points is denoted $\Zcal_{reg}$. Let us recall some classical facts: 
$\Zcal_{reg}$ is a complex submanifold, $\Zcal_{reg}$ 
is a {\em dense} open subset of $\Zcal$, and $\Zcal_{reg}$ is connected when the analytic subset $\Zcal$ is {\em irreducible} (see \cite[Chapter~1]{GR94}).

A main property in the study of K\"ahler Hamiltonian $G$-spaces, which can be found in \cite[Section~3]{GMO11}, 
is that the $G$-action behaves similarly to an algebraic action in the following sense.
\begin{proposition}\label{ex:G-x}
For any $x\in V\times M$, we have the following geometric properties:
\begin{enumerate}
\item The closure $\overline{Gx}$ is an irreducible analytic subset of $V\times M$.
\item The orbit $Gx$ is a dense open subset of $\overline{Gx}$.
\item The orbit $Gx$ is a connected complex submanifold of $V\times M$.
\end{enumerate}
\end{proposition}

We start with the following result, where {\bf one works without $G$-invariance condition}.

\begin{proposition}\label{prop:HKKN-Xcal}
Let $\Xcal$ be a connected complex submanifold of $V\times M$. There exists a unique $\beta\in \Bcal$ such that 
$\Xcal_{\langle\beta\rangle}:=\Xcal\cap \Scal_{\langle\beta\rangle}$ is a dense open subset of $\Xcal$.
\end{proposition}

\begin{proof}  Since the strata $\Scal_{\langle\beta\rangle}$ are disjoints, the ``unicity" property is immediate. We now look after the ``existence" property.

From Proposition~\ref{prop:HKKN-E-M-strata}, we know that $\Bcal=\{\beta_{0}\cdots,\beta_\ell\}$ with $\|\beta_0\|< \|\beta_1\|\leq \cdots\leq\|\beta_\ell\|$. 
We know also that for any $k\in \{0,\ldots, \ell\}$, 
$N_k:=\Scal_{\langle\beta_0\rangle}\bigcup\cdots\bigcup \Scal_{\langle\beta_k\rangle}$ is a dense, connected, open subset of 
$N=V\times M$, and $\Scal_{\langle\beta_k\rangle}$ is a closed submanifold of $N_k$. 

Let $\Xcal_k:=\Xcal\cap N_k$ so that $\Xcal_\ell=\Xcal$. We start a top-down recurrence: we have 
$$
\Xcal= \Xcal_{\ell-1}\cup \Xcal_{\langle\beta_\ell\rangle}
$$ 
where $\Xcal_{\langle\beta_\ell\rangle}:=\Xcal\cap \Scal_{\langle\beta_\ell\rangle}$ is an analytic subset of the complex manifold $\Xcal$. 
Due to the connectedness of $\Xcal$, two cases can occur: either $\Xcal_{\langle\beta_\ell\rangle}$ 
has no interior point in $\Xcal$, or $\Xcal_{\langle\beta_\ell\rangle}=\Xcal$. The last case meaning that $\Xcal\subset  \Scal_{\langle\beta_\ell\rangle}$.

If $\Xcal\cap \Scal_{\langle\beta_\ell\rangle}$ has no interior point, then $\Xcal_{\ell-1}$ is a connected dense open subset of $\Xcal$, 
and we can consider the decomposition
$$
\Xcal_{\ell-1}= \Xcal_{\ell-2}\cup \Xcal_{\langle\beta_{\ell-1}\rangle}
$$ 
where $\Xcal_{\langle\beta_{\ell-1}\rangle}:=\Xcal_{\ell-1}\cap \Scal_{\langle\beta_{\ell-1}\rangle}$ is an analytic subset of the complex 
manifold $\Xcal_{\ell-1}$. Like before, either $\Xcal_{\langle\beta_{\ell-1}\rangle}$ has no interior point in $\Xcal_{\ell-1}$, or 
$\Xcal_{\ell-1}\cap \Scal_{\langle\beta_{\ell-1}\rangle}=\Xcal_{\ell-1}$. Here the last condition means that $\Xcal\cap N_{\ell-2}=\emptyset$ 
and $\Xcal\cap \Scal_{\langle\beta_{\ell-1}\rangle}$ is a dense open subset of $\Xcal$.

Continuing this process, we see that there exists $j\in \{0,\ldots, \ell\}$ for which $\Xcal\cap N_{j-1}=\emptyset$ and 
$\Xcal\cap \Scal_{\langle\beta_{j}\rangle}$ is a dense open subset of $\Xcal$.
\end{proof}

\medskip

Let $X$ be an irreducible analytic subset of $V\times M$.  {\bf  We do not assume that $X$ is invariant under the $G$-action}. 
Let us consider the following subsets: $X^{\Phi-ss}:=\Xcal\cap \Scal_{\langle 0\rangle}$ and 
$(X_{reg})^{\,\Phi-ss}:=X_{reg}\cap \Scal_{\langle 0\rangle}$. As $\Scal_{\langle 0\rangle}$ is open in 
$V\times M$, $X^{\,\Phi-ss}$ (resp. $(X_{reg})^{\Phi-ss}$) is open in $X$ (resp. $X_{reg}$).

\begin{proposition}\label{prop:X-ss-B-invariant} Let $X$ be an irreducible analytic subset of $V\times M$.
The following conditions are equivalent:
\begin{enumerate}
\item $(X_{reg})^{\Phi-ss}\neq \emptyset$.
\item $X^{\,\Phi-ss}\neq \emptyset$.
\item $(X_{reg})^{\Phi-ss}$ is a dense open subset of $X_{reg}$.
\item $X^{\,\Phi-ss}$ is a dense open subset of $X$.
\end{enumerate}
\end{proposition}

\begin{proof}
  The implications {\em 1} $\Rightarrow$ {\em 2} and {\em 3}
  $\Rightarrow$ {\em 4} are immediate. If
  $X^{\,\Phi-ss}$ is a dense open subset of
  $X$, then
  $(X_{reg})^{\Phi-ss}=X^{\,\Phi-ss}\cap X_{reg}\neq
  \emptyset$ because $X_{reg}$ is a dense open subset of
  $X$: we have proved {\em 4} $\Rightarrow$ {\em 1}.

  Suppose now show that $X^{\,\Phi-ss}\neq \emptyset$. Let us apply Proposition~\ref{prop:HKKN-Xcal} 
  to the connected complex  submanifold $X_{reg}$: there exists $\beta_o$ such that $X_{reg}\cap \Scal_{\langle\beta_o\rangle}$ 
  is a dense open subset of $X_{reg}$. 
  If $\beta_o\neq 0$, it would implies, thanks to (\ref{eq:stratification}), that
$$
X=\overline{X_{reg}}\subset \overline{\Scal_{\langle\beta_o\rangle}}\subset
\bigcup_{\beta'\neq 0} \Scal_{\langle\beta'\rangle}
$$
and then $X^{\,\Phi-ss}= \emptyset$, which is a contradiction. So, we have proved that 
$(X_{reg})^{\Phi-ss}=X_{reg}\cap \Scal_{\langle0\rangle}$ is a dense open subset of $X_{reg}$. 
\end{proof}

%%%%%%%%%%%%%%%%%%%%%%%%%%%%%%%%%%%%%%%%%%%%%%%%%%%%
%%%%%%%%%%%%%%%%%%%%%%%%%%%%%%%%%%%%%%%%%%%%%%%%%%%%
\section{Convexity properties of the moment map}\label{sec:convexity}
%%%%%%%%%%%%%%%%%%%%%%%%%%%%%%%%%%%%%%%%%%%%%%%%%%%%
%%%%%%%%%%%%%%%%%%%%%%%%%%%%%%%%%%%%%%%%%%%%%%%%%%%%

The aim of this section is to extend the results of Kirwan et alii on the convexity properties 
of the moment map for Hamiltonian group actions, and on the connectedness of the fibers of the moment map,
 to the case of K\"ahler Hamiltonian $G$-manifold with non-proper moment maps, and also to $B$-stable irreducible analytic subsets.

For any subset $X\subset V\times M$, we define 
$$
\Delta(X):=\Phi(X)\cap\tgot^*_{0,+}.
$$
When $X$ is $K$-invariant, $\Delta(X)$ parametrizes the $K$-orbits in $\Phi(X)$.

For any non-empty subset $C$ of a real vector space, we define its asymptotic cone $As(C)$ as the set of all limits 
$y =\lim_{k\to\infty} t_ky_k$ where $(t_k)$ is a sequence of positive reals converging to $0$ and $y_k\in C$. 
When $C$ is closed and convex, $As(C)$ is equal to the recession cone 
$rec(C):=\{v;\ a+tv\in C, \ \forall a\in C,\forall t\geq 0\}$.

\medskip

Let $B\subset G$ be the Borel subgroup associated with the Weyl chamber $\tgot^*_{0,+}$. The main result of this section is the following

\begin{theorem} \label{th:convexity}
Let $X$ be a $B$-invariant irreducible analytic subset of $V\times M$. Then 
\begin{enumerate}
\item $\Delta(X)$ is closed and convex.
\item When $X=V\times M$, the asymptotic cone of $\Delta(V\times M)$ is equal to the closed convex rational cone $\Delta(V)$.
\end{enumerate}
\end{theorem}

\bigskip

\begin{example}
For any $x\in V\times M$, the closure $\overline{Gx}$ is an irreducible analytic subset (see Proposition~\ref{ex:G-x}). Thus,  
$\Delta(\overline{Gx})$ is closed and convex. 
\end{example}

\bigskip

Theorem~\ref{th:convexity} can be specified in the context of projective over affine varieties.

\begin{theorem} \label{th:convexity-algebrique}
Let $X$ be a $G$-invariant  irreducible subvariety of $V\times \Pbb E$. 
Let $X_0\subset V$ be the projection of $X$ on $V$. Then 
\begin{enumerate}
\item There exists a finite set $\{\lambda_1,\ldots,\lambda_\ell\}$ of dominant weights such that
$$
\Delta(X_0)=\sum_{j=1}^\ell \ \R^{\geq 0}\lambda_j.
$$
\item There is a finite set $\{\xi_1,\ldots,\xi_p\}\subset \tgot^*_{0,+}$ of rationals elements such that 
$$
\Delta(X)={\rm convex\ hull}(\{\xi_1,\ldots,\xi_p\})+\Delta(X_0).
$$
\end{enumerate}
In particular, $\Delta(X)$ is a closed convex polyhedron and $\Delta(X_0)$ is the asymptotic cone of $\Delta(X)$.
\end{theorem}

\bigskip

We have another way to describe the moment polytope of the variety $\overline{Bx}\subset V\times \Pbb E$. Let $\Phi_{T_0}: V\times \Pbb E\to\tgot_0^*$ 
be the moment map for the maximal torus $T_0\subset K$. Let $U:=[B,B]$ be the unipotent radical of the Borel subgroup. 

\bigskip

\begin{theorem} \label{th:convexity-Bx} For any $x\in V\times \Pbb E$, we have
$$
\Delta(\overline{Bx})=\tgot_{+,0}^*\cap \bigcap_{u\in U}\Phi_{T_0} (\overline{Tux}).
$$
\end{theorem}

\bigskip

Our last result is concerned with the connectedness of the fibers of the moment map.

\begin{theorem} \label{th:connectedness}
Let $\Xcal$ be a $G$-invariant, closed and connected complex submanifold of $V\times M$.
Then for every $\lambda\in\kgot^*$, the fiber 
$$
\Phi^{-1}(\lambda)\bigcap \Xcal
$$
is path-connected.
\end{theorem}

The convexity properties of the moment map have been the subject of numerous contributions. The convexity of $\Delta(X)$ 
was proved by Atiyah and Guillemin–Sternberg \cite{Ati82,GS82,GS84}, when $K$ is abelian and $X$ is compact symplectic, 
by Brion \cite{Br87} (see also Mumford \cite{Mumford84} and Guillemin–Sternberg \cite{GS82,GS84}) 
for $X$ a $G$-invariant projective algebraic variety, and in the compact symplectic case by Kirwan when $K$ is non-abelian \cite{Kir84b}. 
Finally, Guillemin and Sjamaar \cite{GSj06} 
generalized Brion's result to irreducible analytic subsets preserved only by a Borel subgroup.

When the symplectic manifold is non-compact, Hilgert-Neeb-Plank \cite{HNP94} proved that convexity holds for proper moment map 
(see also \cite{LMTW}). Another convexity result was obtained by Sjamaar for $X$ affine \cite{Sjamaar98}, and 
finally Heinzner-Huckleberry \cite{HH96} proved that convexity holds for non-compact K\"ahler Hamiltonian 
$G$-manifold (possibly singular), when the $G$-action is {\em regular}.

Theorem~\ref{th:convexity} generalizes Guillemin-Sjamaar's result \cite{GSj06} to the non-compact framework. 
Its proof takes up Heinzner-Huckleberry's main idea in \cite{HH96}, and we will see that the stratification studied above allows 
it to be implemented for irreducible analytic subsets preserved only by a Borel subgroup.

Theorem~\ref{th:convexity-algebrique} is both a generalization of Brion's result \cite{Br87} obtained for projective algebraic $G$-varieties 
and that of Sjamaar \cite{Sjamaar98} obtained for affine $G$-varieties.

Theorem~\ref{th:convexity-Bx} generalises a result obtained by Guillemin and Sjamaar \cite{GSj05} in the projective setting.

%%%%%%%%%%%%%%%%%%%%%%%%%%%%%%%%%%%%%%%%%%%%%%%%%%%%
\subsection{Proof of Theorem~\ref{th:convexity}}\label{sec:preuve-th-convex}
%%%%%%%%%%%%%%%%%%%%%%%%%%%%%%%%%%%%%%%%%%%%%%%%%%%%

%%%%%%%%%%%%%%%%%%%%%%%%%%%%%%%%%%%%%%%%%%%%%%%%%%%%
\subsubsection{$\Delta(X)$ is closed}\label{sec:convexity-closed}
%%%%%%%%%%%%%%%%%%%%%%%%%%%%%%%%%%%%%%%%%%%%%%%%%%%%

In this section, we prove the first part of point {\em 1.} of Theorem~\ref{th:convexity}.

\begin{proposition} Let $X\subset V\times M$ be a $B$-invariant closed subset. 
\begin{enumerate}
\item The following conditions are equivalent: 
\begin{itemize}
\item[(a)] $0\in \Delta(X)$.
\item[(b)] $X^{\Phi - ss}\neq\emptyset$.
\item[(c)] $0\in \overline{\Delta(X)}$.
\end{itemize}
\item The set $\Delta(X)$ is closed.
\end{enumerate}
\end{proposition}

\begin{proof}
Let us prove point $(i)$. The implications $(a)\Longrightarrow (b)$ and $(a)\Longrightarrow (c)$ are immediate. 
Let $x\in X^{\Phi - ss}$: it means that there exists a sequence $x_n\in Gx$ such that $\lim_{n\to\infty} x_n=x_\infty$ satisfies $\Phi(x_\infty)=0$.
Since $G=KB$, we can write $x_n= k_n y_n$ with $y_n\in Bx\subset X$ and $k_n\in K$. Since $K$ is compact, there exists a converging 
subsequence $(k_{\varphi(n)})$. Thus we get $x_\infty=k_\infty y_\infty$, with $k_\infty:= \lim_{n\to\infty}k_{\varphi(n)}$ and 
$y_\infty:=\lim_{n\to\infty}=y_{\varphi(n)}$. Finally $y_\infty\in \overline{Bx}\subset X$ and $\Phi(y_\infty)=0$. We have checked $(b)\Longrightarrow (a)$. 
Suppose that $0\in \overline{\Delta(X)}$, so there exists a sequence $x_n\in X$ such that $\lim_{n\to\infty} \Phi(x_n)=0$. Since $X^{\Phi - ss}$ is an open subset of $X$ containing $X\cap \Phi^{-1}(0)$, there exists $N\in\N$ such that $\forall n\geq N, x_n\in X^{\Phi - ss}$. We have verify that $(c)\Longrightarrow (b)$, and therefore point $(i)$ is demonstrated.

Let $\mu\in \overline{\Delta(X)}$. We consider the K\"ahler
Hamiltonian $G$-manifold $N:= V\times M \times (K\mu)^\circ$.  Here,
$(K\mu)^\circ$ denotes the adjoint orbit $K\beta$ equipped with the
opposite symplectic structure: the complex structure is defined
through the identifications
$(K\mu)^\circ \simeq K/K_{\mu}\simeq G/P({-\mu})$. Notice that $\{\mu\}\in (K\mu)^\circ$ is fixed by the $B$-action since 
$B\subset P({-\mu})$. We consider the closed $B$-invariant subset $Y=X\times \{\mu\}\subset N$, so that
the condition $\mu\in \overline{\Delta(X)}$ (resp.
$\mu\in \Delta(X)$) is equivalent to
$0\in \overline{\Delta(Y)}$ (resp. $0\in \Delta(Y)$).

If we apply the first point to the closed $B$-invariant  subset
$Y\subset N$, we see that $0\in \overline{\Delta(Y)}$
$\Longleftrightarrow$ $0\in \Delta(Y)$.  We have proved finally that
$\mu\in \Delta(X)$, and thus shown that
$\Delta(X)$ is closed. 
\end{proof}

%%%%%%%%%%%%%%%%%%%%%%%%%%%%%%%%%%%%%%%%%%%%%%%%%%%%
\subsubsection{$\Delta(X)$ is convex}\label{sec:convexity-convex}
%%%%%%%%%%%%%%%%%%%%%%%%%%%%%%%%%%%%%%%%%%%%%%%%%%%%

Recall that we can associate a Kempf-Ness function $\Psi_\mu$ to any $\mu\in\tgot^*_{0,+}$ (see Example~\ref{example:psi-mu}), and a Kempf-Ness function 
$\Psi_x$ to any $x\in V\times M$. The second part of point {\em 1.} of Theorem~\ref{th:convexity} follows from the following result.

\begin{proposition}\label{prop:proof-th-convexity}Let $X$ be a $B$-invariant irreducible analytic subset of $V\times M$. Then, 
\begin{enumerate}
\item $\mu\in \Delta(X)$ if and only if there exists a {\em dense} open subset $\Vcal_\mu\subset X$ such that the Kempf-Ness function 
$\Psi_x-\Psi_{\mu}:G\to \R$ is bounded from below for any $x\in \Vcal_\mu$.
\item If $\mu_0,\mu_1\in \Delta(X)$, then $[\mu_0,\mu_1]\in \Delta(X)$.
\end{enumerate}
\end{proposition}

\begin{proof}
  Let $\mu\in\tgot^*_{0,+}$. We consider the K\"ahler Hamiltonian
  $G$-manifold $N_\mu:= V\times M \times (K\mu)^\circ$ and the $B$-invariant irreducible subset 
  $Y_\mu= X\times \{\mu\}$ of $N_\mu$. We know
  that the condition $\mu\in \Delta(X)$ is equivalent
  to $0\in \Delta(Y_\mu)$, and that
  $0\in \Delta(Y_\mu)$ holds if and only if $(Y_\mu)^{\,\Phi-ss}$
is a dense open subset of $Y_\mu$ (See Proposition~\ref{prop:X-ss-B-invariant}). We consider now the subset
$$
\Vcal_\mu:=\left\{x\in X, (x,\mu)\in
(Y_\mu)^{\,\Phi-ss}\right\},
$$ 
so that $(Y_\mu)^{\,\Phi-ss}$ is a dense open subset of $Y_\mu$ if and only if $\Vcal_\mu$ is a
dense open subset of $X$. Notice that, for any $x\in X$,
the following facts are equivalent
\begin{itemize}
\item $x$ belongs to $\Vcal_\mu$,
\item $(x,\mu)$ is $\Phi$-semistable,
\item the Kempf-Ness function $\Psi_{(x,\mu)}=\Psi_{x}-\Psi_{\mu}$ is
  bounded from below.
\end{itemize}
See Proposition~\ref{prop:phi-minimal-orbit}. We have finally proved
the first point : $\mu\in \Delta(X)$ if and only if
$$
\Vcal_\mu:=\left\{x\in X,\ \Psi_{x}-\Psi_{\mu} \
  \hbox{is\ bounded\ from\ below}\right\}
$$
is a dense open subset of $X$.

Now, let $\mu_0,\mu_1\in \Delta(X)$. For $t\in [0,1]$,
we define $\mu_t=t\mu_1+(1-t)\mu_0\in\tgot_{0,+}^*$. Since
$$
\Psi_{x}-\Psi_{\mu_t}= t\Big(\Psi_{x}-\Psi_{\mu_1}\Big)+
(1-t)\Big(\Psi_{x}-\Psi_{\mu_0}\Big),
$$
we have the following relation\footnote{A similar identity is obtained
  in the ``Momentum Inclusion Lemma'' in \cite{HH96}.}:
$$
\Vcal_{\mu_0}\bigcap \Vcal_{\mu_1}\subset \Vcal_{\mu_t},\qquad \forall t\in [0,1].
$$ Since
$\Vcal_{\mu_0}$ and $\Vcal_{\mu_1}$ are dense open subsets of
$X$, we see that $\Vcal_{\mu_t}$ is a dense open subset
of $X$. We have proved that $\mu_t\in \Delta(X)$, $\forall t\in [0,1]$.
\end{proof}

%%%%%%%%%%%%%%%%%%%%%%%%%%%%%%%%%%%%%%%%%%%%%%%%%%%%
\subsubsection{The asymptotic cone of $\Delta(V\times M)$}\label{sec:convexity-asymptotic}
%%%%%%%%%%%%%%%%%%%%%%%%%%%%%%%%%%%%%%%%%%%%%%%%%%%%

We first check that $As(\Delta(V\times M))\subset \Delta(V)$. 

Let $\xi\in As(\Delta(V\times M))$. There exists a sequence $\xi_n=\Phi_V(v_n)+\Phi_M(m_n)\in \Delta(V\times M)$, and  a sequence $(t_n)$ of positive reals 
converging to $0$, such that $\xi=\lim_{n\to\infty}t_n\xi_n\in\tgot_{0,+}^*$. Since the sequence $\Phi_M(m_n)$ is bounded, we have $\xi=\lim_{n\to\infty}t_n z_n$, 
with $z_n:=\Phi_V(v_n)\in \kgot^*$. Let us choose a sequence $k_n\in K$ such that $k_nz_n\in \Delta(V)$. Since $K$ is compact, we can assume that 
$k_n$ converges to $k_\infty\in K$. We see that the sequence $t_n(k_nz_n)$, which belongs to $\Delta(V)$, converges to $k_\infty\xi$: we have $k_\infty\xi\in\tgot_{0,+}^*$ and it implies that 
$\xi=k_\infty\xi\in \Delta(V)$.

\medskip

We notice that the opposite inclusion, $\Delta(V)\subset As(\Delta(V\times M))$, holds if $\Delta(V)$ $+$ $\Delta(V\times M) \subset \Delta(V\times M)$. 

\medskip

We now prove that $\Delta(V)+\Delta(V\times M) \subset \Delta(V\times M)$. Let $(\mu_0,\mu_1)\in \Delta(V)\times \Delta(V\times M)$.  We define the subsets
\begin{align*}
\Ucal_{\mu_0}&:=\left\{v\in V,\ \Psi_{v}-\Psi_{\mu_0} \ \hbox{is\ bounded\ from\ below}\right\},\\
\Vcal_{\mu_1}&:= \left\{(v,m)\in V\times M,\ \Psi_{v}+\Psi_m-\Psi_{\mu_1} \ \hbox{is\ bounded\ from\ below}\right\}.
\end{align*}
Following the arguments of Proposition~\ref{prop:proof-th-convexity}, we know that $\Ucal_{\mu_0}$ and $\Vcal_{\mu_1}$ are respectively dense open 
subsets of $V$ and $V\times M$. Let us denote by $p: V\times M\to V$ the projection. Since $p(\Vcal_{\mu_1})$ is open in $V$ there exists 
$(v_1,m_1)\in \Vcal_{\mu_1}$ such that $v_1=p(v_1,m_1)$ belongs to $\Ucal_{\mu_0}$. Thus, the Kempf-Ness functions $\Psi_{v_1}-\Psi_{\mu_0}$ and 
$\Psi_{v_1}+\Psi_{m_1}-\Psi_{\mu_1}$ are bounded from below. Hence,
$$
\Big(\Psi_{v_1}-\Psi_{\mu_0}\Big)+\Big(\Psi_{v_1}+\Psi_{m_1}-\Psi_{\mu_1}\Big)=\Psi_{\sqrt{2}v_1}+\Psi_{m_1}-\Psi_{\mu_0+\mu_1}
$$
is bounded from below. This proves that $\mu_0+\mu_1\in \Delta(V\times M)$.

%%%%%%%%%%%%%%%%%%%%%%%%%%%%%%%%%%%%%%%%%%%%%%%%%%%%
\subsection{Proof of Theorem~\ref{th:convexity-algebrique}}\label{sec:preuve-th-convex-algebrique}
%%%%%%%%%%%%%%%%%%%%%%%%%%%%%%%%%%%%%%%%%%%%%%%%%%%%

Let $X$ be a $B$-invariant irreducible subvariety of $V\times \Pbb E$.  Let $\Ccal(X)$ be the $\Q$-convex subset of $(\tgot_0^*)_\Q$ 
defined by (\ref{eq:defmpB}).  Theorem~\ref{th:convexity-algebrique} is a consequence of the following 
\begin{proposition}
\begin{enumerate}
\item $\Delta(X)$ is closed and convex.
\item $\Delta(X)\cap\tgot^*_{0,\Q}=\Ccal(X)$.
\item $\overline{\Ccal(X)}=\Delta(X)$.
\end{enumerate}
\end{proposition}

\begin{proof}
Let us check each points. Point (i) is a direct consequence of Theorem~\ref{th:convexity}. Point (ii) is due to the ``Shifting trick'' (see Lemma~\ref{lem:Shifting-2}): 
$$
\tfrac{\mu}{k}\in\Ccal (X)\ \Longleftrightarrow \ 
\left(X_{[k]}\times \{\mu\}\right)^{ss}=\left(X\times \{\tfrac{\mu}{k}\}\right)^{\Phi-ss}\neq \emptyset \ \Longleftrightarrow \ 
\tfrac{\mu}{k}\in \Delta(X).
$$
Here, $\mu\in (K\mu)^{\circ}$ corresponds to the element $[v_{\mu}]\in \Ocal_\mu\subset \Pbb(V_\mu)$ that is used in Lemma~\ref{lem:Shifting-2}.

We will propose two different proofs of the fact that $\Ccal(X)$ is dense in $\Delta(X)$.

The first proof works when we assume that $X$ is $G$-invariant, so that $X_{reg}$ is a $G$-invariant connected complex submanifold of $V\times M$.
If we apply the result of \cite[Section 2.3]{LMTW} to our connected $K$-Hamiltonian manifold $X_{reg}$, we see that 
\begin{itemize} 
\item There exists a unique open wall $\sigma$ of the Weyl chamber $\tgot^*_{0,+}$ with the property that $\Delta(X_{reg})\cap \sigma$ is dense in $\Delta(X_{reg})$. 
\item The preimage $Y=X_{reg}\cap \Phi^{-1}(\sigma)$ is a connected symplectic $T$-invariant submanifold of $X_{red}$, such that $KY$ is dense in  $X_{reg}$.
\item There exists an open, connected and dense subset $Y_{prin}\subset Y$ such that all points in $Y_{prin}$ have the same isotropy Lie algebra $\hgot$ satisfying 
$[\kgot_{\sigma},\kgot_{\sigma}]\subset \hgot\subset \kgot_\sigma$.
\item $\Delta(X_{reg})$ generates an affine subspace of $\sigma$ with direction the annihilator $\hgot^o$ of $\hgot$ in $\kgot^*$.
\end{itemize}

The smallest element $\beta(X)$ of the convex set $\Delta(X)$ is rational since, when $\beta(X)\neq 0$, it is an optimal destabilizing vector for the 
Mumford numerical invariant $\mathbf{M}_{rel}$. Finally, we have proven that the convex set $\Delta(X)$ generates 
the rational affine subspace $\beta(X)+\hgot^o$. Since $\Ccal(X)$ is the set of rational points 
of the convex set $\Delta(X)$, we can conclude that $\Ccal(X)$ is dense in $\Delta(X)$.

Let us give another proof of the denseness of the rational points of $\Delta(X)$ when $X$ is only invariant under the action of the Borel subgroup $B$. 
We replicate the arguments given by Guillemin-Sjamaar in \cite[Section 4]{GSj06}. 
Let $\lambda\in\Delta(X)$. There exists a sequence $(\Delta_n)$ of rational polytopes in $\tgot^*_{0,+}$ subject to the following requirements :
\begin{itemize}
\item $\lambda\in\Delta_n,\forall n\in\N$;
\item the diameter $\delta_n$ of the polytope $\Delta_n$ tends to $0$, when $n\to\infty$;
\item for each $n$ there exists a $G$-invariant projective variety $(Y_n,L_n)$ such that $\Delta(Y_n,L_n)=d_n\Delta_n$ for some $d_n\geq 1$.
\end{itemize}

Let $Y_n^\circ$ be the variety $Y_n$ with opposite complex structure and opposite polarization $L_n^{-1}$. From our construction, 
we have  $0\in\Delta(X_{[d_n]}\times (Y_n^\circ,L_n^{-1}))$, and that means 
that there exist $k\geq 1$ and a invariant section $s\in H^0(V\times \Pbb E \times Y_n, L^{kd_n}\otimes L_n^{-k})^G$ that does not 
vanish when restricted to $X\times Y_n$. This implies that there exists a dominant weight $\mu$, and non-zero sections 
$s_1\in H^0(V\times \Pbb E,L^{kd_n})^{U}_\mu$ and $s_2\in H^0(Y_n, L_n^{k})^{U}_\mu$ such that $s_1$ does not vanish 
on $X$ (here $U$ is the unipotent radical of $B$). We have then 
$$
\xi_n:=\frac{\mu}{kd_n}\in \Ccal(X)\cap\Delta_n, 
$$
and so $\|\xi_n-\lambda\|\leq\delta_n$. By this way, we have proved that $\Ccal(X)$ is dense in $\Delta(X)$. 
\end{proof}

\bigskip

We can now complete the proof of Theorem~\ref{th:convexity-algebrique}. Here, we work with an irreducible subvariety $X\subset V\times \Pbb E$ invariant under $G$, and we denote by $X_0\subset V$ the projection of $X$ onto $V$.

Let us apply Point (iii) to the $G$-affine variety $X_0\subset V$. There exists a finite collection $\{\lambda_1,\ldots,\lambda_r\}$ 
of dominant weights such that $\Ccal(X_0)=\sum_{j=1}^r \Q^{\geq 0}\lambda_j$ (see Section~\ref{sec:polyhedron}), and 
then\footnote{This property was first shown by Sjamaar \cite{Sjamaar98}.} $\Delta(X_0)=\overline{\Ccal (X_0)}=\sum_{j=1}^r \R^{\geq 0}\lambda_j$.

We have proven in Lemma \ref{lem:CX} that there exists is a finite collection $\{\xi_1,\ldots,\xi_p\}$ of rational dominant elements such that 
$$
\Ccal(X)={\rm convex\ hull}_\Q(\{\xi_1,\ldots,\xi_p\})+\Ccal(X_0).
$$
Since $\Delta(X)=\overline{\Ccal (X)}$, we obtain that $\Delta(X)={\rm convex\ hull}(\{\xi_1,\ldots,\xi_p\})+\Delta(X_0)$.

%%%%%%%%%%%%%%%%%%%%%%%%%%%%%%%%%%%%%%%%%%%%%%%%%%%%
\subsection{Proof of Theorem~\ref{th:convexity-Bx}}\label{sec:preuve-th-convex-Bx}
%%%%%%%%%%%%%%%%%%%%%%%%%%%%%%%%%%%%%%%%%%%%%%%%%%%%

In this section, we consider the action of a maximal torus $T\subset B\subset G$ on $V\times\Pbb E$. 
Let $T_0\subset K$ be the maximal torus such that $T$ is the complexification of $T_0$. The moment map $\Phi_{T_0}: V\times \Pbb E\to\tgot_0^*$,  
relative to the $T_0$-action on $V\times \Pbb E$, is the composition of the moment map $\Phi$ with the projection $\kgot^*\to\tgot_0^*$.

We associate a moment polyhedron $\Ccal_T(X)\subset \Xgot^*(T)_\Q$ to any irreducible $T$-stable subvariety 
$X\subset V\times \Pbb E$ (see Definition~\ref{def:P-T-X}). Theorem~\ref{th:convexity-algebrique}, 
applied to the case where the reductive group $G$  is abelian,  insures that $\Phi_{T_0}(X)\subset \tgot^*_0$ 
is a rational closed convex set such that 
$$
\overline{\Ccal_T(X)}=\Phi_{T_0}(X).
$$

Thanks to (\ref{equation:coro-Bx}), we know that 
\begin{equation}\label{eq:preuve-th-convexity-Bx}
\Ccal(\overline{Bx})=\Xgot^*(T)^+_\Q\cap \bigcap_{u\in U}\Ccal_T(\overline{Tux}).
\end{equation}
Finally, the closures of $\Ccal(\overline{Bx})$ and $\Ccal_T(\overline{Tux})$ are respectively equal to 
$\Delta(\overline{Bx})$ and $\Phi_{T_0}(\overline{Tux})$. Thus, (\ref{eq:preuve-th-convexity-Bx}) gives that 
$$
\Delta(\overline{Bx})=\tgot_{+,0}^*\cap \bigcap_{u\in U}\Phi_{T_0}(\overline{Tux}).
$$
The proof of Theorem~\ref{th:convexity-Bx} is completed.

%%%%%%%%%%%%%%%%%%%%%%%%%%%%%%%%%%%%%%%%%%%%%%%%%%%%
\subsection{Proof of Theorem~\ref{th:connectedness}}\label{sec:preuve-th-connexe}
%%%%%%%%%%%%%%%%%%%%%%%%%%%%%%%%%%%%%%%%%%%%%%%%%%%%

Let $N:=V\times M$ and $f=\frac{1}{2}\|\Phi\|^2 : N\to \R$.

The following result is well-known in the compact framework \cite{Lerman05}.
\begin{proposition} \label{prop:retraction}
Suppose that $N^{\Phi-ss}$ is non-empty. The map $N^{\Phi-ss}\to \Phi^{-1}(0), x\mapsto x_\infty$ is a continuous retraction.
\end{proposition}

\begin{proof} Here, we adapt the arguments of \cite{Lerman05} to our non-compact framework. For any $r\geq 0$, we consider 
$$
N_r:= \overline{G\cdot \{v\in V, \|v\|\leq r\}}\times  M,
$$
which is a closed and $G$-invariant subset of $N$. Sjamaar has shown in \cite[Lemma 4.10]{Sjamaar98} that the moment map 
$\Phi_V$ is proper when restricted to $\overline{G\cdot \{v\in V, \|v\|\leq r\}}$. 
It follows that $f$ is proper when restricted to $N_r$, and so 
$$
\Kcal_{r,\eta}:=N_r\bigcap\{ f\leq \eta\} 
$$
is compact for any $r,\eta\geq 0$. Notice that the compact sets $\Kcal_{r,\eta}$ are stable under the gradient flows $x\mapsto x_t$ 
for any $t\geq 0$. Denote by $d$ the distance on $N$ defined by the Riemannian metric.

\begin{lemma}\label{lem:distance-epsilon}
Let $r>0$ such that $N_r\cap N^{\Phi-ss}$ is non-empty. For any $\epsilon>0$, there exists $\eta>0$ such that 
\begin{enumerate}
\item $\Kcal_{r,\eta}\subset N^{\Phi-ss}$,
\item $d(x,x_\infty)\leq \epsilon$, $\forall x\in \Kcal_{r,\eta}$.
\end{enumerate}
\end{lemma}

\begin{proof} The condition $N_r\cap N^{\Phi-ss}\neq\emptyset$ implies that $N_r\cap\{ f=0\}\neq\emptyset$. 
By the Lojasiewicz gradient inequality there exist constants $C_r>0$, and $\alpha_r\in ]0,1[$, such that $C_r\|\nabla f(x)\|\geq  f(x)^{\alpha_r}$ 
for any $x$ in a neighborhood $U_r$ of the compact subset $N_r\cap\{ f=0\}$. Since $f$ is proper on $N_r$, there exists $\eta'>0$ 
such that the compact subset 
$\Kcal_{r,\eta'}$ is contained in $U_r$.  

If we take $x\in \Kcal_{r,\eta'}$, the gradient flow $x_t$ stays in $\Kcal_{r,\eta'}$ for ant $t\geq 0$: hence we get $C_r\|\nabla f(x_t)\|\geq  f(x_t)^{\alpha_r}$ for any 
$t\geq 0$. If we take the limit $t\to\infty$, we obtain that $f(x_\infty)=0$ for any $x\in \Kcal_{r,\eta'}$, i.e. $\Kcal_{r,\eta'}\subset N^{\Phi-ss}$. 
The former inequalities imply also that $d(x,x_\infty)\leq \frac{C_r}{1-\alpha_r}f(x)^{1-\alpha_r}$ for any $x\in \Kcal_{r,\eta'}$. 
Finally, if we take $\eta\in ]0,\eta']$ 
 such that   $\frac{C_r}{1-\alpha_r}(\eta)^{1-\alpha_r}\leq \epsilon$, we obtain 
 $d(x,x_\infty)\leq \epsilon$ for any $x\in \Kcal_{r,\eta}\subset  \Kcal_{r,\eta'}$. 
 \end{proof}
 
 \medskip
 
 We can now terminate the proof of Proposition~\ref{prop:retraction}.
 
 Let us fix $y\in N^{\Phi-ss}$ and $\epsilon>0$. Let $r>0$ such that $y\in N_{\frac{r}{2}}$. Let $\eta>0$, associated to $(\epsilon,r)$, as in Lemma~\ref{lem:distance-epsilon}.
We fix $t_0>0$ such that $y_{t_0}\in  \Kcal_{\frac{r}{2},\frac{\eta}{2}}\subset \Kcal_{r,\eta}$. 

Since the map $x\in N\mapsto x_{t_0}\in N$ is continuous, 
 there exists $\delta>0$ satisfying:  $\forall x\in N$ such that $d(x,y)\leq \delta$ we have 
$$
\begin{cases}
x\in N_r,&\\
 d(x_{t_0},y_{t_0})\leq \epsilon,& \\
|f(x_{t_0})-f(y_{t_0})|\leq \frac{\eta}{2}.&
\end{cases}
$$
 
 In particular, if $d(x,y)\leq \delta$, then $x_{t_0}\in\Kcal_{r,\eta}$ and so $x\in N^{\Phi-ss}$. We obtain finally
 $$
 d(x_\infty,y_\infty)\leq  d(x_\infty,x_{t_0})+ d(x_{t_0},y_{t_0})+ d(y_{t_0},y_\infty)\leq 3\epsilon
 $$
for any $x\in N$ satisfying $d(x,y)\leq \delta$. We have proved that the map $N^{\Phi-ss}\to \Phi^{-1}(0), x\mapsto x_\infty$ is continuous. 
Proposition~\ref{prop:retraction} is proved. 

\end{proof}

\medskip

The rest of this section is devoted to the proof of the Theorem~\ref{th:connectedness}.

Let $\Xcal\subset N$ be a complex submanifold, which we assume to be closed, connected and invariant under the $G$-action. 

Let us show first that the subset $\Xcal\cap \Phi^{-1}(0)$ is path-connected (if non-empty). Thanks to Proposition~\ref{prop:retraction}, we know that 
 $\varphi:\Xcal^{\Phi-ss}\to \Xcal\cap \Phi^{-1}(0), x\mapsto x_\infty$ is a continuous retraction. We have proved that $\Xcal^{\Phi-ss}$ is a connected open subset of $\Xcal$ 
 (see Proposition~\ref{prop:HKKN-Xcal-G-invariant}). Consequently, $\Xcal^{\Phi-ss}$ is path-connected, and this is also the case for $\Phi^{-1}(0)$ since it is equal to the image of 
 $\Xcal^{\Phi-ss}$ through the continuous map $\varphi$. 
 
 \bigskip
 
Now, the proof of Theorem~\ref{th:connectedness} is a consequence of the shifting trick. Let $K\lambda$ be the coadjoint orbit associated to $\lambda\in\kgot^*$. 
We consider the K\"ahler Hamiltonian $G$-manifold $\tilde{N}:= V\times M\times (K\lambda)^\circ$, with moment map $\tilde{\Phi}$. Through the isomorphism 
$\tilde{N}\simeq K\times_{K_\lambda}(V\times M)$, we have the identity $\tilde{\Phi}^{-1}(0)\simeq K\times_{K_\lambda} \Phi^{-1}(\lambda)$. 
Let us consider the complex submanifold $\tilde{\Xcal}:= \Xcal\times (K\lambda)^\circ$ of $\tilde{N}$. From the case studied above, we know that
$$
\tilde{\Phi}^{-1}(0)\bigcap \tilde{\Xcal} \simeq K\times_{K_\lambda} \left(\Phi^{-1}(\lambda)\bigcap \Xcal\right).
$$
is path-connected. Since the stabilizer subgroup $K_\lambda$ is path-connected, it implies that $\Phi^{-1}(\lambda)\bigcap \Xcal$ is also path-connected.

\end{document}